\documentclass[11pt]{article}
\usepackage[numbers,square]{natbib}
\usepackage{amsmath}
\usepackage{amsthm}
\usepackage{amsfonts}
\usepackage{amssymb}
\usepackage{siunitx}
\usepackage{commath}
\usepackage{graphicx}
\usepackage{xspace}
\usepackage{color}
\usepackage[lofdepth,lotdepth]{subfig}
\usepackage{stmaryrd}
\usepackage{url}
\usepackage{amsthm}
\usepackage{footnote}
\makesavenoteenv{tabular}
\makesavenoteenv{table}
\usepackage[hidelinks]{hyperref}
\hypersetup{colorlinks = true, urlcolor = blue,
  linkcolor = blue, citecolor = blue}
\usepackage{cleveref}

\usepackage[margin=0.7in]{geometry}
\makeatletter
\newcommand{\tnorm}{\@ifstar\@tnorms\@tnorm}
\newcommand{\@tnorms}[1]{%
  \left|\mkern-1.5mu\left|\mkern-1.5mu\left|
   #1
  \right|\mkern-1.5mu\right|\mkern-1.5mu\right|
}
\newcommand{\@tnorm}[2][]{%
  \mathopen{#1|\mkern-1.5mu#1|\mkern-1.5mu#1|}
  #2
  \mathclose{#1|\mkern-1.5mu#1|\mkern-1.5mu#1|}
}
\makeatother

\newcommand{\jump}[1]{\llbracket #1 \rrbracket}
\newcommand{\av}[1]{\{\!\!\{#1\}\!\!\}}
\newtheorem{theorem}{Theorem}
\newtheorem{lemma}{Lemma}
\newtheorem{corollary}{Corollary}
\newtheorem{proposition}{Proposition}
\newtheorem{remark}{Remark}
\numberwithin{equation}{section}
\title{A hybridizable discontinuous Galerkin method for the coupled
  Navier--Stokes/Biot problem}
\author{
  A. Cesmelioglu \thanks{Department of Mathematics and Statistics,
    Oakland University, MI, USA, (cesmelio@oakland.edu),
    \url{https://orcid.org/0000-0001-8057-6349} }
  \and
  J. J. Lee \thanks{Department of Mathematics, Baylor University, TX,
    USA, (jeonghun\_lee@baylor.edu),
    \url{https://orcid.org/0000-0001-5201-8526} }
  \and
  S. Rhebergen\thanks{Department of Applied Mathematics, University of
    Waterloo, ON, Canada (srheberg@uwaterloo.ca),
    \url{http://orcid.org/0000-0001-6036-0356} } }
\begin{document}
\maketitle
\begin{abstract}
  In this paper we present a hybridizable discontinuous Galerkin
  method for the time-dependent Navier--Stokes equations coupled to
  the quasi-static poroelasticity equations via interface
  conditions. We determine a bound on the data that guarantees
  stability and well-posedness of the fully discrete problem and prove
  a priori error estimates. A numerical example confirms our analysis.  
\end{abstract}
\section{Introduction}
\label{sec:introduction}

In this paper we consider a system of partial differential equations
such that the governing equations of two different physical models on
two disjoint subdomains are coupled across an interface. The two
models are the time-dependent Navier--Stokes equations of
incompressible fluids and the quasi-static poroelasticity (or Biot)
equations \cite{Biot:1957,Biot:1955,Biot:1941}. The interface
conditions coupling the two governing equations are derived by
fundamental physical laws and experimental data. This fluid and
poroelastic structure interaction problem, which we refer to here as
the coupled Navier--Stokes/Biot problem, has applications in
engineering fields such as hydrogeology, petroleum engineering, and
biomechanics.

To the best of our knowledge, the coupled Stokes/Biot model with
general interface conditions was first proposed in
\cite{Showalter:2005}. Soon after, Badia et al. \cite{Badia:2009}
studied conforming finite element methods for the spatial
discretization of the coupled Navier--Stokes/Biot problem and
monolithic and domain decomposition (partitioned) algorithms to solve
the fully discrete problem. A mathematical proof of existence and
uniqueness of weak solutions to the coupled Navier--Stokes/Biot
problem, under a small data assumption, was given in
\cite{Cesmelioglu:2017a}.

Various finite element methods have been studied for the coupled
Stokes/Biot and Navier--Stokes/Biot problems. A Lagrange multiplier
method for the coupled stationary Stokes and quasi-static
poroelasticity equations was studied in \cite{Ambartsumyan:2018},
which was extended to a nonlinear model with non-Newtonian fluids in
\cite{Ambartsumyan:2019}. A conforming/mixed finite element method was
studied for the coupled stationary Stokes and quasi-static
poroelasticity equations in \cite{Baier:2022,Boon:2022} using the
total pressure formulation \cite{Oyarzua:2016,Lee:2017}. Other
formulations of this system of equations have also been studied. These
include the velocity-pressure (for Stokes) and
stress-displacement-velocity-pressure (for poroelasticity) formulation
\cite{Li:2022} and the stress-velocity-pressure (for Stokes) and
stress-displacement-velocity-pressure (for poroelasticity) formulation
\cite{Caucao:2022}. A conforming finite element method using Nitsche's
technique for the Stokes/Biot problem is studied in
\cite{Ge:2023}. They consider the velocity-pressure formulation of the
Stokes equations and a formulation using displacement, total pressure,
and fluid content as primary variables for the poroelasticity
equations. For the coupled stationary Navier--Stokes and quasi-static
poroelasticity equations, an augmented mixed method using a
pseudo-stress formulation of the Navier--Stokes equations and a
stress-displacement-velocity-pressure formulation of the
poroelasticity model is studied in \cite{Li:2023}. A conforming finite
element method with stabilization for the time-dependent Stokes
equations coupled to the dynamic poroelasticity equations was studied
in \cite{Cesmelioglu:2020a}. Furthermore, many partitioned time
discretization schemes for efficient time discretization of the
(Navier--)Stokes/Biot model have been studied, see, for example,
\cite{Bukac:2015,Bukac:2016,Oyekole:2020,Bergkamp:2020,Guo:2022}.

In our previous work \cite{Cesmelioglu:2023a}, we presented a
locking-free hybridizable discontinuous Galerkin (HDG)
\cite{Cockburn:2009a} method for the coupled stationary Stokes
equations and quasi-static poroelasticity equations. This HDG method
was constructed such that: (i) the discrete velocities and
displacement are divergence-conforming; (ii) the compressibility
equations are satisfied pointwise on the elements; and (iii) mass is
conserved pointwise on the elements for the semi-discrete problem in
the absence of source/sink terms. In this paper we expand on our work
in \cite{Cesmelioglu:2023a} and propose and analyze an HDG method for
the coupled time-dependent Navier--Stokes and quasi-static Biot
equations that inherits the three aforementioned properties of the HDG
method for the coupled Stokes/Biot problem. For this, we couple the
exactly divergence-free HDG method for the time-dependent
Navier--Stokes equations of \cite{Rhebergen:2018a} to the locking-free
HDG method for the Biot equations of \cite{Cesmelioglu:2022a}.

We consider Backward Euler time-stepping for the time discretization,
lagging the convective velocity in the nonlinear term of the
Navier--Stokes equations. To prove stability and well-posedness of our
discretization, the convective velocity across the interface must be
small enough (this observation was also made in
\cite{Discacciati:2017} for the stationary Navier--Stokes/Darcy
problem). We show this by assuming the data is small and extending the
stability and well-posedness analysis of \cite{Chaabane:2017} for a
discontinuous Galerkin discretization of the Navier--Stokes/Darcy
problem to our HDG discretization of the Navier--Stokes/Biot
problem. With well-posedness established we proceed with an a priori
error analysis. Here it is interesting to remark that all error bounds
are independent of the fluid pressure analogous to pressure-robust
estimates found elsewhere in the literature for divergence-conforming
discretizations of incompressible flows (see, for example,
\cite{Cockburn:2007b,Kanschat:2010,Lehrenfeld:2016,Rhebergen:2020,Schroeder:2018,Wang:2007}
and the review paper \cite{John:2017}).

The remainder of this paper is organized as follows. In
\cref{sec:navstobiot}, we present the coupled Navier--Stokes/Biot
problem. Notation, definitions, useful preliminary results, and the
HDG method are introduced in \cref{sec:hdgmethod}. The HDG method is
shown to be stable and well-posed in \cref{sec:wp} and a priori error
estimates are proven in \cref{sec:erroranalysis}. A numerical result
is presented in \cref{sec:numexample} and we end this paper with
concluding remarks in \cref{sec:conclusions}.

\section{The coupled Navier--Stokes/Biot problem}
\label{sec:navstobiot}

Let $\Omega$ be a bounded connected open subset of $\mathbb{R}^d$,
$d=2,3$, with polygonal/polyhedral boundaries, and let $\Omega^f$ and
$\Omega^b$ be two disjoint open connected subsets, both with
polygonal/polyhedral boundaries, such that
$\overline{\Omega} = \overline{\Omega^f} \cup \overline{\Omega^b}$.
Furthermore, let $J=(0,T]$ denote the time interval of interest.

Let $\sigma^j := p^j\mathbb{I} - 2\mu^j\varepsilon(u^j)$ ($j=f,b$) and
$\varepsilon(u) := \del[0]{\nabla u + (\nabla u)^T}/2$. The
Navier--Stokes equations in $\Omega^f \times J$ are given by
\begin{equation}
  \label{eq:navierstokes}
  \partial_t u^f + \nabla \cdot (u^f \otimes u^f) + \nabla \cdot \sigma^f = f^f,
  \qquad
  \nabla \cdot u^f = 0,
\end{equation}
where $u^f$ is the velocity in $\Omega^f$ and $p^f$ is the pressure in
$\Omega^f$. Furthermore, $\mu^f>0$ is the fluid viscosity and $f^f$ is
a given body force term. Biot's equations in $\Omega^b \times J$ are
given by
\begin{equation}
  \label{eq:biot}
  \begin{aligned}
    -\nabla \cdot u^b + \lambda^{-1}(\alpha p^p - p^b) &= 0,
    &
    \nabla \cdot \sigma^b &= f^b,
    &
    \mu^f\kappa^{-1} z + \nabla p^p
    &= 0,    
    \\
    c_0 \partial_t p^p + \alpha \lambda^{-1} (\alpha \partial_t p^p - \partial_t p^b) + \nabla \cdot z
    &= g^b,
    &&&
  \end{aligned}
\end{equation}
where $u^b$ is the displacement,
$p^b:=\alpha p -\lambda \nabla \cdot u$ is the total pressure, $p^p$
is the pore pressure, and $z$ is the Darcy velocity. Furthermore,
$\mu^b$ and $\lambda$ are the Lam\'e constants, $\kappa > 0$ is the
permeability constant, $\alpha \in (0,1)$ is the Biot--Willis
constant, and $c_0\ge 0$ is the specific storage coefficient. The body
force in $\Omega^b$ is denoted by $f^b$ while $g^b$ is a source/sink
term.

Various equations are prescribed on the interface
$\Gamma_I = \overline{\partial\Omega}^f \cap
\overline{\partial\Omega}^b$ that couple the Navier--Stokes and Biot
equations. First, mass conservation across the interface is prescribed
by
\begin{subequations}
  \label{eq:interface}
  \begin{equation}
    \label{eq:I_u}
    u^f\cdot n = \del[1]{\partial_t u^b + z}\cdot n \text{ on } \Gamma_I \times J.
  \end{equation}
  Here we use the convention that $n^j$, $j=f,b$ is the unit outward
  normal to $\Omega^j$ and that on $\Gamma_I$, $n = n^f = -n^b$.
  Next, the balance of stresses is prescribed by
  \begin{equation}
    \label{eq:I_momentum}
    \sigma^fn = \sigma^bn,
    \qquad
    (\sigma^f n)\cdot n = p^p \text{ on } \Gamma_I \times J.
  \end{equation}
  The Beavers--Joseph--Saffmann condition
  \cite{Beavers:1967,Saffman:1971} prescribes slip with friction and
  is given by
  \begin{equation}
    \label{eq:I_bjs}
    -2\mu^f\del[0]{\varepsilon(u^f)n}^t = \gamma \mu^f\kappa^{-1/2}(u^f-\partial_t u^b)^t
    \text{ on } \Gamma_I \times J,    
  \end{equation}
  where $\gamma > 0$ is an experimentally determined dimensionless
  constant and where $(w)^t:= w - (w\cdot n)n$.
\end{subequations}

The boundary of the domain is partitioned as follows. On each
subdomain we define $\Gamma^j = \partial\Omega \cap \partial\Omega^j$,
$j=f,b$. We then partition both $\Gamma^f$ and $\Gamma^b$ into
Dirichlet $\Gamma_D^j$ and Neumann $\Gamma_N^j$ parts such that
$\Gamma^j = \Gamma_D^j \cup \Gamma_N^j$. Note that
$\Gamma^j_D \cap \Gamma^j_N = \emptyset$ and we will assume that
$|\Gamma^j_D|,|\Gamma^j_N|>0$. It will also be useful to define
$\Gamma_{IN}^f := \Gamma_I \cup \Gamma_N^f$. A second partitioning of
$\Gamma^b$ is defined as $\Gamma^b = \Gamma^b_P \cup \Gamma^b_F$ with
$\Gamma^b_P \cap \Gamma^b_F = \emptyset$ and $|\Gamma^b_P|>0$. We now
impose the following boundary conditions:
\begin{equation}
  \label{eq:bcs}
  \begin{aligned}
    u^j &= 0 && \text{on } \Gamma^j_D\times J, \quad j=f,b,
    &
    \sigma^j n &= 0 &&  \text{on } \Gamma^j_N\times J, \quad j=f,b,
    \\    
    p^p &= 0 && \text{on } \Gamma^b_P\times J,
    &
    z \cdot n &= 0 && \text{on } \Gamma^b_F\times J.    
  \end{aligned}
\end{equation}
Initial conditions are given by:
\begin{equation}
  \label{eq:ics}
  u^j(x,0) = u^j_0(x) \text{ in } \Omega^j, \quad j=f,b,
  \qquad
  p^p(x,0) = p_0^p(x) \text{ in } \Omega^b.  
\end{equation}
Note that $p^b(x,0) = \alpha p_0^p(x) - \lambda \nabla \cdot u^b_0(x)$
and that $z(x,0) = -(\kappa/\mu^f)\nabla p_0^p(x)$. Let us write
$u|_{\Omega^j} = u^j$, $p|_{\Omega^j}=p^j$, $f|_{\Omega^j}=f^j$, and
$\mu|_{\Omega^j}=\mu^j$ for $j=f,b$ so that $u$, $p$, $f$, and $\mu$
are defined on the whole domain $\Omega$. 

\section{Notation, the HDG method, and preliminary results}
\label{sec:hdgmethod}

\subsection{Mesh and time partitioning}
\label{ss:mesh}

We discretize the domains $\Omega^j$, $j=f,b$, by shape-regular
triangulations which we denote by $\mathcal{T}^j$. We will assume that
the triangulations consist of simplices, denoted by $K$, that match at
the interface and that the triangulations are free of hanging
nodes. The set of all simplices is denoted by
$\mathcal{T} := \mathcal{T}^f \cup \mathcal{T}^b$. The boundary of an
element $K$ is denoted by $\partial K$ and we define
$\partial \mathcal{T}^j := \cbr[0]{\partial K\, : \, K \in
  \mathcal{T}^j}$ and
$\partial \mathcal{T} := \cbr[0]{\partial K\, : \, K \in
  \mathcal{T}}$. We also consider various sets of facets. The sets of
all facets in $\overline{\Omega}^j$ and $\Omega^j$ are denoted by
$\mathcal{F}^j$ and $\mathcal{F}^j_{int}$, respectively. The sets of
facets on the Dirichlet $\Gamma_D^j$ and Neumann $\Gamma_N^j$
boundaries are denoted by, respectively, $\mathcal{F}_D^j$ and
$\mathcal{F}_N^j$, while on $\Gamma_F^b$ and $\Gamma_P^b$ we denote
the sets of facets by, respectively, $\mathcal{F}_F^b$ and
$\mathcal{F}_P^b$. The set of facets on the interface is denoted by
$\mathcal{F}_I$ and the set of all facets is denoted by
$\mathcal{F}$. The union of facets in $\mathcal{F}^j$ is denoted by
$\Gamma_0^j$. Furthermore, we denote by $\mathcal{F}(K)$ the set of
all facets of $K$. The diameter of an element $K$ is denoted by $h_K$
and we define $h := \max_{K \in \mathcal{T}}h_K$. On the boundary of
an element we denote by $n_K$ the unit outward normal vector, however,
we will drop the subscript $K$ where no confusion can occur.

The time interval $J$ is partitioned as follows:
$0=t_0 < t_1 < \cdots < t_N=T$. For simplicity, we assume a fixed time
step, i.e., $\Delta t = T/N = t_n-t_{n-1}$ for $n=1,\cdots, N$. A
function $f$ evaluated at time level $n$ will be denoted by
$f^n := f(t_n)$. We further introduce the difference
$\delta f^{n+1} := f^{n+1}-f^n$, the first order time derivative
$d_tf^{n+1} := (f^{n+1} - f^n)/\Delta t$ for $n=0,\cdots,N-1$, and the
second order time derivative
$d_{tt}f^{b,n+1} = (f^{b,n+1} - 2f^{b,n} + f^{b,n-1})/(\Delta t)^2$
for $n=1,\cdots,N-1$.

\subsection{Function spaces and norms}
\label{ss:fs-norms}

Various function spaces will be used throughout this paper. First, the
usual Sobolev spaces are denoted by $W^{k,p}(D)$ for $k\ge 0$ and
$1 \le p \le \infty$ on a Lipschitz domain $D \subset
\mathbb{R}^d$. The norm on $W^{k,p}(D)$ is denoted by
$\norm[0]{\cdot}_{p,k,D}$. As usual, $H^k(D) = W^{k,2}(D)$ with norm
$\norm[0]{\cdot}_{k,D} = \norm[0]{\cdot}_{2,k,D}$ and
$L^p(D) = W^{0,p}(D)$ with norm
$\norm[0]{\cdot}_{D}=\norm[0]{\cdot}_{p,0,D}$. The $L^p(S)$ norm on a
surface $S \subset \mathbb{R}^{d-1}$ is defined similarly.

Let $X$ be a Banach space with norm $\norm[0]{\cdot}_X$, then
$W^{k,p}(J;X)$ denotes a Bochner space with norm
$\norm[0]{f}_{W^{k,p}(J;X)}^p := \int_0^T
\sum_{i=0}^k\norm[0]{\partial_t^if(t)}_X^p \dif t$ for
$1 \le p < \infty$. If $k=0$, then $L^p(J;X) = W^{0,p}(J;X)$. The norm
on $L^{\infty}(J;X)$, i.e., the Bochner space for $k=0$ and
$p=\infty$, is defined as
$\norm[0]{f}_{L^{\infty}(J;X)} := \text{ess sup}_{t \in J}
\norm[0]{f(t)}_X$. Furthermore, we denote by $\ell^{p}(J;X)$ the space
equipped with the norm
$\norm[0]{f}_{\ell^{\infty}(J;X)} := \max_{1\le i \le N}
\norm[0]{f^i}_X$ for $p=\infty$ and
$\norm[0]{f}_{\ell^p(J;X)}^p := \Delta t
\sum_{i=1}^N\norm[0]{f^i}_X^p$ for $1 \le p < \infty$.

Let us now define the following function spaces (where $j=f,b$):
\begin{equation*}
  V^j := \cbr[1]{v \in [H^2(\Omega^j)]^{d}\ :\ v|_{\Gamma_D^j} = 0},
  \quad
  Q^j := H^1(\Omega^j),
  \quad
  Q^{b0} := \cbr[1]{ q \in H^2(\Omega^b) \ : \ q|_{\Gamma_P^b} = 0 },  
\end{equation*}
and denote by $\bar{V}^j$ the trace space of $V^j$ restricted to
$\Gamma_0^j$, $\bar{Q}^j$ is the trace space of $Q^j$ restricted to
$\Gamma_0^j$, and $\bar{Q}^{b0}$ is the trace space of $Q^{b0}$
restricted to $\Gamma_0^b$. For a compact notation, we define
$\boldsymbol{V}^j := V^j \times \bar{V}^j$,
$\boldsymbol{Q}^j := Q^j \times \bar{Q}^j$, and
$\boldsymbol{Q}^{b0} := Q^{b0} \times \bar{Q}^{b0}$. Furthermore, we
define
\begin{equation*}
  Z := \cbr[1]{ v \in [H^1(\Omega^b)]^{d}\ :\ v \cdot n|_{\Gamma_F^b} = 0}.
\end{equation*}

To define the HDG method we require the following element and facet
function space pairs on each domain $\Omega^j$, $j=f,b$:
\begin{align*}
  V_h^j
  &:= \cbr[1]{v_h\in \sbr[0]{L^2(\Omega^j)}^d
    : \ v_h \in \sbr[0]{P_k(K)}^d, \ \forall\ K\in\mathcal{T}^j},
  \\
  \bar{V}_h^j
  &:= \cbr[1]{\bar{v}_h \in \sbr[0]{L^2(\Gamma_0^j)}^d:\ \bar{v}_h \in
    \sbr[0]{P_{k}(F)}^d\ \forall\ F \in \mathcal{F}^j,\ \bar{v}_h
    = 0 \text{ on } \Gamma_D^j},
\end{align*}  
and
\begin{equation*}
  Q_h^j
  := \cbr[1]{q_h\in L^2(\Omega^j) : \ q_h \in P_{k-1}(K) ,\
    \forall \ K \in \mathcal{T}^j},
  \quad
  \bar{Q}_h^j
  := \cbr[1]{\bar{q}_h \in L^2(\Gamma_0^j) : \ \bar{q}_h \in
    P_{k}(F) \ \forall\ F \in \mathcal{F}^j},
\end{equation*}
where $P_r(D)$ denotes the set of polynomials of total degree at most
$r \ge 0$ defined on $D$. We will furthermore require:
\begin{align*}
  V_h
  &:= \cbr[1]{v_h \in \sbr[0]{L^2(\Omega)}^d : \ v_h \in \sbr[0]{P_k(K)}^d,\ \forall K \in \mathcal{T}},
  &
  Q_h 
  &:= \cbr[1]{q_h\in L^2(\Omega) : \ q_h \in P_{k-1}(K) ,\
    \forall \ K \in \mathcal{T}},
  \\
  \bar{Q}_h^{b0}
  &:=\cbr[0]{\bar{q}_h \in \bar{Q}_h^b:
    \bar{q}_h=0 \text{ on } \Gamma_P^b}.
  &&
\end{align*}
For $(u_h, p_h) \in V_h \times Q_h$, we will write
$u_h|_{\Omega^j} = u_h^j \in V_h^j$ and
$p_h|_{\Omega^j} = p_h^j \in Q_h^j$ for $j=f,b$. We group element and
facet unknowns together as follows:
\begin{align*}
  \boldsymbol{v}_h = (v_h, \bar{v}_h^f, \bar{v}_h^b)
  &\in \boldsymbol{V}_h := V_h \times \bar{V}_h^f \times \bar{V}_h^b,
  & \boldsymbol{v}_h^j = (v_h^j, \bar{v}_h^j) \in \boldsymbol{V}_h^j := V_h^j \times \bar{V}_h^j,
  \\
  \boldsymbol{q}_h = (q_h, \bar{q}_h^f, \bar{q}_h^b)
  &\in \boldsymbol{Q}_h := Q_h \times \bar{Q}_h^f \times \bar{Q}_h^b,
  & \boldsymbol{q}_h^j = (q_h^j, \bar{q}_h^j) \in \boldsymbol{Q}_h^j := Q_h^j \times \bar{Q}_h^j,
  \\
  \boldsymbol{q}_h^p = (q_h^p, \bar{q}_h^p) &\in \boldsymbol{Q}_h^{b0} := Q_h^b \times \bar{Q}_h^{b0},
\end{align*}
where $j=f,b$, and
$ (\boldsymbol{v}_h, \boldsymbol{q}_h, w_h, \boldsymbol{q}_h^p) \in
\boldsymbol{X}_h := \boldsymbol{V}_h \times \boldsymbol{Q}_h \times
V_h^b \times \boldsymbol{Q}_h^{b0}$. We will also require the
following two subspaces of $\boldsymbol{V}_h^j$ and
$\boldsymbol{V}_h$, respectively:
\begin{equation}
  \label{eq:Vh-tilde}
  \widetilde{\boldsymbol{V}}_h^j
  := \cbr[0]{\boldsymbol{v}_h \in \boldsymbol{V}_h^j \, : \, \bar{v}_h|_{\Gamma_I} = 0},
  \qquad
  \widehat{\boldsymbol{V}}_h
  := \cbr[0]{\boldsymbol{v}_h \in \boldsymbol{V}_h \, : \, \bar{v}_h^f \cdot n = \bar{v}_h^b \cdot n \text{ on } \Gamma_I }.  
\end{equation}
Extended function spaces are defined as (for $j=f,b$):
\begin{align*}
  V^j(h) &:= V_h^j + V^j, &&
  \boldsymbol{V}^j(h) := \boldsymbol{V}_h^j + \boldsymbol{V}^j, &&
  Z(h) := Z_h + Z,
  \\
  \boldsymbol{Q}^j(h) &:= \boldsymbol{Q}_h^j + \boldsymbol{Q}^j, &&   
  \boldsymbol{Q}^{b0}(h) := \boldsymbol{Q}_h^{b0} + \boldsymbol{Q}^{b0}, &&
\end{align*}
and
\begin{equation*}
  V^{f,\text{div}}(h) := \cbr[0]{ v \in V^f(h) \cap H(\text{div};\Omega^f)\, :\,
    \nabla \cdot v = 0 \text{ for } x \in K,\ \forall K \in
    \mathcal{T}^f}.
\end{equation*}
We will work with the following norms:
\begin{align*}
  \tnorm{\boldsymbol{v}^j}_{v,j}^2
  &:=\sum_{K\in \mathcal{T}^j}\del[1]{\norm[0]{\varepsilon(v^j)}_K^2+h_K^{-1}\norm[0]{v^j-\bar{v}^j}_{\partial K}^2}
  &&\forall \boldsymbol{v}^j\in \boldsymbol{V}^j(h), && j=f,b,
  \\
  \tnorm{\boldsymbol{v}^j}_{v',j}^2
  &:=\tnorm{\boldsymbol{v}^j}_{v,j}^2+\sum_{K\in \mathcal{T}^j} h_K^2 |v^j|_{2,K}^2
  &&\forall \boldsymbol{v}^j\in \boldsymbol{V}^j(h), && j=f,b,
  \\
  \tnorm{\boldsymbol{q}}_{q,j}^2
  &:=\norm[0]{q}_{\Omega^j}^2+\sum_{K \in \mathcal{T}^j}h_K\norm[0]{\bar{q}^j}_{\partial K}^2
  &&\forall \boldsymbol{q} \in \boldsymbol{Q}^j(h), && j=f,b,
  \\
  \tnorm{\boldsymbol{v}_h}_{v}^2
  &:= \tnorm{\boldsymbol{v}_h^f}_{v,f}^2+\tnorm{\boldsymbol{v}_h^b}_{v,b}^2 + \norm[0]{(\bar{v}^f_h - \bar{v}^b_h)^t}_{\Gamma_I}^2
  && \forall \boldsymbol{v}_h \in \boldsymbol{V}_h, &&
  \\
  \norm[0]{v_h}_{1,h,\Omega^f}
  &:= \tnorm{(v_h, \av{v_h})}_{v,f} && \forall v_h \in V_h^f,
  \\
  \tnorm{\boldsymbol{q}_h}_q^2
  &:=\tnorm{\boldsymbol{q}_h^f}_{q,f}^2+\tnorm{\boldsymbol{q}_h^b}_{q,b}^2
  && \forall \boldsymbol{q}_h \in \boldsymbol{Q}_h, &&
  \\
  \norm[0]{q_h}_{1,h,\Omega^b}^2
  &:= \sum_{K \in \mathcal{T}^b}\norm[0]{\nabla q_h}_K^2
    + \sum_{F \in \mathcal{F}_{int}^b \cup \mathcal{F}_P^b} h_F^{-1}\norm[0]{\jump{q_h}}_F^2
  && \forall q_h \in Q_h^b, &&
  \\
  \tnorm{\boldsymbol{q}_h}_{1,h,b}^2
  &:= \sum_{K \in \mathcal{T}^b} \del[1]{\norm[0]{\nabla q_h}_K^2 + h_K^{-1}\norm[0]{q_h - \bar{q}_h}_{\partial K}^2}
  && \forall \boldsymbol{q}_h \in \boldsymbol{Q}_h^{b0}.
\end{align*}
It is useful to remark that for all
$\boldsymbol{v} \in \boldsymbol{V}_h^j$
$\tnorm{\boldsymbol{v}}_{v,j} \le \tnorm{\boldsymbol{v}}_{v',j} \le
c_e \tnorm{\boldsymbol{v}}_{v,j}$, with $c_e>0$ a constant independent
of $h$ (see \cite[eq.~(5.5)]{Wells:2011}). By
\cite[Theorem 4.4]{Girault:2009} or \cite[Theorem~4.4]{Buffa:2009} and
\cite{Brenner:2003}, there exist constants $c_p,c_{si,r}>0$,
independent of $h$, such that
\begin{subequations}
  \begin{align}
    \label{eq:dpoincareineq}
    \norm{v_h}_{\Omega^f}
    &\le c_{p} \norm{v_h}_{1,h,\Omega^f}
      \le c_{p} \tnorm{\boldsymbol{v}_h}_{v,f}
      && \forall \boldsymbol{v}_h \in \boldsymbol{V}_h^f,
    \\
    \label{eq:dpoincareineq-b}
    \norm{v_h}_{\Omega^b}
    &\le c_{p} \norm{v_h}_{1,h,\Omega^b}
      \le c_{p} \tnorm{\boldsymbol{v}_h}_{v,b}
      && \forall \boldsymbol{v}_h \in \boldsymbol{V}_h^b,
    \\
    \label{eq:dtrpoincareineq}
    \norm[0]{v_h^f}_{r,0,\Gamma_{IN}^f}
    &\le c_{si,r} \norm{v_h}_{1,h,\Omega^f}
      \le c_{si,r} \tnorm{\boldsymbol{v}_h}_{v,f}
      && \forall \boldsymbol{v}_h \in \boldsymbol{V}_h^f,  \text{ and for } r \ge 2,
    \\
    \label{eq:qhbp}
    \norm[0]{q_h}_{\Omega^b} &\le c_{pp}\norm[0]{q_h}_{1,h,\Omega^b} \le c_{pp}\tnorm{\boldsymbol{q}_h}_{1,h,b}
    && \forall \boldsymbol{q}_h \in \boldsymbol{Q}_h^{b0}.
  \end{align}  
\end{subequations}
Consider two scalar functions $w$ and $z$. We will denote by
$(w, z)_D$ the integral of $wz$ over a domain $D \subset \mathbb{R}^d$
and by $\langle w, z \rangle_D$ the integral of $wz$ over a domain
$D \subset \mathbb{R}^{d-1}$. We furthermore introduce the notation
\begin{align*}
  (w, z)_{\Omega^j} &:= \sum_{K\in\mathcal{T}^j}(w, z)_K,
  & (w, z)_{\Omega}  &:= \sum_{K\in\mathcal{T}}(w, z)_K,
  & \langle w, z \rangle_{\Gamma_I} &:= \sum_{F \in \mathcal{F}_I} \langle w, z \rangle_{F},
  \\
  \langle w, z \rangle_{\partial\mathcal{T}} & := \sum_{K \in \mathcal{T}^j} \langle w, z \rangle_{\partial K},
  & \langle w, z \rangle_{\partial\mathcal{T}} & := \sum_{K \in \mathcal{T}} \langle w, z \rangle_{\partial K},
  & \langle w, z \rangle_{\Gamma_{IN}^f} &:= \sum_{F \in \mathcal{F}_I \cup \mathcal{F}_{N}^f} \langle w, z \rangle_{F}.
\end{align*}
If $w$ and $z$ are vector functions, then
$(w, z)_D := \sum_{i=1}^d (w_i, z_i)_D$ and
$\langle w, z\rangle_D := \sum_{i=1}^d \langle w_i,
z_i\rangle_D$. Similarly, if $w$ and $z$ are matrix functions, then
$(w, z)_D := \sum_{i,j=1}^d (w_{ij}, z_{ij})_D$ and
$\langle w, z\rangle_D := \sum_{i,j=1}^d \langle w_{ij},
z_{ij}\rangle_D$.

\subsection{Forms and their properties}
\label{ss:forms-props}

For $\boldsymbol{u}^j,\boldsymbol{v}^j \in \boldsymbol{V}^j(h)$, we
define
\begin{align*}
  a_h^j(\boldsymbol{u}^j, \boldsymbol{v}^j)
  :=&
      (2\mu^j \varepsilon(u), \varepsilon(v))_{\Omega^j}
      +\sum_{K\in\mathcal{T}^j} \langle 2\beta^j\mu^jh_K^{-1}(u-\bar{u}^j), v-\bar{v}^j \rangle_{\partial K}
  \\
    &\
      -\langle 2\mu^j\varepsilon(u)n^j, v-\bar{v}^j \rangle_{\partial\mathcal{T}^j}
      -\langle 2\mu^j\varepsilon(v)n^j, u-\bar{u}^j \rangle_{\partial\mathcal{T}^j},
  \\
  a_h(\boldsymbol{u}, \boldsymbol{v})
  :=& a_h^f(\boldsymbol{u}^f, \boldsymbol{v}^f) + a_h^b(\boldsymbol{u}^b, \boldsymbol{v}^b),  
\end{align*}
where $\beta^j > 0$ is a penalty parameter. It was shown in
\cite[Lemmas 2 and 3]{Cesmelioglu:2020} and \cite[Lemmas 4.2 and
4.3]{Rhebergen:2017} that there exist constants $\beta_0>0$,
$c_{ae}^j>0$, and $c_{ab}^j>0$, independent of $h$ and $\Delta t$,
such that
\begin{subequations}
  \begin{align}
    \label{eq:ah-coercive-j}
    a_h^j(\boldsymbol{v}_h^j, \boldsymbol{v}_h^j)
    &\geq c_{ae}^j\mu^j\tnorm{\boldsymbol{v}_h^j}_{v,j}^2
    && \forall \boldsymbol{v}_h^j\in \boldsymbol{V}^j_h, \ \beta > \beta_0,
    \\
    \label{eq:ah-continuity-j}
    \envert[0]{a_h^j(\boldsymbol{u}^j, \boldsymbol{v}^j)}
    &\leq c_{ac}^j\mu^j\tnorm{\boldsymbol{u}^j}_{v',j}\tnorm{\boldsymbol{v}^j}_{v',j}
    && \forall \boldsymbol{u}, \boldsymbol{v}\in \boldsymbol{V}^j(h).
  \end{align}  
\end{subequations}
For $\boldsymbol{v}^j \in \boldsymbol{V}^j(h)$ and
$\boldsymbol{q}^j \in \boldsymbol{Q}^j(h)$ we define
\begin{align*}
  b_h^{j}(\boldsymbol{v}^j, \boldsymbol{q}^j)
  &:= -(q, \nabla\cdot v)_{\Omega^j} + \langle \bar{q}^j, (v-\bar{v}^j)\cdot n^j \rangle_{\partial\mathcal{T}^j},
  &
  b_h(\boldsymbol{v}, \boldsymbol{q})
  &:= b_h^f(\boldsymbol{v}^f, \boldsymbol{q}^f) + b_h^b(\boldsymbol{v}^b, \boldsymbol{q}^b).
\end{align*}
The form $b_h^b(\cdot,\cdot)$ is also defined on
$(Z(h)\times\cbr[0]{0}) \times \boldsymbol{Q}^{b0}(h)$. We have the
following inf-sup conditions (see \cite[eqs
(17)-(19)]{Cesmelioglu:2023a} and \cite{Rhebergen:2018b}):
\begin{subequations}
  \begin{align}
    \label{eq:inf-sup-bj}
    \inf_{\boldsymbol{0} \ne \boldsymbol{q}_h \in \boldsymbol{Q}_h^j}
    \sup_{\boldsymbol{0} \ne \boldsymbol{v}_h \in \widetilde{\boldsymbol{V}}_h^j}
    \frac{b_h^j(\boldsymbol{v}_h, \boldsymbol{q}_h)}{\tnorm{\boldsymbol{v}_h}_{v,j}
    \tnorm{\boldsymbol{q}_h}_{q,j} } &\ge c_{bj}, \quad j=f,b,
    \\
    \label{eq:inf-sup-b}
    \inf_{\boldsymbol{0}\neq \boldsymbol{q}_h\in \boldsymbol{Q}_h}
    \sup_{\boldsymbol{0}\neq \boldsymbol{v}_h\in \widehat{\boldsymbol{V}}_h}
    \frac{b_h(\boldsymbol{v}_h, \boldsymbol{q}_h)}{\tnorm{\boldsymbol{v}_h}_v
    \tnorm{\boldsymbol{q}_h}_q}&\geq c_{b},
    \\
    \label{eq:inf-sup-bp}
    \inf_{\boldsymbol{0} \ne \boldsymbol{q}_h^p\in \boldsymbol{Q}_h^{b0}} \sup_{0 \ne w_h \in V_h^b}
    \frac{b_h^b((w_h,0), \boldsymbol{q}_h^p)}{\norm[0]{w_h}_{\Omega^b}
    \tnorm{\boldsymbol{q}_h^p}_{q,b} } &\ge c_{bp},    
  \end{align}
\end{subequations}
where $c_{bf}$, $c_{bb}$, $c_b$, and $c_{bp}$ are positive constants
independent of $h$. On
$(\boldsymbol{Q}^{b0}(h)\times \boldsymbol{Q}^b(h)) \times
\boldsymbol{Q}^{b}(h)$, we define
\begin{equation*}
  c_h((p,r),q) := (\lambda^{-1}(\alpha p - r), q)_{\Omega^b},
\end{equation*}
while on the interface we define
\begin{align*}
  a_h^I((\bar{u}^f,\bar{u}^b), (\bar{v}^f,\bar{v}^b))
  & := \langle \gamma \mu^f \kappa^{-1/2}(\bar{u}^f-\bar{u}^b)^t, (\bar{v}^f-\bar{v}^b)^t \rangle_{\Gamma_I},
  &
  b_h^{I}((\bar{v}^f, \bar{v}^b), \bar{q})
  & := \langle \bar{q}, (\bar{v}^f-\bar{v}^b) \cdot n^f \rangle_{\Gamma_I},
\end{align*}
where $a_h^I$ is defined on
$((\bar{V}^f+\bar{V}_h^f)\times(\bar{V}^b+\bar{V}_h^b)) \times
((\bar{V}^f+\bar{V}_h^f)\times(\bar{V}^b+\bar{V}_h^b))$ and $b_h^I$ is
defined on
$((\bar{V}^f+\bar{V}_h^f)\times(\bar{V}^b+\bar{V}_h^b)) \times
(\bar{Q}^{b0}+\bar{Q}_h^{b0})$.

The forms $a_h(\cdot, \cdot)$, $a_h^I(\cdot, \cdot)$,
$b_h(\cdot, \cdot)$, $b_h^I(\cdot, \cdot)$, and $c_h(\cdot, \cdot)$
discussed above are identical to those considered for the coupled
Stokes and Biot problem in \cite{Cesmelioglu:2023a}. In addition to
these forms, we now also require the following discrete convection
term for $w \in V^{f,div}(h)$ and
$\boldsymbol{u},\boldsymbol{v} \in \boldsymbol{V}^f(h)$:
\begin{multline*}
  t_h(w; \boldsymbol{u}, \boldsymbol{v})
  :=
  -(u\otimes w, \nabla v)_{\Omega^f}
  +\tfrac{1}{2}\langle w\cdot n^f \, (u+\bar{u}), v-\bar{v} \rangle_{\partial\mathcal{T}^f}
  \\
  +\tfrac{1}{2}\langle \envert[0]{w\cdot n^f}(u-\bar{u}), v-\bar{v} \rangle_{\partial\mathcal{T}^f}
  + \langle (w \cdot n^f) \bar{u}, \bar{v} \rangle_{\Gamma_{IN}^f}.
\end{multline*}
We have the following properties for $t_h$.

\begin{proposition}
  For all $w_1,w_2 \in V^{f,\text{div}}(h)$, and
  $\boldsymbol{u}, \boldsymbol{v} \in \boldsymbol{V}^f(h)$, there
  exists a $c_w>0$ such that
  \begin{equation}
    \label{eq:upperboundthdif}
    \envert[0]{t_h(w_1;\boldsymbol{u},\boldsymbol{v}) - t_h(w_2;\boldsymbol{u},\boldsymbol{v})}
    \le c_w \norm[0]{w_1 - w_2}_{1,h,\Omega^f} \tnorm{\boldsymbol{u}}_{v,f}\tnorm{\boldsymbol{v}}_{v,f}.
  \end{equation}
\end{proposition}
\begin{proof}
  Using that
  $\tfrac{1}{2}(w\cdot n^f + |w\cdot n^f|)u + \tfrac{1}{2}(w\cdot n^f
  - |w\cdot n^f|)\bar{u} = w\cdot n^f \bar{u} + S_w(u-\bar{u})$, where
  $S_w = \max(w\cdot n^f, 0)$, we can write $t_h$ after integration by
  parts as
  \begin{equation*}
    t_h(w; \boldsymbol{u}, \boldsymbol{v})
    =(\nabla u, v \otimes w)_{\Omega^f}
    - \langle ((u-\bar{u})\otimes w)n^f, v \rangle_{\partial\mathcal{T}^f}
    + \langle S_w(u - \bar{u}), v - \bar{v} \rangle_{\partial \mathcal{T}^f}.
  \end{equation*}
  The remainder of the proof is given by \cite[Proposition
  3.4]{Cesmelioglu:2017}.
\end{proof}

\begin{proposition}
  For $w \in V^{f,\text{div}}(h)$ and $\boldsymbol{v} \in \boldsymbol{V}(h)$ it holds that
  \begin{equation}
    \label{eq:positivityth}
    t_h(w; \boldsymbol{v}, \boldsymbol{v})
    =\tfrac{1}{2}\langle \envert[0]{w\cdot n^f}, |v-\bar{v}|^2 \rangle_{\partial\mathcal{T}^f}
    + \tfrac{1}{2}\langle w \cdot n^f, |\bar{v}|^2 \rangle_{\Gamma_{IN}^f}.
  \end{equation}
\end{proposition}
\begin{proof}
  Note that
  \begin{equation*}
    t_h(w; \boldsymbol{v}, \boldsymbol{v})
    =
    -(v\otimes w, \nabla v)_{\Omega^f}
    +\tfrac{1}{2}\langle w\cdot n^f \, (v+\bar{v}), v-\bar{v} \rangle_{\partial\mathcal{T}^f}
    +\tfrac{1}{2}\langle |w\cdot n^f|, |v-\bar{v}|^2 \rangle_{\partial\mathcal{T}^f}
    + \langle w \cdot n^f, |\bar{v}|^2 \rangle_{\Gamma_{IN}^f}.    
  \end{equation*}
  Note that $(v+\bar{v}) \cdot (v-\bar{v}) =
  |v|^2-|\bar{v}|^2$. Furthermore,
  $-(v\otimes w, \nabla v)_{\Omega^f} = -\langle \tfrac{1}{2} w\cdot
  n^f, |v|^2 \rangle_{\partial\mathcal{T}^f}$ since
  $-v\otimes w : \nabla v = -\tfrac{1}{2} \nabla
  \cdot(|v|^2w)$. Therefore, also using that
  $\langle \tfrac{1}{2}w\cdot n^f, |\bar{v}|^2
  \rangle_{\partial\mathcal{T}^f} = \langle \tfrac{1}{2}w\cdot n^f,
  |\bar{v}|^2 \rangle_{\Gamma_I} + \langle \tfrac{1}{2}w\cdot n^f,
  |\bar{v}|^2 \rangle_{\Gamma^f_N}$, the result follows.
\end{proof}

\begin{proposition}
  \label{prop:coercivitythahf}
  Let $w \in V^{f,\text{div}}(h)$ and
  $\norm[0]{w\cdot n}_{\Gamma_{IN}^f} \le
  \tfrac{1}{2}\mu^fc_{ae}^f(c_{pq}^2+c_{si,4}^2)^{-1}$. Then, for
  $\beta > \beta_0$,
  \begin{equation}
    \label{eq:coercivityahthresult}
    t_h(w; \boldsymbol{v}_h, \boldsymbol{v}_h)
    + a_h^f(\boldsymbol{v}_h, \boldsymbol{v}_h)
    \ge
    \tfrac{1}{2}c_{ae}^f\mu^f\tnorm{\boldsymbol{v}_h^f}_{v,f}^2
    \quad \forall \boldsymbol{v}_h \in \boldsymbol{V}_h^f.
  \end{equation}
\end{proposition}
\begin{proof}
  By \cref{eq:positivityth}, we find that
  \begin{equation*}
    t_h(w; \boldsymbol{v}_h, \boldsymbol{v}_h)
    \ge -\tfrac{1}{2}\langle |w \cdot n^f|, |\bar{v}_h|^2 \rangle_{\Gamma_{IN}^f}.
  \end{equation*}
  By a scaling identity there exists a constant $c_{pq} > 0$
  independent of $h$ such that
  $\norm[0]{\bar{v}}_{4,0,\partial K} \le c_{pq}
  h^{(1-d)/4}\norm[0]{\bar{v}}_{\partial K}$ for
  $\bar{v} \in \cbr[0]{ q \in L^2(\partial K)\, :\, q|_F \in P_k(F),\
    \forall F \in \mathcal{F}(K)}$. By the identical steps as used in
  the proof of \cite[Lemma 6]{Cesmelioglu:2023c} (see also \cite[Lemma
  2]{Discacciati:2017}) it then follows that
  \begin{equation*}
  t_h(w; \boldsymbol{v}_h, \boldsymbol{v}_h) \ge
  -(c_{pq}^2+c_{si,4}^2)\norm[0]{w\cdot
    n}_{\Gamma_{IN}^f}\tnorm{\boldsymbol{v}_h}_{v,f}^2,       
  \end{equation*}
  where $c_{si,4}$ is the constant from \cref{eq:dtrpoincareineq} with
  $r=4$. We find, using \cref{eq:ah-coercive-j},
  \begin{equation*}
    t_h(w; \boldsymbol{v}_h, \boldsymbol{v}_h)
    + a_h^f(\boldsymbol{v}_h, \boldsymbol{v}_h)
    \ge
    \del[1]{c_{ae}^f\mu^f-(c_{pq}^2+c_{si,4}^2)\norm[0]{w\cdot n}_{\Gamma_{IN}^f}}
    \tnorm{\boldsymbol{v}_h^f}_{v,f}^2.
  \end{equation*}
  The result follows by the assumption on
  $\norm[0]{w\cdot n}_{\Gamma_{IN}^f}$.
\end{proof}

\begin{remark}
  In Proposition \ref{prop:coercivitythahf}, we assume a smallness
  condition on $\norm[0]{w\cdot n}_{\Gamma_{IN}^f}$. If $\Gamma_N^f$
  represents an outflow boundary (on which $w\cdot n > 0$), then the
  smallness condition only needs to hold on
  $\norm[0]{w\cdot n}_{\Gamma_I}$. Numerically, however, it cannot be
  guaranteed that $w \cdot n > 0$ on $\Gamma_N^f$.
\end{remark}

\subsection{The HDG method}
\label{ss:hdgmethod}

The semi-discrete HDG method for the coupled Navier--Stokes/Biot
problem \cref{eq:navierstokes,eq:biot,eq:interface,eq:bcs} is given
by: For $t \in J$, find
$(\boldsymbol{u}_h(t), \boldsymbol{p}_h(t), z_h(t),
\boldsymbol{p}_h^p(t)) \in \boldsymbol{X}_h$ such that for all
$(\boldsymbol{v}_h, \boldsymbol{q}_h, w_h, \boldsymbol{q}_h^p) \in
\boldsymbol{X}_h$,
\begin{subequations}
  \label{eq:hdgsd}
  \begin{align}
    & (\partial_tu_h, v_h)_{\Omega^f}
      + t_h(u_h^f; \boldsymbol{u}_h^f, \boldsymbol{v}_h^f)
      + a_h(\boldsymbol{u}_h, \boldsymbol{v}_h)    
      + b_h(\boldsymbol{v}_h, \boldsymbol{p}_h)
      + a_h^I((\bar{u}_h^f,\partial_t\bar{u}_h^b),(\bar{v}_h^f,\bar{v}_h^b))
      + b_h^I((\bar{v}_h^f,\bar{v}_h^b), \bar{p}_h^p)
    \\ \nonumber
    & \hspace{39em}
      = (f, v_h)_{\Omega},
    \\
    & b_h(\boldsymbol{u}_h, \boldsymbol{q}_h)
      + c_h((p_h^p,p_h^b),q_h^b)
      = 0,    
    \\
    &(c_0\partial_t p_h^p, q_h^p)_{\Omega^b}
      + c_h((\partial_tp_h^p,\partial_tp_h^b), \alpha q_h^p)
      - b_h^{b}((z_h, 0), \boldsymbol{q}_h^p)
      - b_h^I((\bar{u}_h^f, \partial_t\bar{u}_h^b), \bar{q}_h^p)
      =(g^b, q_h^p)_{\Omega^b},
    \\
    &(\mu^f\kappa^{-1} z_h, w_h)_{\Omega^b} + b_h^{b}((w_h, 0), \boldsymbol{p}_h^p)=0.
  \end{align}
\end{subequations}
Using Backward Euler time-stepping, with lagging of the convective
velocity, the fully-discrete HDG method is given by: For
$n=0,1,\cdots,N-1$, find
$(\boldsymbol{u}_h^{n+1}, \boldsymbol{p}_h^{n+1}, z_h^{n+1},
\boldsymbol{p}_h^{p,n+1}) \in \boldsymbol{X}_h$ such that for all
$(\boldsymbol{v}_h, \boldsymbol{q}_h, w_h, \boldsymbol{q}_h^p) \in
\boldsymbol{X}_h$,
\begin{subequations}
  \label{eq:hdgfd}
  \begin{align}
    \label{eq:hdgfd-a}
    & (d_tu_h^{n+1}, v_h)_{\Omega^f}
      + t_h(u_h^{f,n}; \boldsymbol{u}_h^{f,n+1}, \boldsymbol{v}_h^f)
      + a_h(\boldsymbol{u}_h^{n+1}, \boldsymbol{v}_h)
      + b_h(\boldsymbol{v}_h, \boldsymbol{p}_h^{n+1})
      + a_h^I((\bar{u}_h^{f,n+1},d_t\bar{u}_h^{b,n+1}),(\bar{v}_h^f,\bar{v}_h^b))      
    \\ \nonumber
    & \hspace{28em}
      + b_h^I((\bar{v}_h^f,\bar{v}_h^b), \bar{p}_h^{p,n+1})
      = (f^{n+1}, v_h)_{\Omega},
    \\
    \label{eq:hdgfd-b}
    & b_h(\boldsymbol{u}_h^{n+1}, \boldsymbol{q}_h)
      + c_h((p_h^{p,n+1},p_h^{b,n+1}),q_h^b)
      = 0,
    \\
    \label{eq:hdgfd-c}
    &(c_0d_t p_h^{p,n+1}, q_h^p)_{\Omega^b}
      + c_h((d_tp_h^{p,n+1},d_tp_h^{b,n+1}), \alpha q_h^p)
      - b_h^{b}((z_h^{n+1}, 0), \boldsymbol{q}_h^p)      
      - b_h^I((\bar{u}_h^{f,n+1}, d_t\bar{u}_h^{b,n+1}), \bar{q}_h^p)
    \\ \nonumber
    & \hspace{37em}
      =(g^{b,n+1}, q_h^p)_{\Omega^b},
    \\
    \label{eq:hdgfd-d}
    &(\mu^f\kappa^{-1} z_h^{n+1}, w_h)_{\Omega^b} + b_h^{b}((w_h, 0), \boldsymbol{p}_h^{p,n+1})=0.
  \end{align}
\end{subequations}
Note that despite the coupled Navier--Stokes/Biot problem being
nonlinear, the fully-discrete HDG method \cref{eq:hdgfd} is linear at
each time step due to lagging of the convective velocity.

\begin{remark}
  The HDG method \cref{eq:hdgfd} is an extension of the HDG method
  previously presented in \cite{Cesmelioglu:2023a} for the coupled
  Stokes/Biot model. For this HDG method, it was proven in \cite[Lemma
  1]{Cesmelioglu:2023a} that the discrete velocities and displacement
  are divergence-conforming. Furthermore, it was shown that the
  compressibility equations are satisfied pointwise on the elements
  and that, for the semi-discrete problem \cref{eq:hdgsd} and in the
  absence of source/sink terms, that mass is conserved pointwise on
  the elements. These properties are inherited by \cref{eq:hdgfd}. The
  proof is identical to that of \cite[Lemma 1]{Cesmelioglu:2023a} and
  therefore not included here.
\end{remark}

\section{Stability and well-posedness}
\label{sec:wp}

Before showing well-posedness, let us recall the following result from
\cite[Lemma 2]{Cesmelioglu:2023b}: if $\boldsymbol{p}_h^{p,n}$ and
$z_h^n$ are part of the solution to \cref{eq:hdgfd} for $n \ge 1$,
then there exists a constant $c_{pd}$, independent of $h$, such that
\begin{equation}
  \label{eq:tnormppbound}
  \norm[0]{p_h^{p,n}}_{\Omega^b} \le c_{pp}\norm[0]{p_h^{p,n}}_{1,h,\Omega^b} \le
  c_{pp} \tnorm{\boldsymbol{p}_h^{p,n}}_{1,h,b} \le c_{td}\mu^f\kappa^{-1}\norm[0]{z_h^n}_{\Omega^b},
\end{equation}
where $c_{td}=c_{pp}c_{pd}$. 

The following lemma shows uniqueness of the discrete solution at
$t_{n+1}$ given the discrete solution at $t_n$ and under a smallness
assumption on $\norm[0]{u_h^{f,n} \cdot n}_{\Gamma_{IN}^f}$.

\begin{lemma}[Uniqueness]
  \label{lem:uniqueness}
  Let $u_h^0$ and $p_h^{p,0}$ be the given initial conditions. Assume
  $(\boldsymbol{u}_h^{n}, \boldsymbol{p}_h^{n}, z_h^{n},
  \boldsymbol{p}_h^{p,n}) \in \boldsymbol{X}_h$ is the given solution
  to \cref{eq:hdgfd} for some $1 \le n \le N-1$. If
  $u_h^{f,n} \in V_h^f$ is such that
  $\norm[0]{u_h^{f,n} \cdot n}_{\Gamma_{IN}^f} \le
  \tfrac{1}{2}\mu^fc_{ae}^f(c_{pq}^2+c_{si,4}^2)^{-1}$ for
  $0 \le n \le N-1$, and if the solution
  $(\boldsymbol{u}_h^{n+1}, \boldsymbol{p}_h^{n+1}, z_h^{n+1},
  \boldsymbol{p}_h^{p,n+1}) \in \boldsymbol{X}_h$ to \cref{eq:hdgfd}
  exists, then it is unique.
\end{lemma}
\begin{proof}
  Assume that both
  $(\boldsymbol{u}_h^{n+1}, \boldsymbol{p}_h^{n+1}, z_h^{n+1},
  \boldsymbol{p}_h^{p,n+1})$ and
  $(\widehat{\boldsymbol{u}}_h^{n+1},
  \widehat{\boldsymbol{p}}_h^{n+1}, \widehat{z}_h^{n+1},
  \widehat{\boldsymbol{p}}_h^{p,n+1})$ are solutions to the fully
  discrete system \cref{eq:hdgfd}. Let us define their difference by
  $(\boldsymbol{x}_h^{n+1}, \boldsymbol{r}_h^{n+1}, y_h^{n+1},
  \boldsymbol{r}_h^{p,n+1}) =
  (\boldsymbol{u}_h^{n+1}-\widehat{\boldsymbol{u}}_h^{n+1},
  \boldsymbol{p}_h^{n+1}-\widehat{\boldsymbol{p}}_h^{n+1},
  z_h^{n+1}-\widehat{z}_h^{n+1},
  \boldsymbol{p}_h^{p,n+1}-\widehat{\boldsymbol{p}}_h^{p,n+1})$. We
  need to show that
  $(\boldsymbol{x}_h^{n+1}, \boldsymbol{r}_h^{n+1}, y_h^{n+1},
  \boldsymbol{r}_h^{p,n+1}) = (\boldsymbol{0}, \boldsymbol{0}, 0,
  \boldsymbol{0})$. Let us first note that
  $(\boldsymbol{x}_h^{n+1}, \boldsymbol{r}_h^{n+1}, y_h^{n+1},
  \boldsymbol{r}_h^{p,n+1})$ satisfies
  \begin{subequations}
    \begin{align}
      \label{eq:difeq-a}
      & \frac{1}{\Delta t}(x_h^{n+1}, v_h^f)_{\Omega^f}
        + t_h(u_h^{f,n}; \boldsymbol{x}_h^{f,n+1}, \boldsymbol{v}_h^f)
        + a_h^f(\boldsymbol{x}_h^{f,n+1}, \boldsymbol{v}_h^f)
        + a_h^b(\boldsymbol{x}_h^{b,n+1}, \boldsymbol{v}_h^b)
        + b_h^f(\boldsymbol{v}_h^f, \boldsymbol{r}_h^{f,n+1})
      \\ \nonumber
      & \hspace{10em}        
        + b_h^b(\boldsymbol{v}_h^b, \boldsymbol{r}_h^{b,n+1})
        + a_h^I((\bar{x}_h^{f,n+1},\frac{1}{\Delta t}\bar{x}_h^{b,n+1}),(\bar{v}_h^f,\bar{v}_h^b))
        + b_h^I((\bar{v}_h^f,\bar{v}_h^b), \bar{r}_h^{p,n+1})
        = 0,
      \\
      \label{eq:difeq-b}
      & b_h^f(\boldsymbol{x}_h^{f,n+1}, \boldsymbol{q}_h^f)
        + b_h^b(\boldsymbol{x}_h^{b,n+1}, \boldsymbol{q}_h^b)
        + c_h((r_h^{p,n+1},r_h^{b,n+1}),q_h^b)
        = 0,
      \\
      \label{eq:difeq-c}
      &\frac{1}{\Delta t}(c_0r_h^{p,n+1}, q_h^p)_{\Omega^b}
        + \frac{1}{\Delta t}c_h((r_h^{p,n+1},r_h^{b,n+1}), \alpha q_h^p)
        - b_h^{b}((y_h^{n+1}, 0), \boldsymbol{q}_h^p)      
        - b_h^I((\bar{x}_h^{f,n+1}, \frac{1}{\Delta t}\bar{x}_h^{b,n+1}), \bar{q}_h^p)
        = 0,
      \\
      \label{eq:difeq-d}
      &(\mu^f\kappa^{-1} y_h^{n+1}, w_h)_{\Omega^b} + b_h^{b}((w_h, 0), \boldsymbol{r}_h^{p,n+1})=0,
    \end{align}
  \end{subequations}
  for all
  $(\boldsymbol{v}_h, \boldsymbol{q}_h, w_h, \boldsymbol{q}_h^p) \in
  \boldsymbol{X}_h$. Add the above equations, choose
  $\boldsymbol{v}_h^f = \boldsymbol{x}_h^{f,n+1}$,
  $\boldsymbol{v}_h^b = \tfrac{1}{\Delta t}\boldsymbol{x}_h^{b,n+1}$,
  $\boldsymbol{q}_h^f = -\boldsymbol{r}_h^{f,n+1}$,
  $\boldsymbol{q}_h^b = -\tfrac{1}{\Delta t}\boldsymbol{r}_h^{b,n+1}$,
  $w_h = y_h^{n+1}$, and
  $\boldsymbol{q}_h^p = \boldsymbol{r}^{p,n+1}$. By the smallness
  assumption on $\norm[0]{u_h^{f,n} \cdot n}_{\Gamma_{IN}^f}$ we find,
  using Proposition \ref{prop:coercivitythahf} and
  \cref{eq:ah-coercive-j}, that
  \begin{multline*}
    \tfrac{1}{\Delta t}\norm[0]{x_h^{f,n+1}}_{\Omega^f}^2
    + \tfrac{1}{2}c_{ae}^f\mu^f\tnorm{\boldsymbol{x}_h^{f,n+1}}_{v,f}^2    
    + \tfrac{1}{\Delta t}c_{ae}^b\mu^b\tnorm{\boldsymbol{x}_h^{b,n+1}}_{v,b}^2
    + \gamma\mu^f\kappa^{-1/2}\norm[0]{ (\bar{x}_h^{f,n+1} - \tfrac{1}{\Delta t}\bar{x}_h^{b,n+1})^t }_{\Gamma_I}^2
    \\
    + \tfrac{1}{\Delta t}\lambda^{-1}\norm[0]{\alpha r_h^{p,n+1} - r_h^{b,n+1}}_{\Omega^b}^2
    + \tfrac{c_0}{\Delta t}\norm[0]{r_h^{p,n+1}}_{\Omega^b}^2
    + \mu^f\kappa^{-1}\norm[0]{y_h^{n+1}}_{\Omega^b}^2
    \le 0,    
  \end{multline*}
  so that $\boldsymbol{x}_h^{n+1}=\boldsymbol{0}$ and
  $y_h^{n+1}=0$. We are left to show that $\boldsymbol{r}_h^{n+1}$ and
  $\boldsymbol{r}_h^{p,n+1}$ are zero. To show that
  $\boldsymbol{r}_h^{n+1} = \boldsymbol{0}$, substitute
  $\boldsymbol{x}_h^{n+1}=\boldsymbol{0}$ into \cref{eq:difeq-a} and
  choose $\bar{v}_h^f=\bar{v}_h^b=0$ on $\Gamma_I$ to find
  $b_h^j(\boldsymbol{v}_h^j, \boldsymbol{r}_h^{j,n+1})=0$
  $\forall \boldsymbol{v}_h^j \in \widetilde{V}_h^j$, $j=f,b$. The
  result follows by the inf-sup condition
  \cref{eq:inf-sup-bj}. Similarly, to show that
  $\boldsymbol{r}_h^{p,n+1} = \boldsymbol{0}$, substitute
  $y_h^{n+1}=0$ into \cref{eq:difeq-d} to find
  $b_h^b((w_h,0), \boldsymbol{r}_h^{p,n+1})=0$
  $\forall w_h \in V_h^b$. The result follows by the inf-sup condition
  \cref{eq:inf-sup-bp}.
\end{proof}

The goal of the remainder of this section is to show that the
smallness assumption in Lemma \ref{lem:uniqueness} holds. For this, it
will be useful to define, for $0 \le n \le N-1$,
\begin{equation}
  \label{eq:defXn}
  X^n :=
  \norm[0]{d_t u_h^{f,n+1}}_{\Omega^f}^2
  + \tfrac{1}{2}a_h^b(d_t \boldsymbol{u}_h^{b,n+1}, d_t \boldsymbol{u}_h^{b,n+1})
  + \lambda^{-1}\norm[0]{\alpha d_t p_h^{p,n+1} - d_t p_h^{b,n+1}}_{\Omega^b}^2
  + c_0\norm[0]{d_t p_h^{p,n+1}}_{\Omega^b}^2,
\end{equation}
and for $1 \le m \le N-1$,
\begin{multline}
  \label{eq:defFm}
    F^m := \frac{4c_p^2}{c_{ae}^f\mu^f}\Delta t\sum_{k=1}^m\norm[0]{d_t f^{f,k+1}}_{\Omega^f}^2
    + \frac{2c_{td}^2\mu^f}{\kappa}\Delta t\sum_{k=1}^m\norm[0]{d_t g^{b,k+1}}_{\Omega^b}^2 
    \\ 
    + 2 c_p^2(c_{ac}^b \mu^b)^{-1} (\max_{1\le k \le m}\norm[0]{d_t f^{b,k+1}}_{\Omega^b}
    + \norm[0]{d_t f^{b,2}}_{\Omega^b} + \Delta t\sum_{k=2}^m \norm[0]{d_{tt}f^{b,k+1}}_{\Omega^b} )^2. 
\end{multline}
Furthermore, we define
\begin{equation*}
  F^0:=
  3\norm[0]{f^{f,1}}_{\Omega^f}^2
  + \frac{3c_p^2}{c_{ae}^b\mu^b}\norm[0]{d_tf^{b,1}}_{\Omega^b}^2
  + \frac{3c_{td}^2\mu^f}{\kappa} \Delta t \norm[0]{d_tg^{b,1}}_{\Omega^b}^2.    
\end{equation*}
The following lemma extends \cite[Lemma 5]{Cesmelioglu:2023b} for
Navier--Stokes/Darcy to Navier--Stokes/Biot.

\begin{lemma}
  \label{lem:boundXn}
  Let $\boldsymbol{u}_h^{0} = 0$ and $\boldsymbol{p}_h^{p,0} = 0$
  (then also $\boldsymbol{p}_h^{b,0} = \boldsymbol{0}$ and
  $z_h^0 = 0$) and assume $f^{b,0}=0$ and $g^{b,0}=0$. Let
  $(\boldsymbol{u}_h^{k}, \boldsymbol{p}_h^{k}, z_h^{k},
  \boldsymbol{p}_h^{p,k}) \in \boldsymbol{X}_h$ be the solution to
  \cref{eq:hdgfd} for $1 \le k \le n$. Then
  \begin{subequations}
    \label{eq:boundX0terms}
    \begin{align}
      \label{eq:boundX0terms-a}
      X^0
      & \le F^0,
      \\
      \label{eq:boundX0terms-b}
      c_{ae}^f\mu^f\Delta t\tnorm{d_t\boldsymbol{u}_h^{f,1}}_{v,f}^2
      &\le \tfrac{1}{12}F^0.
    \end{align}
  \end{subequations}
  Furthermore, if for $1 \le n \le N-1$,
  \begin{equation}
    \label{eq:assumpuhfkfirst}
    \tnorm{\boldsymbol{u}_h^{f,k}}_{v,f}
    \le \min\del[2]{ \frac{\mu^f c_{ae}^f}{2c_{si,2}(c_{pq}^2 + c_{si,4}^2)}, \frac{c_{ae}^f\mu^f}{4c_w} }
    \quad \forall 0 \le k \le n,
  \end{equation}
  then
  \begin{equation}
    \label{eq:boundXnterms}
    X^n \le 6X^0
    + \tfrac{1}{2}c_{ae}^f\mu^f\Delta t\tnorm{d_t\boldsymbol{u}_h^{f,1}}_{v,f}^2 + F^n.
  \end{equation}
\end{lemma}
\begin{proof}
  Let $n=0$ in \cref{eq:hdgfd}. Use that the initial conditions are
  zero and choose $\boldsymbol{v}_h^f = \boldsymbol{u}_h^{f,1}$,
  $\boldsymbol{v}_h^b = \tfrac{1}{\Delta t}\boldsymbol{u}_h^{b,1}$,
  $\boldsymbol{q}_h^f = - \boldsymbol{p}_h^{f,1}$,
  $\boldsymbol{q}_h^p = \boldsymbol{p}_h^{p,1}$, and $w_h =
  z_h^1$. Furthermore, choose
  $\boldsymbol{q}_h^b = -\tfrac{1}{\Delta t}\boldsymbol{p}_h^{b,1}$
  and note that
  $b_h^b(\boldsymbol{u}_h^{b,1}, -\tfrac{1}{\Delta
    t}\boldsymbol{p}_h^{b,1}) = -b_h^b(\tfrac{1}{\Delta
    t}\boldsymbol{u}_h^{b,1}, \boldsymbol{p}_h^{b,1})$ and
  $c_h((p_h^{p,1},p_h^{b,1}), -\tfrac{1}{\Delta t}p_h^{b,1}) =
  -c_h((\tfrac{1}{\Delta t}p_h^{p,1}, \tfrac{1}{\Delta t}p_h^{b,1}),
  p_h^{b,1})$. \Cref{eq:hdgfd} becomes:
  \begin{multline*}
    \tfrac{1}{\Delta t}\norm[0]{u_h^{f,1}}_{\Omega^f}^2
    + a_h^f(\boldsymbol{u}_h^{f,1}, \boldsymbol{u}_h^{f,1})
    + a_h^b(\boldsymbol{u}_h^{b,1}, \tfrac{1}{\Delta t}\boldsymbol{u}_h^{b,1})
    + \gamma \mu^f\kappa^{-1/2}\norm[0]{ (\bar{u}_h^{f,1}-\tfrac{1}{\Delta t}\bar{u}_h^{b,1})^t }_{\Gamma_I}^2
    + \tfrac{\lambda^{-1}}{\Delta t} \norm[0]{\alpha p_h^{p,1} - p_h^{b,1}}_{\Omega^b}^2
    \\
    + \tfrac{c_0}{\Delta t} \norm[0]{p_h^{p,1}}_{\Omega^b}^2
    + \mu^f\kappa^{-1} \norm[0]{z_h^{1}}_{\Omega^b}^2
    =
    (f^{f,1}, u_h^{f,1})_{\Omega^f} + (f^{b,1}, \tfrac{1}{\Delta t}u_h^{b,1})_{\Omega^b}
    + (g^{b,1}, p_h^{p,1})_{\Omega^b}.    
  \end{multline*}
  Coercivity of $a_h^f$ (see \cref{eq:ah-coercive-j}), the
  Cauchy--Schwarz inequality, and using \cref{eq:tnormppbound} so that
  $\norm[0]{p_h^{p,1}}_{\Omega^b} \le
  c_{td}\mu^f\kappa^{-1}\norm[0]{z_h^{1}}_{\Omega^b}$, we obtain
  \begin{multline}
    \label{eq:uhf1termubf1}
    \tfrac{1}{\Delta t}\norm[0]{u_h^{f,1}}_{\Omega^f}^2
    + c_{ae}^f\mu^f\tnorm{\boldsymbol{u}_h^{f,1}}_{v,f}^2
    + \tfrac{1}{\Delta t}a_h^b(\boldsymbol{u}_h^{b,1}, \boldsymbol{u}_h^{b,1})
    + \gamma \mu^f\kappa^{-1/2}\norm[0]{ (\bar{u}_h^{f,1}-\tfrac{1}{\Delta t}\bar{u}_h^{b,1})^t }_{\Gamma_I}^2
    + \tfrac{\lambda^{-1}}{\Delta t} \norm[0]{\alpha p_h^{p,1} - p_h^{b,1}}_{\Omega^b}^2
    \\
    + \tfrac{c_0}{\Delta t} \norm[0]{p_h^{p,1}}_{\Omega^b}^2
    + \mu^f\kappa^{-1} \norm[0]{z_h^{1}}_{\Omega^b}^2
    \le
    \norm[0]{f^{f,1}}_{\Omega^f}\norm[0]{u_h^{f,1}}_{\Omega^f} + \tfrac{1}{\Delta t}\norm[0]{f^{b,1}}_{\Omega^b}\norm[0]{u_h^{b,1}}_{\Omega^b}
    + c_{td}\mu^f\kappa^{-1}\norm[0]{g^{b,1}}_{\Omega^b}\norm[0]{z_h^{1}}_{\Omega^b}.
  \end{multline}
  Using \cref{eq:dpoincareineq-b} and \cref{eq:ah-coercive-j} so that
  \begin{equation}
  \label{eq:uhb1-ahb}
  \norm[0]{u_h^{b,1}}_{\Omega^b} \le
    c_p\tnorm{\boldsymbol{u}_h^{b,1}}_{v,b} \le c_p
     (c_{ae}^b\mu^b)^{-1/2}a_h^b(\boldsymbol{u}_h^{b,1},\boldsymbol{u}_h^{b,1})^{1/2},
  \end{equation}
  and noting the nonnegativity of the second and the fourth terms of
  \cref{eq:uhf1termubf1}, we further have
  \begin{multline}
    \label{eq:uf1leterm}
    \tfrac{1}{\Delta t}\norm[0]{u_h^{f,1}}_{\Omega^f}^2
    + \tfrac{1}{\Delta t}a_h^b(\boldsymbol{u}_h^{b,1}, \boldsymbol{u}_h^{b,1})
    + \tfrac{\lambda^{-1}}{\Delta t} \norm[0]{\alpha p_h^{p,1} - p_h^{b,1}}_{\Omega^b}^2
    + \tfrac{c_0}{\Delta t} \norm[0]{p_h^{p,1}}_{\Omega^b}^2
    + \mu^f\kappa^{-1} \norm[0]{z_h^{1}}_{\Omega^b}^2
    \\
    \le
    \norm[0]{f^{f,1}}_{\Omega^f}\norm[0]{u_h^{f,1}}_{\Omega^f}
    + \tfrac{1}{\Delta t}c_p (c_{ae}^b\mu^b)^{-1/2}\norm[0]{f^{b,1}}_{\Omega^b}a_h^b(\boldsymbol{u}_h^{b,1},\boldsymbol{u}_h^{b,1})^{1/2}
    + c_{td}\mu^f\kappa^{-1}\norm[0]{g^{b,1}}_{\Omega^b}\norm[0]{z_h^{1}}_{\Omega^b}.    
  \end{multline}
  Let us define  
  \begin{equation*}
    \mathcal{Z}^2 :=
    \norm[0]{u_h^{f,1}}_{\Omega^f}^2
    + a_h^b(\boldsymbol{u}_h^{b,1}, \boldsymbol{u}_h^{b,1})
    + \lambda^{-1}\norm[0]{\alpha p_h^{p,1} - p_h^{b,1}}_{\Omega^b}^2
    + c_0\norm[0]{p_h^{p,1}}_{\Omega^b}^2
    + \Delta t\mu^f\kappa^{-1} \norm[0]{z_h^{1}}_{\Omega^b}^2
  \end{equation*}
  and write \cref{eq:uf1leterm} as:
  \begin{equation*}
    \tfrac{1}{\Delta t}\mathcal{Z}^2 \le \big(
    \norm[0]{f^{f,1}}_{\Omega^f}
    + \tfrac{1}{\Delta t}c_p (c_{ae}^b\mu^b)^{-1/2}\norm[0]{f^{b,1}}_{\Omega^b}
    + c_{td}\tfrac{1}{(\Delta t)^{1/2}}(\mu^f\kappa^{-1})^{1/2}\norm[0]{g^{b,1}}_{\Omega^b}\big)\mathcal{Z}.
  \end{equation*}
  This immediately implies, using that
  $(X^0)^{1/2} \le \tfrac{1}{\Delta t}\mathcal{Z}$ and that
  $\tfrac{1}{\Delta t}\norm[0]{f^{b,1}}_{\Omega^b} =
  \norm[0]{d_tf^{b,1}}_{\Omega^b}$ and
  $ \tfrac{1}{(\Delta t)^{1/2}}\norm[0]{g^{b,1}}_{\Omega^b} = (\Delta
  t)^{1/2}\norm[0]{d_tg^{b,1}}_{\Omega^b}$ because $f^{b,0}=0$ and
  $g^{b,0}=0$,
  \begin{equation*}
    (X^0)^{1/2} \le
    \norm[0]{f^{f,1}}_{\Omega^f}
    + c_p (c_{ae}^b\mu^b)^{-1/2}\norm[0]{d_tf^{b,1}}_{\Omega^b}
    + c_{td}(\Delta t)^{1/2}(\mu^f\kappa^{-1})^{1/2}\norm[0]{d_tg^{b,1}}_{\Omega^b},
  \end{equation*}
  so that \cref{eq:boundX0terms-a} follows after squaring.

  To prove \cref{eq:boundX0terms-b} we return to
  \cref{eq:uhf1termubf1}. Applying Young's inequality
  $ab \le a^2/(2\psi) + \psi b^2/2$ to each term on the right hand
  side, using \cref{eq:uhb1-ahb} for the second term, choosing
  $\psi = 2/\Delta t$, $\psi=2/\Delta t$, and
  $\psi = 2\mu^f\kappa^{-1}$ for the first, second, and third terms,
  respectively, and dividing both sides by $\Delta t$, we obtain,
  using that $f^{b,0}=0$ and $g^{b,0}=0$,
  \begin{equation*}
    c_{ae}^f\mu^f\Delta t\tnorm{d_t\boldsymbol{u}_h^{f,1}}_{v,f}^2
    \le
    \tfrac{1}{4}\norm[0]{f^{f,1}}_{\Omega^f}^2
    + \tfrac{1}{4}c_p^2(c_{ae}^b\mu^b)^{-1}\norm[0]{d_tf^{b,1}}_{\Omega^b}^2
    + \tfrac{1}{4}c_{td}^2\mu^f\kappa^{-1}\Delta t\norm[0]{d_tg^{b,1}}_{\Omega^b}^2.
  \end{equation*}
  This proves \cref{eq:boundX0terms-b}.
  
  We now prove \cref{eq:boundXnterms}. Let $1 \le n \le N-1$. Subtract
  \cref{eq:hdgfd} for the solution at time-level $t_n$ from
  \cref{eq:hdgfd} for the solution at time-level $t_{n+1}$, choose
  $\boldsymbol{v}_h^f = \delta \boldsymbol{u}_h^{f,n+1}$,
  $\boldsymbol{v}_h^b = \tfrac{1}{\Delta t}(\delta
  \boldsymbol{u}_h^{b,n+1}-\delta \boldsymbol{u}_h^{b,n})$,
  $\boldsymbol{q}_h^f=-\delta \boldsymbol{p}_h^{f,n+1}$,
  $\boldsymbol{q}_h^b = -\tfrac{1}{\Delta
    t}\delta\boldsymbol{p}_h^{b,n+1}$, $w_h = \delta z_h^{n+1}$,
  $\boldsymbol{q}_h^p = \delta\boldsymbol{p}_h^{p,n+1}$ and add the
  resulting equations:
  \begin{equation}
    \label{eq:afteraddingequationswithtfs}
    \begin{split}
      & \tfrac{1}{\Delta t}(\delta u_h^{f,n+1} - \delta u_h^{f,n}, \delta u_h^{f,n+1})_{\Omega^f}
      + t_h(u_h^{f,n}; \boldsymbol{u}_h^{f,n+1}, \delta \boldsymbol{u}_h^{f,n+1})
      - t_h(u_h^{f,n-1}; \boldsymbol{u}_h^{f,n}, \delta \boldsymbol{u}_h^{f,n+1})
      \\
      & \hspace{2em}
      + a_h^f(\delta \boldsymbol{u}_h^{f,n+1}, \delta \boldsymbol{u}_h^{f,n+1})
      + a_h^b(\delta \boldsymbol{u}_h^{b,n+1}, \tfrac{1}{\Delta t}(\delta \boldsymbol{u}_h^{b,n+1}-\delta \boldsymbol{u}_h^{b,n}))
      \\
      & \hspace{2em}
      + b_h^b(\tfrac{1}{\Delta t}(\delta \boldsymbol{u}_h^{b,n+1}-\delta \boldsymbol{u}_h^{b,n}), \delta \boldsymbol{p}_h^{b,n+1})
      - b_h^b(\delta \boldsymbol{u}_h^{b,n+1}, \tfrac{1}{\Delta t}\delta\boldsymbol{p}_h^{b,n+1})
      \\
      & \hspace{2em}
      + a_h^I((\delta \bar{u}_h^{f,n+1}, \tfrac{1}{\Delta t}(\delta \bar{u}_h^{b,n+1}- \delta \bar{u}_h^{b,n})),
      (\delta\bar{u}_h^{f,n+1}, \tfrac{1}{\Delta t}(\delta \bar{u}_h^{b,n+1}-\delta \bar{u}_h^{b,n})))
      \\
      & - c_h((\delta p_h^{p,n+1}, \delta p_h^{b,n+1}), \tfrac{1}{\Delta t}\delta p_h^{b,n+1})
      + \tfrac{1}{\Delta t} c_h((\delta p_h^{p,n+1} - \delta p_h^{p,n}, \delta p_h^{b,n+1} - \delta p_h^{b,n}), \alpha \delta p_h^{p,n+1})
      \\
      & + \tfrac{c_0}{\Delta t}(\delta p_h^{p,n+1} - \delta p_h^{p,n}, \delta p_h^{p,n+1})_{\Omega^b}      
      + (\mu^f\kappa^{-1}\delta z_h^{n+1}, \delta z_h^{n+1})_{\Omega^b} 
      \\
      & = (\delta f^{f,n+1}, \delta u_h^{f,n+1})_{\Omega^f}
      + (\delta f^{b,n+1}, \tfrac{1}{\Delta t}(\delta u_h^{b,n+1} - \delta u_h^{b,n}))_{\Omega^b}
      + (\delta g^{b,n+1}, \delta p_h^{p,n+1})_{\Omega^b}.      
    \end{split}
  \end{equation}
  We simplify next the sum of the $6^{\rm th}$, $7^{\rm th}$,
  $9^{\rm th}$, and $10^{\rm th}$ terms on the left hand side of
  \cref{eq:afteraddingequationswithtfs} as follows. First,
  \begin{multline*}
    \sbr[0]{b_h^b(\tfrac{1}{\Delta t}(\delta \boldsymbol{u}_h^{b,n+1}-\delta \boldsymbol{u}_h^{b,n}), \delta \boldsymbol{p}_h^{b,n+1})
    - b_h^b(\delta \boldsymbol{u}_h^{b,n+1}, \tfrac{1}{\Delta t}\delta\boldsymbol{p}_h^{b,n+1})}
    - c_h((\delta p_h^{p,n+1}, \delta p_h^{b,n+1}), \tfrac{1}{\Delta t}\delta p_h^{b,n+1})
    \\
    + \tfrac{1}{\Delta t} c_h((\delta p_h^{p,n+1} - \delta p_h^{p,n}, \delta p_h^{b,n+1} - \delta p_h^{b,n}),
    \alpha \delta p_h^{p,n+1}) =: I_1+I_2+I_3.    
  \end{multline*}
  Note that 
  \begin{equation*}
    \begin{split}
      I_1 =
      b_h^b(\tfrac{1}{\Delta t}\delta \boldsymbol{u}_h^{b,n+1}, \delta \boldsymbol{p}_h^{b,n+1})
      - b_h^b(\tfrac{1}{\Delta t}\delta \boldsymbol{u}_h^{b,n}, \delta \boldsymbol{p}_h^{b,n+1})
      - b_h^b(\delta \boldsymbol{u}_h^{b,n+1}, \tfrac{1}{\Delta t}\delta\boldsymbol{p}_h^{b,n+1})
      =&- b_h^b(\tfrac{1}{\Delta t}\delta \boldsymbol{u}_h^{b,n}, \delta \boldsymbol{p}_h^{b,n+1})
      \\
      =& c_h((\delta p_h^{p,n}, \delta p_h^{b,n}), \tfrac{1}{\Delta t} \delta p_h^{b,n+1}),
    \end{split}
  \end{equation*}
  where the last equality follows by subtracting \cref{eq:hdgfd-b} at
  time-level $t_n$ from \cref{eq:hdgfd-b} at time-level $t_{n+1}$ and
  choosing $\boldsymbol{q}_h^f=\boldsymbol{0}$,
  $\boldsymbol{q}_h^b = \tfrac{1}{\Delta t} \delta
  \boldsymbol{p}_h^{b,n+1}$. We immediately find that
  \begin{equation*}
    I_1 + I_2 +I_3
    = \tfrac{1}{\Delta t} c_h((\delta p_h^{p,n+1} - \delta p_h^{p,n}, \delta p_h^{b,n+1} - \delta p_h^{b,n}),
    \alpha \delta p_h^{p,n+1} - \delta p_h^{b,n+1}).
  \end{equation*}
  Noting also that
  $t_h(u_h^{f,n}; \boldsymbol{u}_h^{f,n+1}, \delta
  \boldsymbol{u}_h^{f,n+1}) = t_h(u_h^{f,n}; \delta
  \boldsymbol{u}_h^{f,n+1}, \delta \boldsymbol{u}_h^{f,n+1}) +
  t_h(u_h^{f,n}; \boldsymbol{u}_h^{f,n}, \delta
  \boldsymbol{u}_h^{f,n+1})$, we write
  \cref{eq:afteraddingequationswithtfs} as:
  \begin{equation}
    \label{eq:1dtdeltauhfnp1}
    \begin{split}
      \tfrac{1}{\Delta t}&(\delta u_h^{f,n+1} - \delta u_h^{f,n}, \delta u_h^{f,n+1})_{\Omega^f}
      + t_h(u_h^{f,n}; \delta \boldsymbol{u}_h^{f,n+1}, \delta \boldsymbol{u}_h^{f,n+1})
      + a_h^f(\delta \boldsymbol{u}_h^{f,n+1}, \delta \boldsymbol{u}_h^{f,n+1})
      \\
      &
      + a_h^b(\delta \boldsymbol{u}_h^{b,n+1}, \tfrac{1}{\Delta t}(\delta \boldsymbol{u}_h^{b,n+1}-\delta \boldsymbol{u}_h^{b,n}))
      \\
      & + a_h^I((\delta \bar{u}_h^{f,n+1}, \tfrac{1}{\Delta t}(\delta \bar{u}_h^{b,n+1}- \delta \bar{u}_h^{b,n})),
      (\delta\bar{u}_h^{f,n+1}, \tfrac{1}{\Delta t}(\delta \bar{u}_h^{b,n+1}-\delta \bar{u}_h^{b,n})))
      \\
      & + \tfrac{1}{\Delta t} c_h((\delta p_h^{p,n+1} - \delta p_h^{p,n}, \delta p_h^{b,n+1} - \delta p_h^{b,n}), \alpha \delta p_h^{p,n+1} - \delta p_h^{b,n+1})
      \\
      & + \tfrac{c_0}{\Delta t}(\delta p_h^{p,n+1} - \delta p_h^{p,n}, \delta p_h^{p,n+1})_{\Omega^b}      
      + (\mu^f\kappa^{-1}\delta z_h^{n+1}, \delta z_h^{n+1})_{\Omega^b} 
      \\
      =& (\delta f^{f,n+1}, \delta u_h^{f,n+1})_{\Omega^f} + (\delta f^{b,n+1}, \tfrac{1}{\Delta t}(\delta u_h^{b,n+1} - \delta u_h^{b,n}))_{\Omega^b} + (\delta g^{b,n+1}, \delta p_h^{p,n+1})_{\Omega^b}
      \\
      &-\sbr[1]{t_h(u_h^{f,n}; \boldsymbol{u}_h^{f,n}, \delta \boldsymbol{u}_h^{f,n+1})
        - t_h(u_h^{f,n-1}; \boldsymbol{u}_h^{f,n}, \delta \boldsymbol{u}_h^{f,n+1})}.
    \end{split}
  \end{equation}
  For the left hand side of \cref{eq:1dtdeltauhfnp1}, by assumption
  \cref{eq:assumpuhfkfirst} and \cref{eq:dtrpoincareineq},
  $\norm[0]{u_h^{f,n}}_{\Gamma_{IN}^f} \le c_{si,2}
  \tnorm{\boldsymbol{u}_h^{f,n}}_{v,f} \le \mu^f c_{ae}^f/(2(c_{pq}^2
  + c_{si,4}^2))$. Therefore, using \cref{eq:coercivityahthresult},
  \begin{equation*}
    t_h(u_h^{f,n}; \delta \boldsymbol{u}_h^{f,n+1}, \delta \boldsymbol{u}_h^{f,n+1})
    + a_h^f(\delta \boldsymbol{u}_h^{f,n+1}, \delta \boldsymbol{u}_h^{f,n+1})
    \ge
    \tfrac{1}{2}c_{ae}^f\mu^f\tnorm{\delta \boldsymbol{u}_h^{f,n+1}}_{v,f}^2.
  \end{equation*}
  For the right hand side of \cref{eq:1dtdeltauhfnp1} (see also
  \cite[Proof of Lemma 4.5]{Cesmelioglu:2023b}):
  \begin{equation*}
    \begin{split}
      |t_h(u_h^{f,n}; \boldsymbol{u}_h^{f,n}, \delta \boldsymbol{u}_h^{f,n+1})
      - t_h(u_h^{f,n-1}; \boldsymbol{u}_h^{f,n}, \delta \boldsymbol{u}_h^{f,n+1})|
    \le& c_w \norm[0]{u_h^{f,n} - u_h^{f,n-1}}_{1,h,\Omega^f} \tnorm{\boldsymbol{u}_h^{f,n}}_{v,f} \tnorm{\delta \boldsymbol{u}_h^{f,n+1}}_{v,f}
    \\
    =& c_w \norm[0]{\delta u_h^{f,n}}_{1,h,\Omega^f} \tnorm{\boldsymbol{u}_h^{f,n}}_{v,f} \tnorm{\delta \boldsymbol{u}_h^{f,n+1}}_{v,f}
    \\
    \le& c_w \tnorm{\delta \boldsymbol{u}_h^{f,n}}_{v,f} \tnorm{\boldsymbol{u}_h^{f,n}}_{v,f} \tnorm{\delta \boldsymbol{u}_h^{f,n+1}}_{v,f}.
    \end{split}
  \end{equation*}
  Applying also the Cauchy--Schwarz inequality to the terms on the
  right hand side of \cref{eq:1dtdeltauhfnp1} and using
  $\norm[0]{\delta u_h^{f,n+1}}_{\Omega^f} \le c_p \tnorm{\delta
    \boldsymbol{u}_h^{f,n+1}}_{v,f}$ (by \cref{eq:dpoincareineq}), and
  $\norm[0]{\delta p_h^{p,n+1}}_{\Omega^b} \le
  c_{td}\mu^f\kappa^{-1}\norm[0]{\delta z_h^{n+1}}_{\Omega^b}$ (by a
  simple modification of the proof to \cref{eq:tnormppbound}), we
  obtain
  \begin{equation}
    \label{eq:duhfnp1duhfnterms}
    \begin{split}
      \tfrac{1}{\Delta t}&(\delta u_h^{f,n+1} - \delta u_h^{f,n}, \delta u_h^{f,n+1})_{\Omega^f}
      + \tfrac{1}{\Delta t} a_h^b(\delta \boldsymbol{u}_h^{b,n+1}, \delta \boldsymbol{u}_h^{b,n+1}-\delta \boldsymbol{u}_h^{b,n})
      \\
      & + \tfrac{1}{\Delta t} c_h((\delta p_h^{p,n+1} - \delta p_h^{p,n}, \delta p_h^{b,n+1} - \delta p_h^{b,n}), \alpha \delta p_h^{p,n+1} - \delta p_h^{b,n+1})
      \\
      & + \tfrac{c_0}{\Delta t}(\delta p_h^{p,n+1} - \delta p_h^{p,n}, \delta p_h^{p,n+1})_{\Omega^b}
      + \tfrac{1}{2}c_{ae}^f\mu^f\tnorm{\delta \boldsymbol{u}_h^{f,n+1}}_{v,f}^2
      \\
      &+ \mu^f\kappa^{-1}\norm[0]{\delta z_h^{n+1}}_{\Omega^b}^2
      + \gamma \mu^f\kappa^{-1/2} \norm[0]{ (\delta \bar{u}_h^{f,n+1} - \tfrac{1}{\Delta t}(\delta \bar{u}_h^{b,n+1}- \delta \bar{u}_h^{b,n})^t}_{\Gamma_I}^2
      \\
      \le& c_p \norm[0]{\delta f^{f,n+1}}_{\Omega^f}\tnorm{\delta \boldsymbol{u}_h^{f,n+1}}_{v,f}
      + \tfrac{1}{\Delta t}(\delta f^{b,n+1}, \delta u_h^{b,n+1} - \delta u_h^{b,n})_{\Omega^b}
      \\
      &  + c_{td}\mu^f\kappa^{-1}\norm[0]{\delta g^{b,n+1}}_{\Omega^b}\norm[0]{\delta z_h^{n+1}}_{\Omega^b}
      + c_w \tnorm{\delta \boldsymbol{u}_h^{f,n}}_{v,f} \tnorm{\boldsymbol{u}_h^{f,n}}_{v,f} \tnorm{\delta \boldsymbol{u}_h^{f,n+1}}_{v,f}.
    \end{split}
  \end{equation}
  Apply Young's inequality $ab \le a^2/(2\psi) + \psi b^2/2$ to the
  first, third, and fourth terms on the right hand side of
  \cref{eq:duhfnp1duhfnterms}, choose
  $\psi = \tfrac{1}{2}c_{ae}^f\mu^f$, $\psi = \mu^f\kappa^{-1}$, and
  $\psi = 1$ for the first, third, and fourth terms, respectively, and
  note that by assumption \cref{eq:assumpuhfkfirst} that
  $\tfrac{1}{2}c_{ae}^f\mu^f - c_w
  \tnorm{\boldsymbol{u}_h^{f,n}}_{v,f} \ge
  \tfrac{1}{4}c_{ae}^f\mu^f$. We find from \cref{eq:duhfnp1duhfnterms}:
  \begin{equation}
    \label{eq:duhfnp1duhfnterms2}
    \begin{split}
      \tfrac{1}{\Delta t} &(\delta u_h^{f,n+1} - \delta u_h^{f,n}, \delta u_h^{f,n+1})_{\Omega^f}
      + \tfrac{1}{\Delta t} a_h^b(\delta \boldsymbol{u}_h^{b,n+1}, \delta \boldsymbol{u}_h^{b,n+1}-\delta \boldsymbol{u}_h^{b,n})
      \\
      & + \tfrac{1}{\Delta t} c_h((\delta p_h^{p,n+1} - \delta p_h^{p,n}, \delta p_h^{b,n+1} - \delta p_h^{b,n}), \alpha \delta p_h^{p,n+1} - \delta p_h^{b,n+1})
      \\
      & + \tfrac{c_0}{\Delta t}(\delta p_h^{p,n+1} - \delta p_h^{p,n}, \delta p_h^{p,n+1})_{\Omega^b}
      + \tfrac{1}{8}c_{ae}^f\mu^f\tnorm{\delta \boldsymbol{u}_h^{f,n+1}}_{v,f}^2
      \\
      &+ \tfrac{1}{2}\mu^f\kappa^{-1}\norm[0]{\delta z_h^{n+1}}_{\Omega^b}^2
      + \gamma \mu^f\kappa^{-1/2} \norm[0]{ (\delta \bar{u}_h^{f,n+1} - \tfrac{1}{\Delta t}(\delta \bar{u}_h^{b,n+1}- \delta \bar{u}_h^{b,n})^t}_{\Gamma_I}^2
      \\
      \le& \frac{c_p^2}{c_{ae}^f\mu^f}\norm[0]{\delta f^{f,n+1}}_{\Omega^f}^2
      + \frac{c_{td}^2\mu^f}{2\kappa}\norm[0]{\delta g^{b,n+1}}_{\Omega^b}^2 
      + \tfrac{1}{\Delta t}(\delta f^{b,n+1}, \delta u_h^{b,n+1} - \delta u_h^{b,n})_{\Omega^b}
      + \frac{c_w}{2}\tnorm{\boldsymbol{u}_h^{f,n}}_{v,f}\tnorm{\delta \boldsymbol{u}_h^{f,n}}_{v,f}^2.      
    \end{split}
  \end{equation}
  For the last term on the right hand side of
  \cref{eq:duhfnp1duhfnterms2} use
  $\tnorm{\boldsymbol{u}_h^{f,n}} \le c_{ae}^f\mu^f/(4c_w)$
  (assumption \cref{eq:assumpuhfkfirst}). Furthermore, by
  $2a(a-b) = a^2-b^2 + (a-b)^2$, coercivity of $a_h^b$ (see
  \cref{eq:ah-coercive-j}), and \cref{eq:dpoincareineq-b}, we find:
  \begin{equation*}
    \begin{split}
      2a_h^b(\delta \boldsymbol{u}_h^{b,n+1}, \delta \boldsymbol{u}_h^{b,n+1}-\delta \boldsymbol{u}_h^{b,n})
      =& a_h^b(\delta \boldsymbol{u}_h^{b,n+1}, \delta \boldsymbol{u}_h^{b,n+1})
      - a_h^b(\delta \boldsymbol{u}_h^{b,n}, \delta \boldsymbol{u}_h^{b,n})
      \\
      &+ a_h^b(\delta \boldsymbol{u}_h^{b,n+1} - \delta \boldsymbol{u}_h^{b,n},
      \delta \boldsymbol{u}_h^{b,n+1} - \delta \boldsymbol{u}_h^{b,n})
      \\
      \ge& a_h^b(\delta \boldsymbol{u}_h^{b,n+1}, \delta \boldsymbol{u}_h^{b,n+1})
      - a_h^b(\delta \boldsymbol{u}_h^{b,n}, \delta \boldsymbol{u}_h^{b,n})
      + c_{ae}^b\mu^bc_p^{-1}\norm[0]{\delta u_h^{b,n+1} - \delta u_h^{b,n}}_{\Omega^b}^2.
    \end{split}
  \end{equation*}
  Also using $2a(a-b) \ge a^2-b^2$, we obtain after multiplying
  \cref{eq:duhfnp1duhfnterms2} by $2/\Delta t$, replacing $n$ by $k$,
  using that $d_tf^{k+1}=\Delta t^{-1}\delta f^{k+1}$, and summing for
  $k=1$ to $k=n$:
  \begin{equation}
    \label{eq:aftersumminguhfnp1}
    \begin{split}
      & \norm[0]{d_t u_h^{f,n+1}}_{\Omega^f}^2
      + a_h^b(d_t \boldsymbol{u}_h^{b,n+1}, d_t \boldsymbol{u}_h^{b,n+1})
      + c_{ae}^b\mu^bc_p^{-1}\sum_{k=1}^n\norm[0]{d_t u_h^{b,k+1} - d_t u_h^{b,k}}_{\Omega^b}^2
      \\
      & + \lambda^{-1}\norm[0]{\alpha d_t p_h^{p,n+1} - d_t p_h^{b,n+1}}_{\Omega^b}^2
      + c_0\norm[0]{d_t p_h^{p,n+1}}_{\Omega^b}^2
      + \tfrac{1}{4}c_{ae}^f\mu^f\Delta t \tnorm{d_t\boldsymbol{u}_h^{f,n+1}}_{v,f}^2
      \\
      &+ \mu^f\kappa^{-1}\Delta t\sum_{k=1}^n\norm[0]{d_t z_h^{k+1}}_{\Omega^b}^2
      + 2\gamma \mu^f\kappa^{-1/2}\Delta t\sum_{k=1}^n \norm[0]{ d_t\bar{u}_h^{f,k+1}
        - \tfrac{1}{\Delta t}(d_t\bar{u}_h^{b,k+1}- d_t\bar{u}_h^{b,k})^t}_{\Gamma_I}^2
      \\
      \le &
      \norm[0]{d_t u_h^{f,1}}_{\Omega^f}^2
      + a_h^b(d_t \boldsymbol{u}_h^{b,1}, d_t \boldsymbol{u}_h^{b,1})
      + \lambda^{-1}\norm[0]{\alpha d_t p_h^{p,1} - d_t p_h^{b,1}}_{\Omega^b}^2
      + c_0\norm[0]{d_t p_h^{p,1}}_{\Omega^b}^2
      \\
      & + \tfrac{1}{4}c_{ae}^f\mu^f\Delta t\tnorm{d_t\boldsymbol{u}_h^{f,1}}_{v,f}^2
      + \frac{2c_p^2}{c_{ae}^f\mu^f}\Delta t\sum_{k=1}^n\norm[0]{d_t f^{f,k+1}}_{\Omega^f}^2
      \\
      &       
      + \frac{c_{td}^2\mu^f}{\kappa}\Delta t\sum_{k=1}^n\norm[0]{d_t g^{b,k+1}}_{\Omega^b}^2
      + \tfrac{2}{(\Delta t)^2}\sum_{k=1}^n(\delta f^{b,k+1}, \delta u_h^{b,k+1} - \delta u_h^{b,k})_{\Omega^b}.      
    \end{split}
  \end{equation}
  Apply summation-by-parts to the last term on the right hand side of
  \cref{eq:aftersumminguhfnp1}:
  \begin{equation}
    \label{eq:sbpterms}
    \begin{split}
      &\tfrac{1}{(\Delta t)^2} \sum_{k=1}^n(\delta f^{b,k+1}, \delta u_h^{b,k+1} - \delta u_h^{b,k})_{\Omega^b}
      \\
      &= (d_tf^{b,n+1}, d_tu_h^{b,n+1})_{\Omega^b} - (d_tf^{b,2}, d_tu_h^{b,1})_{\Omega^b}
      - \Delta t \sum_{k=2}^n(d_{tt}f^{b,k+1}, d_tu_h^{b,k})_{\Omega^b}
      \\
      &\le \norm[0]{d_t f^{b,n+1}}_{\Omega^b} \norm[0]{d_t u_h^{b,n+1}}_{\Omega^b}
      + \norm[0]{d_t f^{b,2}}_{\Omega^b} \norm[0]{d_t u_h^{b,1}}_{\Omega^b}
      + \Delta t \sum_{k=2}^n \norm[0]{d_{tt}f^{b,k+1}}_{\Omega^b} \norm[0]{d_t u_h^{b,k}}_{\Omega^b},
    \end{split}
  \end{equation}
  and note that by \cref{eq:dpoincareineq-b} and \cref{eq:ah-coercive-j} we have
  \begin{equation}
    \label{eq:dtuhbkineq}
    \norm[0]{d_tu_h^{b,k}}_{\Omega^b}^2 \le c_p^2\tnorm{d_t\boldsymbol{u}_h^{b,k}}_{v,b}
    \le c_p^2(c_{ae}^b\mu^b)^{-1}a_h^b(d_t\boldsymbol{u}_h^{b,k}, d_t\boldsymbol{u}_h^{b,k}).
  \end{equation}
  Combining \cref{eq:aftersumminguhfnp1,eq:sbpterms,eq:dtuhbkineq},
  and using the definition of $X^n$ (see \cref{eq:defXn}), we obtain
  \begin{multline*}
    X^n
    \le 3X^0
    + \tfrac{1}{4}c_{ae}^f\mu^f\Delta t\tnorm{d_t\boldsymbol{u}_h^{f,1}}_{v,f}^2
    + \frac{2c_p^2}{c_{ae}^f\mu^f}\Delta t\sum_{k=1}^n\norm[0]{d_t f^{f,k+1}}_{\Omega^f}^2
    + \frac{c_{td}^2\mu^f}{\kappa}\Delta t\sum_{k=1}^n\norm[0]{d_t g^{b,k+1}}_{\Omega^b}^2
    \\
    + c_p(c_{ac}^b \mu^b)^{-1/2} (\norm[0]{d_t f^{b,n+1}}_{\Omega^b} (2X^{n})^{1/2}
    + \norm[0]{d_t f^{b,2}}_{\Omega^b} (2X^0)^{1/2}
    + \Delta t\sum_{k=2}^n \norm[0]{d_{tt}f^{b,k+1}}_{\Omega^b} (2X^{k-1})^{1/2}).
  \end{multline*}
  Assume $\max_{1\le k \le n}X^k = X^n$. Then, 
  \begin{multline}
    \label{eq:Xnlethanterms1}
    X^n
    \le 3X^0
    + \tfrac{1}{4}c_{ae}^f\mu^f\Delta t\tnorm{d_t\boldsymbol{u}_h^{f,1}}_{v,f}^2
    + \frac{2c_p^2}{c_{ae}^f\mu^f}\Delta t\sum_{k=1}^n\norm[0]{d_t f^{f,k+1}}_{\Omega^f}^2
    + \frac{c_{td}^2\mu^f}{\kappa}\Delta t\sum_{k=1}^n\norm[0]{d_t g^{b,k+1}}_{\Omega^b}^2
    \\
    + \sqrt{2}c_p(c_{ac}^b \mu^b)^{-1/2} (\norm[0]{d_t f^{b,n+1}}_{\Omega^b} + \norm[0]{d_t f^{b,2}}_{\Omega^b}
    + \Delta t\sum_{k=2}^n \norm[0]{d_{tt}f^{b,k+1}}_{\Omega^b} ) (X^n)^{1/2}.
  \end{multline}
  For nonnegative $A$ and $B$ the following holds:
  \begin{equation*}
    X^n \le A + B (X^n)^{1/2}
    \Leftrightarrow \del[2]{ (X^n)^{1/2}- \tfrac{1}{2}B }^2
    \le A + \tfrac{1}{4}B^2
    \Rightarrow X^n \le \sbr[2]{ \tfrac{1}{2}B + \del[1]{A + \tfrac{1}{4}B^2}^{1/2} }^2
    \le 2A + B^2.
  \end{equation*}
  Applying this to \cref{eq:Xnlethanterms1} and using the definition
  of $F^n$ (see \cref{eq:defFm}),
  \begin{equation}
    \label{eq:finalXnresult}
    X^n \le 6 X^0 + \tfrac{1}{2}c_{ae}^f\mu^f\Delta t\tnorm{d_t\boldsymbol{u}_h^{f,1}}_{v,f}^2 + F^n,
  \end{equation}
  Note that if the assumption $\max_{1\le k \le n}X^k = X^n$ does not
  hold, then there exists $0 \le m < n$ such that
  $\max_{1\le k \le n} X^k = X^m$. In this case,
  \cref{eq:finalXnresult} holds with $n$ replaced by $m$ and we find:
  \begin{equation*}
    X^n < X^m
    \le 6 X^0
    + \tfrac{1}{2}c_{ae}^f\mu^f\Delta t\tnorm{d_t\boldsymbol{u}_h^{f,1}}_{v,f}^2 + F^m
    \le 6 X^0
    + \tfrac{1}{2}c_{ae}^f\mu^f\Delta t\tnorm{d_t\boldsymbol{u}_h^{f,1}}_{v,f}^2 + F^n.
  \end{equation*}
  This completes the proof of \cref{eq:boundXnterms}. 
\end{proof}

An immediate consequence of \cref{eq:boundX0terms,eq:boundXnterms} is
the following result.

\begin{corollary}
  \label{cor:Xnbound}
  Define $G^0:=F^0$ and $G^n := \tfrac{145}{24}F^0 + F^n$ for
  $1 \le n \le N-1$. Under the assumptions of Lemma \ref{lem:boundXn},
  \begin{equation}
    \label{eq:defGn}
    X^n \le G^n \qquad \forall 0 \le n \le N-1.
  \end{equation}
\end{corollary}

The following lemma is a minor modification of
\cite[Lemma~2]{Cesmelioglu:2022a}.

\begin{lemma}
  \label{lemma:not-gronwall-s}
  Let $\cbr[0]{A_i}_i, \cbr[0]{B_i}_i$, $\cbr[0]{E_i}_i$,
  $\cbr[0]{\overline{E}_i}_i$, $\cbr[0]{\widetilde{E}_i}_i$, and
  $\cbr[0]{D_i}_i$ be nonnegative sequences. Suppose these sequences
  satisfy
  \begin{equation}
    \label{eq:abcd-assumption-modified-s}
    A_n^2 + \sum_{i=1}^n B_i^2
    \le
    \sum_{i=1}^{n} E_i A_i + \sum_{i=1}^{n-1} \overline{E}_i A_i+ \widetilde{E}_n A_n + \sum_{i=1}^n D_i,
  \end{equation}	
  for all $n \ge 1$. Then for any $n \ge 1$,
  \begin{equation}
    \label{eq:An-Bn-estimate-modified-s}
    \del[2]{A_n^2 + \sum_{i=1}^n B_i^2}^{1/2} 
    \leq  \sum_{i=1}^{n} E_i + \sum_{i=1}^{n-1} \overline{E}_i + \max_{1\le i \le n} \widetilde{E}_i + \del[2]{\sum_{i=1}^n D_i}^{1/2} 
  \end{equation}	    
  with $C>0$ independent of $n$.
\end{lemma}
\begin{proof}
  Suppose that 
  \begin{equation}
    \label{eq:An-Bn-max-assumption-modified-s}
    A_n^2 + \sum_{i=1}^n B_i^2
    = \max_{1 \le \ell \le n} \cbr[2]{ A_{\ell}^2 + \sum_{i=1}^{\ell} B_i^2 } .
  \end{equation}
  If $A_n^2 + \sum_{i=1}^n B_i^2 \le \sum_{i=1}^n D_i$, then
  \cref{eq:An-Bn-estimate-modified-s} naturally holds. If
  $A_n^2 + \sum_{i=1}^n B_i^2 > \sum_{i=1}^n D_i$, then
  \cref{eq:abcd-assumption-modified-s} and
  \cref{eq:An-Bn-max-assumption-modified-s} imply
   \begin{align*}
    A_n^2 + \sum_{i=1}^n B_i^2
    &\le
      \del[2]{\sum_{i=1}^{n} E_i + \sum_{i=1}^{n-1} \overline{E}_i  + \widetilde{E}_n}\max_{1\leq i\leq n}A_i
      + \del[2]{\sum_{i=1}^n D_i}^{1/2}\del[2]{\sum_{i=1}^n D_i}^{1/2}
    \\
    &\le
      \del[2]{\sum_{i=1}^{n} E_i + \sum_{i=1}^{n-1} \overline{E}_i  + \widetilde{E}_n
      + \del[2]{\sum_{i=1}^n D_i}^{1/2}} \del[2]{A_n^2 + \sum_{i=1}^n B_i^2}^{1/2}.
  \end{align*}	
  \Cref{eq:An-Bn-estimate-modified-s} follows from dividing
  this by $\del[2]{A_n^2 + \sum_{i=1}^n B_i^2}^{1/2}$.
    
  If \cref{eq:An-Bn-max-assumption-modified-s} is not true, then there
  exists $1 \le n_0 < n$ such that
  \begin{equation*}
    A_{n_0}^2 + \sum_{i=1}^{n_0} B_i^2 = \max_{1 \le \ell \le n} \cbr[2]{ A_{\ell}^2 + \sum_{i=0}^{\ell} B_i^2 }.
  \end{equation*}
  By the same argument as above, we have 
  \begin{equation}
    \label{eq:An-Bn-estimate2-modified-s}
    \del[2]{A_{n_0}^2 + \sum_{i=1}^{n_0} B_i^2}^{1/2} 
    \leq  \sum_{i=1}^{{n_0}} E_i + \sum_{i=0}^{{n_0-1}} \overline{E}_i+ \max_{1\le i \le {n_0}} \widetilde{E}_i
    + \del[2]{\sum_{i=1}^{n_0} D_i}^{1/2}.
  \end{equation}
  Then, \cref{eq:An-Bn-estimate-modified-s} follows by
  \cref{eq:An-Bn-estimate2-modified-s}, $n_0<n$, and the nonnegativities of
  $E_i$, $\overline{E}_i$, $\widetilde{E}_i$, $D_i$.
\end{proof}

For the following lemma we define
\begin{multline}
  \label{eq:defH2}
  H :=
  2 \norm[0]{f^f}_{\ell^1(J;\Omega^f)}
  +  c_{td} (\mu^f/\kappa)^{1/2} \norm[0]{g^b}_{\ell^2(J;\Omega^b)}
  + 2c_p(c_{ae}^b\mu^b)^{-1/2} \norm[0]{d_tf^b}_{\ell^1(J;\Omega^b)}
  \\
  + 2c_p(c_{ae}^b\mu^b)^{-1/2} \norm[0]{f^b}_{\ell^{\infty}(J;\Omega^b)}.
\end{multline}

\begin{lemma}
  \label{lem:uhfk-uhfnp1bounds}
  Let $\boldsymbol{u}_h^{0} = 0$ and $\boldsymbol{p}_h^{p,0} = 0$
  (then also $\boldsymbol{p}_h^{b,0} = \boldsymbol{0}$ and
  $z_h^0 = 0$) and assume $f^{b,0}=0$ and $g^{b,0}=0$. Let
  $1 \le n \le N-1$. Assume
  $(\boldsymbol{u}_h^{k}, \boldsymbol{p}_h^{k}, z_h^{k},
  \boldsymbol{p}_h^{p,k}) \in \boldsymbol{X}_h$ is a solution to
  \cref{eq:hdgfd} for $0 \le k \le n$ such that
  \cref{eq:assumpuhfkfirst} holds. Define for $1\le i \le N$
  \begin{subequations}
    \label{eq:definition-AiBi}
    \begin{align}
      \label{eq:definition-Ai}
      A_i^2
      :=& \tfrac{1}{2}\norm[0]{u_h^{f,i}}_{\Omega^f}^2 
          + \tfrac{1}{2}a_h^b(\boldsymbol{u}_h^{b,i}, \boldsymbol{u}_h^{b,i})
          + \tfrac{1}{2}\lambda^{-1}\norm[0]{\alpha p_h^{p,i} - p_h^{b,i}}_{\Omega^b}^2
          + \tfrac{1}{2}c_0 \norm[0]{p_h^{p,i}}_{\Omega^b}^2,
      \\
      \label{eq:definition-Bi}
      B_i^2
      :=& \tfrac{1}{4} \Delta t c_{ae}^f\mu^f\tnorm{\boldsymbol{u}_h^{f,i}}_{v,f}^2
          + \tfrac{1}{2} \Delta t \gamma \mu^f \kappa^{-1/2} \norm[0]{ (\bar{u}_h^{f,i} - d_t\bar{u}_h^{b,i})^t }_{\Gamma_I}^2
          + \tfrac{1}{2} \Delta t \mu^f\kappa^{-1}\norm[0]{z_h^{i}}_{\Omega^b}^2.        
    \end{align}
  \end{subequations}
  Then for all $1 \le k \le N$,
  \begin{equation}
    \label{eq:boundinguhfm}
    \del[2]{A_k^2 + \sum_{i=1}^k B_i^2}^{1/2} 
    \leq\tfrac{1}{\sqrt{2}}H,
  \end{equation}
  and
  \begin{multline}
    \label{eq:boundinguhfntnorm}
    \tfrac{1}{2}c_{ae}^f\mu^f\tnorm{\boldsymbol{u}_h^{f,k}}_{v,f}^2
    + \gamma \mu^f \kappa^{-1/2} \norm[0]{ (\bar{u}_h^{f,k} - d_t\bar{u}_h^{b,k})^t }_{\Gamma_I}^2
    + \tfrac{1}{2} \mu^f\kappa^{-1}\norm[0]{z_h^{k}}_{\Omega^b}^2
    \\
    \le \norm[0]{f^{f,k}}_{\Omega^f} H + \tfrac{1}{2} c_{td}^2 \mu^f \kappa^{-1} \norm[0]{g^{b,k}}_{\Omega^b}^2
    + c_p(c_{ae}^b \mu^b)^{-1/2} \norm[0]{f^{b,k}}_{\Omega^b} (2G^{k-1})^{1/2} + H(2G^{k-1})^{1/2}.          
  \end{multline}
\end{lemma}
\begin{proof}
  We first prove \cref{eq:boundinguhfm}. Split \cref{eq:hdgfd-b} as
  \begin{subequations}
    \label{eq:splitting-bfdt}
    \begin{align}
      \label{eq:splitting-dt}
      b_h^b(d_t\boldsymbol{u}_h^{b,n+1}, \boldsymbol{q}_h^b)
      + c_h((d_tp_h^{p,n+1}, d_tp_h^{b,n+1}),q_h^b)
      &= 0,
      \\
      b_h^f(\boldsymbol{u}_h^{f,n+1}, \boldsymbol{q}_h^f) &= 0,
    \end{align}
  \end{subequations}
  where we applied $d_t$ to the terms in the Biot region to obtain
  \cref{eq:splitting-dt}. Now, choose
  $\boldsymbol{v}_h^f = \boldsymbol{u}_h^{f,n+1}$,
  $\boldsymbol{v}_h^b = d_t\boldsymbol{u}_h^{b,n+1}$,
  $\boldsymbol{q}_h^f = -\boldsymbol{p}_h^{f,n+1}$,
  $\boldsymbol{q}_h^b = -\boldsymbol{p}_h^{b,n+1}$, $w_h = z_h^{n+1}$,
  $\boldsymbol{q}_h^{p} = \boldsymbol{p}_h^{p,n+1}$ in
  \cref{eq:hdgfd,eq:splitting-bfdt} and add the resulting equations to
  find:
  \begin{equation}
    \label{eq:afteraddingequations}
    \begin{split}
      & (d_tu_h^{f,n+1}, u_h^{f,n+1})_{\Omega^f}
      + t_h(u_h^{f,n}; \boldsymbol{u}_h^{f,n+1}, \boldsymbol{u}_h^{f,n+1})
      + a_h^f(\boldsymbol{u}_h^{f,n+1}, \boldsymbol{u}_h^{f,n+1})
      + a_h^b(\boldsymbol{u}_h^{b,n+1}, d_t\boldsymbol{u}_h^{b,n+1})
      \\
      & \hspace{2em}
      + \gamma \mu^f \kappa^{-1/2} \norm[0]{ (\bar{u}_h^{f,n+1} - d_t\bar{u}_h^{b,n+1})^t }_{\Gamma_I}^2
      + c_h((d_tp_h^{p,n+1}, d_tp_h^{b,n+1}), \alpha p_h^{p,n+1} - p_h^{b,n+1})
      \\
      & \hspace{2em}      
      + (c_0d_t p_h^{p,n+1}, p_h^{p,n+1})_{\Omega^b}
      + \mu^f\kappa^{-1}\norm[0]{z_h^{n+1}}_{\Omega^b}^2
      \\
      &
      = (f^{f,n+1}, u_h^{f,n+1})_{\Omega^f} + (f^{b,n+1}, d_tu_h^{b,n+1})_{\Omega^b} + (g^{b,n+1}, p_h^{p,n+1})_{\Omega^b}.      
    \end{split}
  \end{equation}
  Using the algebraic inequality $a(a-b) \ge (a^2 - b^2)/2$ for the
  discrete time-derivative terms, \cref{eq:coercivityahthresult}
  (which holds by assumption \cref{eq:assumpuhfkfirst} and
  \cref{eq:dtrpoincareineq}), the Cauchy--Schwarz inequality applied
  to the first and third term on the right hand side of
  \cref{eq:afteraddingequations}, and \cref{eq:tnormppbound}, we find,
  using the definitions of $A_i$ and $B_i$:
  \begin{multline}
    \label{eq:Anp1AnBnp1}
    A_{n+1}^2 - A_{n}^2 + 2B_{n+1}^2
    \le \sqrt{2} \Delta t\norm[0]{f^{f,n+1}}_{\Omega^f} A_{n+1}
    + c_{td}(2\Delta t \mu^f \kappa^{-1})^{1/2} \norm[0]{g^{b,n+1}}_{\Omega^b}B_{n+1}
    \\
    + \Delta t (f^{b,n+1}, d_tu_h^{b,n+1})_{\Omega^b}.    
  \end{multline}
  Let us pause to note that, by summation-by-parts, using that
  $u_h^{b,0}=0$, the Cauchy--Schwarz inequality,
  \cref{eq:dpoincareineq-b}, the coercivity result
  \cref{eq:ah-coercive-j}, and the definition of $A_i$, we have
  \begin{align*}
    \Delta t \sum_{i=0}^n (f^{b,i+1}, d_tu_h^{b,i+1})_{\Omega^b}
    =& (f^{b,n+1}, u_h^{b,n+1})_{\Omega^b} - \Delta t \sum_{i=1}^n (d_tf^{b,i+1}, u_h^{b,i})_{\Omega^b}
    \\
    \le& c_p \norm[0]{f^{b,n+1}}_{\Omega^b} \tnorm{\boldsymbol{u}_h^{b,n+1}}_{v,b}
         + c_p \Delta t \sum_{i=1}^n \norm[0]{d_tf^{b,i+1}}_{\Omega^b}\tnorm{\boldsymbol{u}_h^{b,i}}_{v,b}
    \\
    \le& c_p \norm[0]{f^{b,n+1}}_{\Omega^b} (c_{ae}^b\mu^b)^{-1/2}\sqrt{2} A_{n+1}
    \\
     &+ c_p \Delta t \sum_{i=1}^n \norm[0]{d_tf^{b,i+1}}_{\Omega^b} (c_{ae}^b\mu^b)^{-1/2}\sqrt{2}A_i.
  \end{align*}
  Therefore, replacing $n$ by $i$ in \cref{eq:Anp1AnBnp1} and summing
  for $i$ from $0$ to $n$, and Young's inequality,
  \begin{equation}
    \label{eq:Anp1termslerest}
    \begin{split}
      A_{n+1}^2 + \sum_{i=0}^n B_{i+1}^2
      \le& A_{0}^2 + \sqrt{2} \sum_{i=0}^n \Delta t\norm[0]{f^{f,i+1}}_{\Omega^f} A_{i+1}
      + \Delta t \sum_{i=1}^n \norm[0]{d_tf^{b,i+1}}_{\Omega^b} c_p(c_{ae}^b\mu^b)^{-1/2}\sqrt{2}A_i
      \\
      & + \norm[0]{f^{b,n+1}}_{\Omega^b} c_p(c_{ae}^b\mu^b)^{-1/2}\sqrt{2} A_{n+1}            
      + \tfrac{1}{2} \sum_{i=0}^n c_{td}^2 \Delta t \mu^f \kappa^{-1} \norm[0]{g^{b,i+1}}_{\Omega^b}^2.
    \end{split}
  \end{equation}
  Defining
  \begin{align*}
    E_{i} &= \sqrt{2} \Delta t \norm[0]{f^{f,i}}_{\Omega^f},
    &
    \bar{E}_{i} &= \Delta t \norm[0]{d_tf^{b,i+1}}_{\Omega^b} c_p(c_{ae}^b\mu^b)^{-1/2}\sqrt{2},
    \\
    \tilde{E}_i &= \norm[0]{f^{b,i}}_{\Omega^b} c_p(c_{ae}^b\mu^b)^{-1/2}\sqrt{2},
    &
    D_i &= \tfrac{1}{2} c_{td}^2 \Delta t \mu^f \kappa^{-1} \norm[0]{g^{b,i}}_{\Omega^b}^2,
  \end{align*}
  and noting that $A_0=0$ by our assumption on the initial conditions,
  we can write \cref{eq:Anp1termslerest}, for all $n \ge 1$, as
  \begin{equation*}
    A_{n}^2 + \sum_{i=1}^{n} B_{i}^2
    \le
    \sum_{i=1}^{n}E_{i}A_{i} + \sum_{i=1}^{n-1}\bar{E}_iA_i + \tilde{E}_{n}A_{n} + \sum_{i=1}^{n}D_i.
  \end{equation*}
  Therefore, by Lemma \ref{lemma:not-gronwall-s}, for any $n \ge 1$,
  \begin{equation*}
    \begin{split}
      \del[2]{A_n^2 + \sum_{i=1}^n B_i^2}^{1/2} 
      \leq&  
      \sqrt{2} \Delta t \sum_{i=1}^{n} \norm[0]{f^{f,i}}_{\Omega^f}
      +  c_{td} (\tfrac{1}{2} \mu^f/\kappa)^{1/2} \del[2]{ \Delta t \sum_{i=1}^n \norm[0]{g^{b,i}}_{\Omega^b}^2}^{1/2} 
      \\
      &+ \sqrt{2}c_p(c_{ae}^b\mu^b)^{-1/2} \Delta t \sum_{i=1}^{n-1} \norm[0]{d_tf^{b,i+1}}_{\Omega^b} 
      + \sqrt{2}c_p(c_{ae}^b\mu^b)^{-1/2}\max_{1\le i \le n} \norm[0]{f^{b,i}}_{\Omega^b}.
    \end{split}
  \end{equation*}
  By definition of $A_i$ and $B_i$ (see \cref{eq:definition-AiBi}) and
  $H$ (see \cref{eq:defH2}) we conclude \cref{eq:boundinguhfm} (since
  we consider $1 \le n \le N$).
  
  We now proceed with proving \cref{eq:boundinguhfntnorm}. Let us
  start with \cref{eq:afteraddingequations}. For $1 \le n \le N$:
  \begin{equation*}
    \begin{split}
      & (d_tu_h^{f,n}, u_h^{f,n})_{\Omega^f}
      + t_h(u_h^{f,n-1}; \boldsymbol{u}_h^{f,n}, \boldsymbol{u}_h^{f,n})
      + a_h^f(\boldsymbol{u}_h^{f,n}, \boldsymbol{u}_h^{f,n})
      + a_h^b(\boldsymbol{u}_h^{b,n}, d_t\boldsymbol{u}_h^{b,n})
      \\
      & \hspace{2em}
      + \gamma \mu^f \kappa^{-1/2} \norm[0]{ (\bar{u}_h^{f,n} - d_t\bar{u}_h^{b,n})^t }_{\Gamma_I}^2
      + c_h((d_tp_h^{p,n}, d_tp_h^{b,n}), \alpha p_h^{p,n} - p_h^{b,n})
      \\
      & \hspace{2em}      
      + (c_0d_t p_h^{p,n}, p_h^{p,n})_{\Omega^b}
      + \mu^f\kappa^{-1}\norm[0]{z_h^{n}}_{\Omega^b}^2
      = (f^{f,n}, u_h^{f,n})_{\Omega^f} + (f^{b,n}, d_tu_h^{b,n})_{\Omega^b} + (g^{b,n}, p_h^{p,n})_{\Omega^b}.           
    \end{split}
  \end{equation*}
  By assumption \cref{eq:assumpuhfkfirst} we know that
  \begin{equation*}
    t_h(u_h^{f,n-1}; \boldsymbol{u}_h^{f,n}, \boldsymbol{u}_h^{f,n})
    + a_h^f(\boldsymbol{u}_h^{f,n}, \boldsymbol{u}_h^{f,n})
    \ge
    \tfrac{1}{2}c_{ae}^f\mu^f\tnorm{\boldsymbol{u}_h^{f,n}}_{v,f}^2,
  \end{equation*}
  and so,
  \begin{equation}
    \label{eq:almostuhfnp1bound}
    \begin{split}
      & \tfrac{1}{2}c_{ae}^f\mu^f\tnorm{\boldsymbol{u}_h^{f,n}}_{v,f}^2
      + \gamma \mu^f \kappa^{-1/2} \norm[0]{ (\bar{u}_h^{f,n} - d_t\bar{u}_h^{b,n})^t }_{\Gamma_I}^2
      + \mu^f\kappa^{-1}\norm[0]{z_h^{n}}_{\Omega^b}^2
      \\
      &
      \le \bigg| (f^{f,n}, u_h^{f,n})_{\Omega^f} + (f^{b,n}, d_tu_h^{b,n})_{\Omega^b} + (g^{b,n}, p_h^{p,n})_{\Omega^b}
      \\
      & \hspace{2em}      
      - (d_tu_h^{f,n}, u_h^{f,n})_{\Omega^f}
      - a_h^b(\boldsymbol{u}_h^{b,n}, d_t\boldsymbol{u}_h^{b,n})
      - c_h((d_tp_h^{p,n}, d_tp_h^{b,n}), \alpha p_h^{p,n} - p_h^{b,n})
      - (c_0d_t p_h^{p,n}, p_h^{p,n})_{\Omega^b}\bigg|.      
    \end{split}
  \end{equation}
  The Cauchy--Schwarz and Young's inequalities,
  \cref{eq:dpoincareineq-b}, \cref{eq:ah-coercive-j},
  \cref{eq:tnormppbound}, and the definitions of $X^{n}$ \cref{eq:defXn}
  and $A_n$ \cref{eq:definition-Ai} yields
  \begin{subequations}
    \begin{align}
      \label{eq:CS-then-need-dtbounds-a-new}
      |(f^{f,n}, u_h^{f,n})_{\Omega^f}|
      &\le \norm[0]{f^{f,n}}_{\Omega^f} \norm[0]{u_h^{f,n}}_{\Omega^f}
        \le \norm[0]{f^{f,n}}_{\Omega^f} \sqrt{2} A_{n},
      \\
      \label{eq:CS-then-need-dtbounds-b-new}
      |(g^{b,n}, p_h^{p,n})_{\Omega^b}|
      &\le \norm[0]{g^{b,n}}_{\Omega^b} c_{td}\mu^f\kappa^{-1}\norm[0]{z_h^{n}}_{\Omega^b}
        \le \tfrac{1}{2} c_{td}^2 \mu^f\kappa^{-1} \norm[0]{g^{b,n}}_{\Omega^b}^2
        + \tfrac{1}{2}\mu^f\kappa^{-1}\norm[0]{z_h^{n}}_{\Omega^b}^2,
      \\ 
      \label{eq:CS-then-need-dtbounds-c-new}
      |(f^{b,n}, d_tu_h^{b,n})_{\Omega^b}|
      &\le \norm[0]{f^{b,n}}_{\Omega^b} \norm[0]{d_tu_h^{b,n}}_{\Omega^b}
      \le c_p \norm[0]{f^{b,n}}_{\Omega^b} \tnorm{d_t\boldsymbol{u}_h^{b,n}}_{v,b}
      \\ \nonumber
      &\le c_p (c_{ae}^b\mu^b)^{-1/2} \norm[0]{f^{b,n}}_{\Omega^b} a_h^b\del[0]{d_t\boldsymbol{u}_h^{b,n},d_t\boldsymbol{u}_h^{b,n} }^{1/2}
      \\ \nonumber
      &\le c_p (c_{ae}^b\mu^b)^{-1/2} \norm[0]{f^{b,n}}_{\Omega^b} (2X^{n-1})^{1/2}
      \\
      \label{eq:CS-then-need-dtbounds-d-new}
      |(d_tu_h^{f,n}, u_h^{f,n})_{\Omega^f}|
      &\le \norm[0]{d_tu_h^{f,n}}_{\Omega^f} \norm[0]{u_h^{f,n}}_{\Omega^f},
      \\
      \label{eq:CS-then-need-dtbounds-e-new}
      |a_h^b(\boldsymbol{u}_h^{b,n}, d_t\boldsymbol{u}_h^{b,n})|
      &\le a_h^b(\boldsymbol{u}_h^{b,n}, \boldsymbol{u}_h^{b,n})^{1/2} a_h^b(d_t\boldsymbol{u}_h^{b,n}, d_t\boldsymbol{u}_h^{b,n})^{1/2},
      \\
      \label{eq:CS-then-need-dtbounds-f-new}
      |c_h((d_tp_h^{p,n}, d_tp_h^{b,n}), \alpha p_h^{p,n} - p_h^{b,n})|
      &\le \lambda^{-1/2}\norm[0]{d_t(\alpha p_h^{p,n} - p_h^{b,n})}_{\Omega^b}\lambda^{-1/2}\norm[0]{\alpha p_h^{p,n} - p_h^{b,n}}_{\Omega^b},
      \\
      \label{eq:CS-then-need-dtbounds-g-new}
      |(c_0d_t p_h^{p,n}, p_h^{p,n})_{\Omega^b}|
      &\le c_0^{1/2} \norm[0]{p_h^{p,n}}_{\Omega^b} c_0^{1/2} \norm[0]{d_t p_h^{p,n}}_{\Omega^b}.
    \end{align}    
  \end{subequations}
  Furthermore, combining
  \cref{eq:CS-then-need-dtbounds-d-new,eq:CS-then-need-dtbounds-e-new,eq:CS-then-need-dtbounds-f-new,eq:CS-then-need-dtbounds-g-new}
  and using the definitions of $A_n$ and $X^n$, we find
  \begin{multline*}
      |(d_tu_h^{f,n}, u_h^{f,n})_{\Omega^f}| + |a_h^b(\boldsymbol{u}_h^{b,n}, d_t\boldsymbol{u}_h^{b,n})| 
      + |c_h((d_tp_h^{p,n}, d_tp_h^{b,n}), \alpha p_h^{p,n} - p_h^{b,n})|
      + |(c_0d_t p_h^{p,n}, p_h^{p,n})_{\Omega^b}|
      \\
      \le 2 A_{n} (X^{n-1})^{1/2}.    
  \end{multline*}
  Combining the above inequality,
  \cref{eq:CS-then-need-dtbounds-a-new,eq:CS-then-need-dtbounds-b-new,eq:CS-then-need-dtbounds-c-new}
  with \cref{eq:almostuhfnp1bound}, and using
  \cref{eq:boundinguhfm} and \cref{eq:defGn},
  \begin{align*}
    &\tfrac{1}{2}c_{ae}^f\mu^f\tnorm{\boldsymbol{u}_h^{f,n}}_{v,f}^2
      + \gamma \mu^f \kappa^{-1/2} \norm[0]{ (\bar{u}_h^{f,n} - d_t\bar{u}_h^{b,n})^t }_{\Gamma_I}^2
      + \tfrac{1}{2} \mu^f\kappa^{-1}\norm[0]{z_h^{n}}_{\Omega^b}^2
    \\
    &\le \norm[0]{f^{f,n}}_{\Omega^f} \sqrt{2} A_{n} + \tfrac{1}{2} c_{td}^2 \mu^f \kappa^{-1} \norm[0]{g^{b,n}}_{\Omega^b}^2
      + c_p(c_{ae}^b \mu^b)^{-1/2} \norm[0]{f^{b,n}}_{\Omega^b} (2X^{n-1})^{1/2} + 2A_{n}(X^{n-1})^{1/2}
    \\
    &\le \norm[0]{f^{f,n}}_{\Omega^f} H + \tfrac{1}{2} c_{td}^2 \mu^f \kappa^{-1} \norm[0]{g^{b,n}}_{\Omega^b}^2
      + c_p(c_{ae}^b \mu^b)^{-1/2} \norm[0]{f^{b,n}}_{\Omega^b} (2G^{n-1})^{1/2} + H(2G^{n-1})^{1/2}.
  \end{align*}
  We therefore conclude \cref{eq:boundinguhfntnorm}.
\end{proof}

We arrive at the main result of this section.

\begin{theorem}[Well-posedness]
  \label{thm:well-posedness}
  Assume the data satisfies
  \begin{multline}
    \label{eq:dataassump}
    \max\del[2]{H^2,\norm[0]{f^f}_{\ell^{\infty}(J;\Omega^f)} H
    + \tfrac{1}{2} c_{td}^2 \mu^f \kappa^{-1} \norm[0]{g^b}_{\ell^{\infty}(J;\Omega^b)}^2
    + c_p(c_{ae}^b \mu^b)^{-1/2} \norm[0]{f^b}_{\ell^{\infty}(J;\Omega^b)} (2G^{N-1})^{1/2}
    + H(2G^{N-1})^{1/2}}
    \\
    \le \tfrac{1}{4}(\mu^fc_{ae}^f)^2\min\del[1]{ \frac{1}{c_{si,2}(c_{pq}^2 + c_{si,4}^2)}, \frac{1}{2c_w} },    
  \end{multline}
  and that $f^{b,0}=0$ and $g^{b,0}=0$. Then, starting with
  $\boldsymbol{u}_h^{0} = 0$ and $\boldsymbol{p}_h^{p,0} = 0$,
  \cref{eq:hdgfd} defines a unique solution. We furthermore have the
  following uniform bound (in $n$ and $h$) for $1 \le n \le N$:
  \begin{multline}
    \label{eq:bound-vel}
    \tfrac{1}{2}c_{ae}^f\mu^f\tnorm{\boldsymbol{u}_h^{f,n}}_{v,f}^2
    + \gamma \mu^f \kappa^{-1/2} \norm[0]{ (\bar{u}_h^{f,n} - d_t\bar{u}_h^{b,n})^t }_{\Gamma_I}^2
    + \tfrac{1}{2}\mu^f\kappa^{-1}\norm[0]{z_h^{n}}_{\Omega^b}^2
    \\
    \le
    \tfrac{1}{4}(\mu^fc_{ae}^f)^2\min\del[1]{ \frac{1}{c_{si,2}(c_{pq}^2 + c_{si,4}^2)}, \frac{1}{2c_w} }.
  \end{multline}
  Furthermore, for $1 \le n \le N$:
  \begin{equation}
    \label{eq:bound-ubpp}
    c_{ae}^b\mu^b\tnorm{\boldsymbol{u}_h^{b,n}}_{v,b}^2
    + \lambda^{-1} \norm[0]{\alpha p_h^{p,n} - p_h^{b,n}}_{\Omega^b}^2
    + c_0 \norm[0]{p_h^{p,n}}_{\Omega^b}^2          
    \le \tfrac{1}{4}(\mu^fc_{ae}^f)^2\min\del[1]{ \frac{1}{c_{si,2}(c_{pq}^2 + c_{si,4}^2)}, \frac{1}{2c_w} }.    
  \end{equation}
  The pressures are bounded as:
  \begin{subequations}
    \begin{align}
      \label{eq:ppnbound}
      \tnorm{\boldsymbol{p}_h^{p,n}}_{1,h,b}^2
      \le&
           \tfrac{1}{2}c_{pd}^2\mu^f (c_{ae}^f)^2\kappa^{-1}\min\del[1]{ \frac{1}{c_{si,2}(c_{pq}^2 + c_{si,4}^2)}, \frac{1}{2c_w} },
      \\
      \label{eq:pfbnbound}
      c_b\tnorm{\boldsymbol{p}_h^{n+1}}_q
      \le&
           c_p\norm[0]{f^{f}}_{\ell^{\infty}(J;L^2(\Omega^f))} + c_p\norm[0]{f^{b}}_{\ell^{\infty}(J;L^2(\Omega^b))}
           + c_p(G^{N-1})^{1/2}
      \\ \nonumber
         &+ \tfrac{1}{2}\mu^f c_{ae}^f c_w \min\del[1]{ \frac{1}{c_{si,2}(c_{pq}^2 + c_{si,4}^2)}, \frac{1}{2c_w} }
      \\ \nonumber
         &+ \del[3]{c_{ac}^f\mu^f + c_{ac}^b(c_{ae}^f/c_{ae}^b)^{1/2}(\mu^f/\mu^b)^{1/2}
           + (\tfrac{1}{2}c_{ae}^f\gamma\kappa^{-1/2})^{1/2}\mu^f} \times
      \\ \nonumber
         & \hspace{15em} \sbr[2]{\min\del[1]{ \frac{\mu^f c_{ae}^f}{2c_{si,2}(c_{pq}^2 + c_{si,4}^2)}, \frac{c_{ae}^f\mu^f}{4c_w} }}^{1/2}.
    \end{align}
  \end{subequations}
\end{theorem}
\begin{proof}
  The bound \cref{eq:bound-vel} follows directly from
  \cref{eq:boundinguhfntnorm} and \cref{eq:dataassump}. To prove
  uniqueness we immediately note that \cref{eq:bound-vel} and
  \cref{eq:dtrpoincareineq} result in
  \begin{equation*}
    \norm[0]{u_h^{f,k}\cdot n}_{\Gamma_{IN}^f}
    \le c_{si,2} \tnorm{\boldsymbol{u}_h^{f,k}}_{v,f}
    \le \mu^f c_{ae}^f / (2(c_{pq}^2 + c_{si,4}^2)).
  \end{equation*}
  Uniqueness therefore follows from Lemma \ref{lem:uniqueness}. The
  bound \cref{eq:bound-ubpp} follows from \cref{eq:boundinguhfm},
  \cref{eq:ah-coercive-j}, and \cref{eq:dataassump}. The bound
  \cref{eq:ppnbound} is a direct consequence of \cref{eq:tnormppbound}
  and \cref{eq:bound-vel}. Finally, we consider the pressure bound
  \cref{eq:pfbnbound}. From \cref{eq:inf-sup-b} and \cref{eq:hdgfd} we
  obtain:
  \begin{equation}
    \label{eq:usinginfsuppq}
    c_b\tnorm{\boldsymbol{p}_h^{n+1}}_q
    \le 
    \sup_{\boldsymbol{0}\neq \boldsymbol{v}_h\in \widehat{\boldsymbol{V}}_h}
          \frac{|b_h(\boldsymbol{v}_h, \boldsymbol{p}_h^{n+1})|}{\tnorm{\boldsymbol{v}_h}_v}
    \le
    \sup_{\boldsymbol{0}\neq \boldsymbol{v}_h\in \widehat{\boldsymbol{V}}_h}
          \frac{B}{\tnorm{\boldsymbol{v}_h}_v}
  \end{equation}
  where
  \begin{multline*}
    B = |(f^{n+1}, v_h)_{\Omega}| + |(d_tu_h^{n+1}, v_h)_{\Omega^f}|
    + |t_h(u_h^{f,n}; \boldsymbol{u}_h^{f,n+1}, \boldsymbol{v}_h^f)|
    + |a_h(\boldsymbol{u}_h^{n+1}, \boldsymbol{v}_h)|
    \\
    + |a_h^I((\bar{u}_h^{f,n+1},d_t\bar{u}_h^{b,n+1}),(\bar{v}_h^f,\bar{v}_h^b))|
    + |b_h^I((\bar{v}_h^f,\bar{v}_h^b), \bar{p}_h^{p,n+1})|.
  \end{multline*}
  First, note that
  $b_h^I((\bar{v}_h^f,\bar{v}_h^b), \bar{p}_h^{p,n+1})=0$ by the
  definition of $\widehat{\boldsymbol{V}}_h$. Using the
  Cauchy--Schwarz inequality, \cref{eq:dpoincareineq},
  \cref{eq:dpoincareineq-b}, \cref{eq:upperboundthdif},
  \cref{eq:ah-continuity-j}:
  \begin{align*}
    |(f^{n+1}, v_h)_{\Omega}|
    &\le c_p\norm[0]{f^{f,n+1}}_{\Omega^f}\tnorm{\boldsymbol{v}_h^f}_{v,f} + c_p\norm[0]{f^{b,n+1}}_{\Omega^b}\tnorm{\boldsymbol{v}_h^b}_{v,b},
    \\
    |(d_tu_h^{n+1}, v_h)_{\Omega^f}|
    & \le c_p\norm[0]{d_tu_h^{n+1}}_{\Omega^f}\tnorm{\boldsymbol{v}_h^f}_{v,f},
    \\
    |t_h(u_h^{f,n}; \boldsymbol{u}_h^{f,n+1}, \boldsymbol{v}_h^f)|
    &\le c_w \tnorm{\boldsymbol{u}_h^{f,n}}_{v,f} \tnorm{\boldsymbol{u}_h^{f,n+1}}_{v,f}\tnorm{\boldsymbol{v}_h^f}_{v,f},
    \\
    |a_h(\boldsymbol{u}_h^{n+1}, \boldsymbol{v}_h)|
    &\le c_{ac}^f\mu^f\tnorm{\boldsymbol{u}_h^{f,n+1}}_{v,f}\tnorm{\boldsymbol{v}_h^f}_{v,f}
      + c_{ac}^b\mu^b\tnorm{\boldsymbol{u}_h^{b,n+1}}_{v,b}\tnorm{\boldsymbol{v}_h^b}_{v,b},
    \\
    |a_h^I((\bar{u}_h^{f,n+1},d_t\bar{u}_h^{b,n+1}),(\bar{v}_h^f,\bar{v}_h^b))|
    &\le \gamma\mu^f\kappa^{-1/2}\norm[0]{(\bar{u}_h^{f,n+1}-d_t\bar{u}_h^{b,n+1})^t}_{\Gamma_I}
      \tnorm{\boldsymbol{v}_h}_{v}.
  \end{align*}
  Combined now with \cref{eq:usinginfsuppq}, Corollary
  \ref{cor:Xnbound}, \cref{eq:bound-vel}, and \cref{eq:bound-ubpp} we
  obtain the result \cref{eq:pfbnbound}.
\end{proof}

\section{Error analysis}
\label{sec:erroranalysis}

For the error analysis we first define interpolation operators to
decompose the errors. For scalar functions, we denote by $\Pi_Q^j$,
$\bar{\Pi}_Q^j$, and $\bar{\Pi}_{Q^{0}}^b$ the $L^2$-projection
operators onto $Q_h^j$, $\bar{Q}_h^j$, and $\bar{Q}_h^{b0}$,
respectively. For vector valued functions, we define
$\Pi_V^j:H(\text{div},\Omega^j) \cap [L^r(\Omega^j)]^d \to V_h^j$, for $r > 2$ and
$j=f,b$, to be the interpolation operator to the
Brezzi--Douglas--Marini (BDM) finite element spaces \cite[Section
III.3]{Brezzi:book} and
$\bar{\Pi}_V^j:[L^2(\mathcal{F}^j)]^d \to \bar{V}_h^j$ to be the
$L^2$-projection onto $\bar{V}_h^j$. It is known that
\begin{equation*}
  (q_h, \nabla \cdot \Pi_V^j u^j)_K = (q_h, \nabla \cdot u^j)_K \quad
  \forall q_h \in P_{k-1}(K), K \in \mathcal{T}^j,
  \quad
  \langle \bar{q}_h, n \cdot \Pi_V^j u^j\rangle_F = \langle \bar{q}_h, n \cdot u^j\rangle_F \quad
  \forall \bar{q}_h \in P_{k}(F), F \in \mathcal{F}^j,
\end{equation*}
and
\begin{align*}
  &\text{if } u^j|_K \in [H^{k+1}(K)]^d:
  && \norm[0]{u^j - \Pi_V^j u^j}_{m,K} \le Ch_K^{l-m} \norm[0]{u^j}_{l,K}, \quad m=0,1,2, \quad \max\{1, m\} \le l \le k+1,
  \\
  &\text{if } u^j|_{K} \in [W_{\infty}^1(K)]^d:
  && \norm[0]{u^j - \Pi_V^j u^j}_{L^{\infty}(K)} \le Ch_K |u|_{W_{\infty}^1(K)}.
\end{align*}
For time stepping index $n$, we
introduce the following notation for the errors:
\begin{subequations}
  \label{eq:errorsplitting}
  \begin{align}
    u^{j,n} - u_h^{j,n} &= (u^{j,n} - \Pi_V^j u^{j,n}) - (u_h^{j,n} - \Pi_V^j u^{j,n}) 
    =: e_{u^j}^{I,n} - e_{u^j}^{h,n},
    \\
    \bar{u}^{j,n} - \bar{u}_h^{j,n} &= (\bar{u}^{j,n} - \bar{\Pi}_V^{j}\bar{u}^{j,n}) - (\bar{u}_h^{j,n} - \bar{\Pi}_V^{j}\bar{u}^{j,n}) 
    =: \bar{e}_{u^j}^{I,n} - \bar{e}_{u^j}^{h,n},
    \\
    z^n - z_h^n &= (z^n - \Pi_V^b z^n) - (z_h^n - \Pi_V^b z^n) 
    =: e_z^{I,n} - e_z^{h,n},
    \\
    p^{j,n} - p_h^{j,n} &= (p^{j,n} - \Pi_Q^jp^{j,n}) - (p_h^{j,n} - \Pi_Q^jp^{j,n}) 
    =: e_{p^j}^{I,n} - e_{p^j}^{h,n},
    \\
    \bar{p}^{j,n} - \bar{p}_h^{j,n} &= (\bar{p}^{j,n} - \bar{\Pi}_Q^j\bar{p}^{j,n}) - (\bar{p}_h^{j,n} - \bar{\Pi}_Q^jp^{j,n}) 
    =: \bar{e}_{p^j}^{I,n} - \bar{e}_{p^j}^{h,n},
    \\
    p^{p,n} - p_h^{p,n} &= (p^{p,n} - \Pi_Q^bp^{p,n}) - (p_h^{p,n} - \Pi_Q^bp^{p,n}) 
    =: e_{p^p}^{I,n} - e_{p^p}^{h,n},
    \\
    \bar{p}^{p,n} - \bar{p}_h^{p,n} &= (\bar{p}^{p,n} - \bar{\Pi}_{Q^0}^b\bar{p}^{p,n}) - (\bar{p}_h^{p,n} - \bar{\Pi}_{Q^0}^b\bar{p}^{p,n}) 
    =: \bar{e}_{p^p}^{I,n} - \bar{e}_{p^p}^{h,n},
  \end{align}  
\end{subequations}
and define $e_u^{\omega,n}$ and $e_p^{\omega,n}$ such that
$e_u^{\omega,n}|_{\Omega^j} = e_{u^j}^{\omega,n}$ and
$e_p^{\omega,n}|_{\Omega^j} = e_{p^j}^{\omega,n}$, for $\omega = I, h$.

The following lemma now determines the error equations.

\begin{lemma}[Error equations]
  \label{lem:errorequation}
  Suppose that
  $\cbr[0]{(\boldsymbol{u}_h^n, \boldsymbol{p}_h^n, z_h^{n},
    \boldsymbol{p}_h^{p,n})}_n$ are the solutions to
  \cref{eq:hdgfd}. Furthermore, assume that $(u,p,z,p^p)$ is the
  solution to
  \cref{eq:navierstokes,eq:biot,eq:interface,eq:bcs,eq:ics}, with
  $u_0^j(x) = 0$ and $p_0^p(x)=0$, on the time interval
  $J=(0,T]$. Define
  $\boldsymbol{u}:=(u,u|_{\Gamma_0^f},u|_{\Gamma_0^b})$,
  $\boldsymbol{p} := (p,p|_{\Gamma_0^f},p|_{\Gamma_0^b})$, and
  $\boldsymbol{p}^p := (p^p,p^p|_{\Gamma_0^b})$. Then, for $n \ge 0$
  and for all
  $(\boldsymbol{v}_h, \boldsymbol{q}_h, w_h, \boldsymbol{q}_h^p)\in
  \boldsymbol{X}_h$ we have:
  \begin{subequations}
    \label{eq:reduced-error-eqs}
    \begin{align}
      \label{eq:reduced-error-eq1}
      & (d_t e_{u^f}^{h,n+1}, v_h^f)_{\Omega^f}
        + t_h(u_h^{f,n}; \boldsymbol{e}_{u^{f}}^{h,n+1}, \boldsymbol{v}_h^f)
        + a_h(\boldsymbol{e}_{u}^{h,n+1}, \boldsymbol{v}_h)+ b_h(\boldsymbol{v}_h, \boldsymbol{e}_{p}^{h,n+1})
      \\ \nonumber
      & \hspace{2em}
        + a_h^I((\bar{e}_{u^f}^{h,n+1},d_t\bar{e}_{u^b}^{h,n+1}),(\bar{v}_h^f,\bar{v}_h^b))      + b_h^I((\bar{v}_h^f,\bar{v}_h^b), \bar{e}_{p^p}^{h,n+1})
      \\ \nonumber
      & \hspace{2em}
        = (\partial_t u^{f,n+1}-(\Delta t)^{-1} (\Pi_V^{f} u^{f,n+1} - \Pi_V^{f} u^{f,n}), v_h^f)_{\Omega^f} + a_h(\boldsymbol{e}_{u}^{I,n+1}, \boldsymbol{v}_h)
      \\ \nonumber
      & \hspace{4em}
        + a_h^I((0, \partial_t \bar{u}^{b,n+1} - d_t \bar{u}^{b,n+1} ),(\bar{v}_h^f,\bar{v}_h^b))  
        + \sbr[1]{t_h(u^{f,n+1}; \boldsymbol{u}^{f,n+1}, \boldsymbol{v}_h^f) - t_h(u^{f,n}; \boldsymbol{u}^{f,n+1}, \boldsymbol{v}_h^f)}
      \\ \nonumber
      & \hspace{4em}
        + t_h(u^{f,n}; \boldsymbol{e}_{u^f}^{I,n+1}, \boldsymbol{v}_h^f)
        + \sbr[1]{ t_h(u^{f,n}; \boldsymbol{\Pi}_{V}^{f} \boldsymbol{u}^{f,n+1}, \boldsymbol{v}_h^f)
        - t_h(u_h^{f,n}; \boldsymbol{\Pi}_{V}^{f} \boldsymbol{u}^{f,n+1}, \boldsymbol{v}_h^f) },
      \\     
      \label{eq:reduced-error-eq2}
      & b_h^f(\boldsymbol{e}_{u^f}^{h,n+1}, \boldsymbol{q}_h^f) + b_h^b(d_t\boldsymbol{e}_{u^b}^{h,n+1}, \boldsymbol{q}_h^b)
        + c_h((d_t e_{p^p}^{h,n+1},d_t e_{p^b}^{h,n+1}),q_h^b) = 0,
      \\
      \label{eq:reduced-error-eq3}
      &(c_0d_t e_{p^p}^{h,n+1}, q_h^p)_{\Omega^b}
        + c_h((d_te_{p^p}^{h,n+1},d_te_{p^b}^{h,n+1}), \alpha q_h^p)     - b_h^{b}((e_z^{h,n+1}, 0), \boldsymbol{q}_h^p)      
        - b_h^I((\bar{e}_{u^f}^{h,n+1}, d_t\bar{e}_{u^b}^{h,n+1}), \bar{q}_h^p)
      \\ \nonumber
      & \hspace{2em}
        =(c_0(\partial_tp^{p,n+1} - d_t p^{p,n+1} ), q_h^p)_{\Omega^b}
        + c_h((\partial_tp^{p,n+1}-d_t p^{p,n+1},\partial_tp^{b,n+1}-d_t p^{b,n+1}), \alpha q_h^p)
      \\ \nonumber
      & \hspace{4em}
        - b_h^I((0, \partial_t \bar{u}^{b,n+1} - d_t \bar{u}^{b,n+1}), \bar{q}_h^p),
      \\
      \label{eq:reduced-error-eq4}
      &(\mu^f\kappa^{-1} e_z^{h,n+1}, w_h)_{\Omega^b} + b_h^{b}((w_h, 0), \boldsymbol{e}_{p^p}^{h,n+1})
        =(\mu^f\kappa^{-1} e_z^{I,n+1}, w_h)_{\Omega^b}.
    \end{align}
  \end{subequations}
\end{lemma}
\begin{proof}
  It can easily be shown, by standard arguments, that the
  semi-discrete HDG method \cref{eq:hdgsd} is consistent. Therefore,
  substituting the exact solution at time level $t=t_{n+1}$ into
  \cref{eq:hdgsd} and subtracting \cref{eq:hdgfd} from the result, we
  obtain:
  \begin{subequations}
    \label{eq:uminuhdif}
    \begin{align}
      & (\partial_t u^{f,n+1} - d_tu_h^{f,n+1}, v_h^f)_{\Omega^f}
      + t_h(u^{f,n+1}; \boldsymbol{u}^{f,n+1}, \boldsymbol{v}_h^f)
      - t_h(u_h^{f,n}; \boldsymbol{u}_h^{f,n+1}, \boldsymbol{v}_h^f)
      \\ \nonumber
      & \hspace{2em} 
      + a_h(\boldsymbol{u}^{n+1} - \boldsymbol{u}_h^{n+1}, \boldsymbol{v}_h)
      + b_h(\boldsymbol{v}_h, \boldsymbol{p}^{n+1} - \boldsymbol{p}_h^{n+1})
      \\ \nonumber
      & \hspace{2em} 
      + a_h^I((\bar{u}^{f,n+1} - \bar{u}_h^{f,n+1}, \partial_t \bar{u}^{b,n+1} - d_t\bar{u}_h^{b,n+1}),(\bar{v}_h^f,\bar{v}_h^b))      
      + b_h^I((\bar{v}_h^f,\bar{v}_h^b), \bar{p}^{p,n+1} - \bar{p}_h^{p,n+1})
      = 0,
      \\
      &b_h(\boldsymbol{u}^{n+1} - \boldsymbol{u}_h^{n+1}, \boldsymbol{q}_h)
      + c_h((p^{p,n+1} - p_h^{p,n+1},p^{b,n+1} - p_h^{b,n+1}),q_h^b)
      = 0,
      \\
      &(c_0 (\partial_t p^{p,n+1} - d_t p_h^{p,n+1}), q_h^p)_{\Omega^b}
      + c_h((\partial_t p^{p,n+1} - d_tp_h^{p,n+1}, \partial_t p^{b,n+1} - d_tp_h^{b,n+1}), \alpha q_h^p)
      \\ \nonumber
      & \hspace{2em}
      - b_h^{b}((z^{n+1} - z_h^{n+1}, 0), \boldsymbol{q}_h^p)      
      - b_h^I((\bar{u}^{f,n+1} - \bar{u}_h^{f,n+1}, \partial_t \bar{u}^{b,n+1} - d_t\bar{u}_h^{b,n+1}), \bar{q}_h^p)
      =0,
      \\
      &\mu^f\kappa^{-1} (z^{n+1} - z_h^{n+1}, w_h)_{\Omega^b} + b_h^{b}((w_h, 0), \boldsymbol{p}^{p,n+1} - \boldsymbol{p}_h^{p,n+1})=0.
    \end{align}    
  \end{subequations}
  Noting that
  $\partial_t u^{f,n+1} - d_tu_h^{f,n+1} = - d_te_{u^f}^{h,n+1} +
  \sbr[1]{\partial_t u^{f,n+1} - (\Delta
    t)^{-1}(\Pi_V^fu^{f,n+1}-\Pi_V^fu^{f,n})}$, and similar for the
  other time derivative terms, and using the error decomposition
  \cref{eq:errorsplitting}, we can write \cref{eq:uminuhdif} as
  \begin{subequations}
    \label{eq:error-eqs-1}
    \begin{align}
      \label{eq:error-eq1-1}
      & (d_t e_{u^f}^{h,n+1}, v_h^f)_{\Omega^f}
        - t_h(u^{f,n+1}; \boldsymbol{u}^{f,n+1}, \boldsymbol{v}_h^f)
        + t_h(u_h^{f,n}; \boldsymbol{u}_h^{f,n+1}, \boldsymbol{v}_h^f)
      \\ \nonumber
      & \hspace{2em}
        + a_h(\boldsymbol{e}_{u}^{h,n+1}, \boldsymbol{v}_h)
        + b_h(\boldsymbol{v}_h, \boldsymbol{e}_{p}^{h,n+1})
        + a_h^I((\bar{e}_{u^f}^{h,n+1},d_t\bar{e}_{u^b}^{h,n+1}),(\bar{v}_h^f,\bar{v}_h^b))      
        + b_h^I((\bar{v}_h^f,\bar{v}_h^b), \bar{e}_{p^p}^{h,n+1})
      \\ \nonumber
      & \hspace{2em}
        = (\partial_t u^{f,n+1} - (\Delta t)^{-1} (\Pi_V^{f} u^{f,n+1} - \Pi_V^{f} u^{f,n}), v_h^f)_{\Omega^f}
        + a_h(\boldsymbol{e}_{u}^{I,n+1}, \boldsymbol{v}_h)
        + b_h(\boldsymbol{v}_h, \boldsymbol{e}_{p}^{I,n+1})
      \\ \nonumber
      & \hspace{4em}
        + a_h^I((\bar{e}_{u^f}^{I,n+1}, \partial_t \bar{u}^{b,n+1} - (\Delta t)^{-1} (\bar{\Pi}_V^{b} \bar{u}^{b,n+1} - \bar{\Pi}_V^{b} \bar{u}^{b,n} ) ),(\bar{v}_h^f,\bar{v}_h^b))      
        + b_h^I((\bar{v}_h^f,\bar{v}_h^b), \bar{e}_{p^p}^{I,n+1}) ,
      \\
      \label{eq:error-eq2-1}
      & b_h(\boldsymbol{e}_u^{h,n+1}, \boldsymbol{q}_h)
        + c_h((e_{p^p}^{h,n+1},e_{p^b}^{h,n+1}),q_h^b)
        =  b_h(\boldsymbol{e}_u^{I,n+1}, \boldsymbol{q}_h)
        + c_h((e_{p^p}^{I,n+1},e_{p^b}^{I,n+1}),q_h^b),
      \\
      \label{eq:error-eq3-1}
      &(c_0d_t e_{p^p}^{h,n+1}, q_h^p)_{\Omega^b}
        + c_h((d_te_{p^p}^{h,n+1},d_te_{p^b}^{h,n+1}), \alpha q_h^p)
        - b_h^{b}((e_z^{h,n+1}, 0), \boldsymbol{q}_h^p)      
      \\ \nonumber
      & \hspace{2em}
        - b_h^I((\bar{e}_{u^f}^{h,n+1}, d_t\bar{e}_{u^b}^{h,n+1}), \bar{q}_h^p)
      \\ \nonumber
      & \hspace{2em}
        =(c_0\sbr[0]{\partial_t p^{p,n+1} - (\Delta t)^{-1} (\Pi_Q^bp^{p,n+1} - \Pi_Q^bp^{p,n}) }, q_h^p)_{\Omega^b}
      \\ \nonumber
      & \hspace{4em}
        + c_h((\partial_t p^{p,n+1}- (\Delta t)^{-1} (\Pi_Q^{b} p^{p,n+1} - \Pi_Q^{b} p^{p,n})),
        \partial_t p^{b,n+1} - (\Delta t)^{-1} (\Pi_Q^{b} p^{b,n+1} - \Pi_Q^{b} p^{b,n})), \alpha q_h^p)
      \\ \nonumber
      & \hspace{4em} 
        - b_h^{b}((e_z^{I,n+1}, 0), \boldsymbol{q}_h^p)      
        - b_h^I((\bar{e}_{u^f}^{I,n+1},  \partial_t \bar{u}^{b,n+1} - (\Delta t)^{-1} (\bar{\Pi}_V^{b} \bar{u}^{b,n+1} - \bar{\Pi}_V^{b} \bar{u}^{b,n} )), \bar{q}_h^p),
      \\
      \label{eq:error-eq4-1}
      &(\mu^f\kappa^{-1} e_z^{h,n+1}, w_h)_{\Omega^b} + b_h^{b}((w_h, 0), \boldsymbol{e}_{p^p}^{h,n+1})
        =(\mu^f\kappa^{-1} e_z^{I,n+1}, w_h)_{\Omega^b} + b_h^{b}((w_h, 0), \boldsymbol{e}_{p^p}^{I,n+1}) .
    \end{align}
  \end{subequations}
  By properties of $\Pi_V^j$ and $\bar{\Pi}_V^j$, $j=f,b$, we have
  that $b_h(\boldsymbol{e}_{u}^{I,n+1}, \boldsymbol{q}_h) = 0$ for all
  $\boldsymbol{q}_h \in \boldsymbol{Q}_h$ and
  $b_h^b((e_z^{I,n+1},0), \boldsymbol{q}_h^p) = 0$ for all
  $\boldsymbol{q}_h^p \in \boldsymbol{Q}_h^{b0}$. Similarly,
  $c_h((e_{p^p}^{I,n+1}, e_{p^b}^{I,n+1}), q_h^b) = 0$ for all
  $q_h^b \in Q_h^b$ because $\Pi_{Q}^b$ is the $L^2$-projection onto
  $Q_h^b$, and $b_h(\boldsymbol{v}_h, \boldsymbol{e}_p^{I,n+1}) = 0$
  for all $\boldsymbol{v}_h \in \boldsymbol{V}_h$,
  $b_h^b((w_h, 0), \boldsymbol{e}_{p^p}^{I,n+1}) = 0$ for all
  $w_h \in V_h^b$, and
  $b_h^I((\bar{v}_h^f, \bar{v}_h^b), \bar{e}_{p^p}^{I,n+1}) = 0$ for
  all $(\bar{v}_h^f, \bar{v}_h^b) \in \bar{V}_h^f \times \bar{V}_h^b$
  because $\Pi_Q^j$, $\bar{\Pi}_Q^j$, $j=f,b$, and $\bar{\Pi}_{Q^0}^b$
  are $L^2$-projections. Furthermore, since $\bar{\Pi}_V^j$,
  $\bar{\Pi}_Q^j$, $j=f,b$, $\bar{\Pi}_{Q^0}^b$ are $L^2$-projections,
  \begin{equation*}
    \begin{split}
      a_h^I((\bar{e}_{u^f}^{I,n+1}, (\Delta t)^{-1} (\bar{\Pi}_V^{b} \bar{u}^{b,n+1} - \bar{\Pi}_V^{b} \bar{u}^{b,n} )
      - \partial_t \bar{u}^{b,n+1} ),(\bar{v}_h^f,\bar{v}_h^b)) 
      &= a_h^I((0, d_t \bar{u}^{b,n+1}- \partial_t \bar{u}^{b,n+1} ),(\bar{v}_h^f,\bar{v}_h^b)),
      \\
      b_h^I((\bar{e}_{u^f}^{I,n+1}, (\Delta t)^{-1} (\bar{\Pi}_V^{b} \bar{u}^{b,n+1} - \bar{\Pi}_V^{b} \bar{u}^{b,n} )
      - \partial_t \bar{u}^{b,n+1}), \bar{q}_h^p)    
      &= b_h^I((0, d_t \bar{u}^{b,n+1} - \partial_t \bar{u}^{b,n+1}), \bar{q}_h^p),
      \\
      (c_0((\Delta t)^{-1} (\Pi_Q^{b}p^{p,n+1} - \Pi_Q^{b}p^{p,n}) - \partial_t p^{p,n+1} ), q_h^p)_{\Omega^b}
      &= (c_0(d_tp^{p,n+1} - \partial_t p^{p,n+1} ), q_h^p)_{\Omega^b},
      \\ 
      c_h(((\Delta t)^{-1} (\Pi_Q^{b}p^{p,n+1} - \Pi_Q^{b}p^{p,n})-\partial_t p^{p,n+1},(\Delta t)^{-1} (\Pi_Q^{b}p^{b,n+1}
      &- \Pi_Q^{b}p^{b,n})-\partial_t p^{b,n+1}), \alpha q_h^p)
      \\
      =c_h((&d_tp^{p,n+1}-\partial_t p^{p,n+1},d_tp^{b,n+1}-\partial_t p^{b,n+1}), \alpha q_h^p).      
    \end{split}
  \end{equation*}
  Therefore, we can write \cref{eq:error-eqs-1} as
  \begin{subequations}
    \begin{align}
      \label{eq:errors-eq1-1}
      & (d_t e_{u^f}^{h,n+1}, v_h^f)_{\Omega^f}
        - t_h(u^{f,n+1}; \boldsymbol{u}^{f,n+1}, \boldsymbol{v}_h^f)
        + t_h(u_h^{f,n}; \boldsymbol{u}_h^{f,n+1}, \boldsymbol{v}_h^f)
      \\ \nonumber
      & \hspace{2em}
        + a_h(\boldsymbol{e}_{u}^{h,n+1}, \boldsymbol{v}_h)
        + b_h(\boldsymbol{v}_h, \boldsymbol{e}_{p}^{h,n+1})
        + a_h^I((\bar{e}_{u^f}^{h,n+1},d_t\bar{e}_{u^b}^{h,n+1}),(\bar{v}_h^f,\bar{v}_h^b))      
        + b_h^I((\bar{v}_h^f,\bar{v}_h^b), \bar{e}_{p^p}^{h,n+1})
      \\ \nonumber
      & \hspace{2em}
        = (\partial_t u^{f,n+1} - (\Delta t)^{-1} (\Pi_V^{f} u^{f,n+1} - \Pi_V^{f} u^{f,n}), v_h^f)_{\Omega^f}
        + a_h(\boldsymbol{e}_{u}^{I,n+1}, \boldsymbol{v}_h)
      \\ \nonumber
      & \hspace{4em}
        + a_h^I((0, \partial_t \bar{u}^{b,n+1} - d_t\bar{u}^{b,n+1}) ),(\bar{v}_h^f,\bar{v}_h^b)),
      \\
      \label{eq:errors-eq2-1}
      & b_h(\boldsymbol{e}_u^{h,n+1}, \boldsymbol{q}_h)
        + c_h((e_{p^p}^{h,n+1},e_{p^b}^{h,n+1}),q_h^b)
        =  0,
      \\
      \label{eq:errors-eq3-1}
      &(c_0d_t e_{p^p}^{h,n+1}, q_h^p)_{\Omega^b}
        + c_h((d_te_{p^p}^{h,n+1},d_te_{p^b}^{h,n+1}), \alpha q_h^p)
        - b_h^{b}((e_z^{h,n+1}, 0), \boldsymbol{q}_h^p)      
        - b_h^I((\bar{e}_{u^f}^{h,n+1}, d_t\bar{e}_{u^b}^{h,n+1}), \bar{q}_h^p)
      \\ \nonumber
      & \hspace{2em}
        =(c_0(\partial_t p^{p,n+1} - d_tp^{p,n+1}), q_h^p)_{\Omega^b}
        + c_h((\partial_t p^{p,n+1}- d_tp^{p,n+1}, \partial_t p^{b,n+1} - d_tp^{b,n+1}), \alpha q_h^p)
      \\ \nonumber
      & \hspace{4em} 
        - b_h^I((0,  \partial_t \bar{u}^{b,n+1} - d_t\bar{u}^{b,n+1}), \bar{q}_h^p),
      \\
      \label{eq:errors-eq4-1}
      &(\mu^f\kappa^{-1} e_z^{h,n+1}, w_h)_{\Omega^b} + b_h^{b}((w_h, 0), \boldsymbol{e}_{p^p}^{h,n+1})
        =(\mu^f\kappa^{-1} e_z^{I,n+1}, w_h)_{\Omega^b}.
    \end{align}
  \end{subequations}
  To simplify this further, note that
  \begin{equation}
    \label{eq:thdecompositionerror}
    \begin{split}
      - t_h(u^{f,n+1}; \boldsymbol{u}^{f,n+1}, \boldsymbol{v}_h^f) +
      t_h(u_h^{f,n}; \boldsymbol{u}_h^{f,n+1}, \boldsymbol{v}_h^f)
      =&-t_h(u^{f,n+1}; \boldsymbol{u}^{f,n+1}, \boldsymbol{v}_h^f) + t_h(u^{f,n}; \boldsymbol{u}^{f,n+1}, \boldsymbol{v}_h^f)
      \\
      &- t_h(u^{f,n}; \boldsymbol{e}_{u^f}^{I,n+1}, \boldsymbol{v}_h^f)
      -  t_h(u^{f,n}; \boldsymbol{\Pi}_{V}^{f} \boldsymbol{u}^{f,n+1}, \boldsymbol{v}_h^f)
      \\
      &+t_h(u_h^{f,n}; \boldsymbol{\Pi}_{V}^{f} \boldsymbol{u}^{f,n+1}, \boldsymbol{v}_h^f)
      + t_h(u_h^{f,n}; \boldsymbol{e}_{u^{f}}^{h,n+1}, \boldsymbol{v}_h^f).      
    \end{split}
  \end{equation}
  Furthermore, splitting \cref{eq:errors-eq2-1} into its terms on
  $\Omega^f$ and $\Omega^b$ and applying the discrete time derivative
  on $\Omega^b$, we can write \cref{eq:errors-eq2-1} as
  \begin{equation}
    \label{eq:bhfdtbhb}
    b_h^f(\boldsymbol{e}_{u^f}^{h,n+1}, \boldsymbol{q}_h^f)
    + b_h^b(d_t\boldsymbol{e}_{u^b}^{h,n+1}, \boldsymbol{q}_h^b) + c_h((d_t e_{p^p}^{h,n+1}, d_t e_{p^b}^{h,n+1}), q_h^b)
    = 0.
  \end{equation}
  Combining
  \cref{eq:errors-eq1-1,eq:errors-eq3-1,eq:errors-eq4-1,eq:thdecompositionerror,eq:bhfdtbhb}
  we find \cref{eq:reduced-error-eqs}.
\end{proof}

The following auxiliary result will be useful to prove the error
estimate in Theorem \ref{thm:errorEstimate}.

\begin{lemma}
  \label{lem:BDM-inf-sup}
  Let $\boldsymbol{p}_h^{p,n}$, $z_h^n$, $p^p$, and $z^n$ be as
  defined in Lemma \ref{lem:errorequation}. There exists a $C>0$,
  independent of $h$, $\Delta t$, and $n$, such that for $n\ge 0$:
  \begin{subequations}
    \begin{align}
      \label{eq:BDM-inf-sup}
      \norm[0]{e_{p^p}^{h,n+1}}_{\Omega^b}&\le C \mu^f\kappa^{-1} (\norm[0]{e_z^{h,n+1}}_{\Omega^b} + \norm[0]{e_z^{I,n+1}}_{\Omega^b}),
      \\
      \label{eq:trace-estimate}
      \norm[0]{\bar{e}_{p^p}^{h,n+1}}_{\Gamma_I}&\le C \mu^f\kappa^{-1} (\norm[0]{e_z^{h,n+1}}_{\Omega^b} + \norm[0]{e_z^{I,n+1}}_{\Omega^b}).
    \end{align}    
  \end{subequations}  
\end{lemma}
\begin{proof}
  By the inf-sup condition of BDM elements
  (cf.~\cite{Brezzi:1985,Boffi:book}), there exists
  $w_h \in V_h^b \cap H(\text{div}, \Omega^b)$ such that
  $\nabla \cdot w_h = e_{p^p}^{h,n}$, $w_h\cdot n = 0$ on
  $\partial \Omega^b \setminus \Gamma_P^b$ and
  $\norm[0]{w_h}_{\Omega^b} \le C \norm[0]{e_{p^p}^{h,n}}_{\Omega^b}$,
  with $C>0$ independent of $h$, $\Delta t$, and
  $n$. \Cref{eq:BDM-inf-sup} follows by taking this $(w_h,0)$ in
  \cref{eq:reduced-error-eq4} and using the Cauchy--Schwarz inequality
  (note that
  $\langle \bar{e}_{p^p}^{h,n}, w_h \cdot n^b
  \rangle_{\partial\mathcal{T}^b}=0$ because
  $w_h \in V_h^b \cap H(\text{div},\Omega^b)$, $\bar{e}_{p^p}^{h,n}$
  is single-valued on element boundaries,
  $p^p=\bar{p}_h^p=\bar{\Pi}_{Q^0}^bp^p=0$ on $\Gamma_P^b$, and
  $w_h \cdot n = 0$ on $\partial\Omega^b \setminus \Gamma_P^b$). The
  proof of \cref{eq:trace-estimate} is given by
  \cite[Lemma~4]{Cesmelioglu:2023a}.
\end{proof}

The main result of this section is the following theorem.

\begin{theorem}
  \label{thm:errorEstimate}
  Suppose that
  $\{(\boldsymbol{u}_h^n, \boldsymbol{p}_h^n, z_h^{n},
  \boldsymbol{p}_h^{p,n})\}_n$ are the solutions to \cref{eq:hdgfd},
  that the assumptions of Lemma~\ref{lem:boundXn} hold and that
  $(u,p,z,p^p)$ is the solution to
  \cref{eq:navierstokes,eq:biot,eq:interface,eq:bcs,eq:ics}, with
  $u_0^j(x) = 0$, $j=f,b$ and $p_0^p(x)=0$, on the time interval
  $J=(0,T]$. Furthermore, let
  $\boldsymbol{u}:=(u,u|_{\Gamma_0^f},u|_{\Gamma_0^b})$,
  $\boldsymbol{p} := (p,p|_{\Gamma_0^f},p|_{\Gamma_0^b})$, and
  $\boldsymbol{p}^p := (p^p,p^p|_{\Gamma_0^b})$. If 
  \begin{equation}
    \label{eq:regularity-assumptions}
    \begin{split}
      u^f &\in H^{1}(J,[H^k(\Omega^f)]^d)\cap H^{2}(J,[L^2(\Omega^f)]^d) \cap \ell^{\infty}(J,[W^{1,3}(\Omega^f) \cap H^{k+1}(\Omega^f)]^d),
      \\
      u^b &\in H^{2}(J,[H^{1}(\Omega^b)]^d) \cap W^{2,1}(J,[H^{k+1}(\Omega^b)]^d),
      \\
      z &\in \ell^{\infty}(J, H^{k}(\Omega^b)),
      \qquad 
      p^p, p^b \in H^{2}(J, L^2(\Omega^b)),      
    \end{split}
  \end{equation}
  then,
  \begin{multline}
    \label{eq:ehufubppzestimate}
    \norm[0]{e_{u^f}^{h,m}}_{\Omega^f}^2 + a_h^b(\boldsymbol{e}_{u^b}^{h,m},\boldsymbol{e}_{u^b}^{h,m})
    + \lambda^{-1}\norm[0]{\alpha e_{p^p}^{h,m} - e_{p^b}^{h,m}}_{\Omega^b}^2
    + c_0 \norm[0]{e_{p^p}^{h,m}}_{\Omega^b}^2
    \\
    + \Delta t\sum_{i=1}^m \sbr[3]{\mu^f\tnorm{\boldsymbol{e}_{u^f}^{h,i}}_{v,f}^2
      + \gamma\mu^f\kappa^{-1/2} \norm[0]{\alpha \bar{e}_{u^f}^{h,i} - d_t \bar{e}_{u^b}^{h,i}}_{\Gamma_I}^2
      +  \mu^f\kappa^{-1}\norm[0]{e_z^{h,i}}_{\Omega^b}^2 }
    \le C_G \sbr[1]{h^{2k} +(\Delta t)^2},         
  \end{multline}
  with $C_G$ a constant resulting from a discrete Gr\"onwall
  inequality that depends on $T$ and the norms of $u^f$, $u^b$, $z$,
  $p^p$, and $p^b$ in \cref{eq:regularity-assumptions}. Moreover,
  \begin{equation}
    \label{eq:ehpbestimate}
    \tnorm{\boldsymbol{e}_{p^b}^{h,m}}_{q,b}
    \le C (\mu^b)^{1/2}a_h^b(\boldsymbol{e}_{u^b}^{h,m},\boldsymbol{e}_{u^b}^{h,m})^{1/2}
    + C \mu^b h^k\norm[0]{u^{b,m}}_{k+1,\Omega^b}.
  \end{equation}
\end{theorem}
\begin{proof}
  Set $\boldsymbol{v}_h^f = \boldsymbol{e}_{u^f}^{h,n+1}$,
  $\boldsymbol{v}_h^b = d_t \boldsymbol{e}_{u^b}^{h,n+1}$,
  $\boldsymbol{q}_h^f = -\boldsymbol{e}_{p^f}^{h,n+1}$,
  $\boldsymbol{q}_h^b = -\boldsymbol{e}_{p^b}^{h,n+1}$,
  $w_h = e_z^{h,n+1}$,
  $\boldsymbol{q}_h^p = \boldsymbol{e}_{p^p}^{h,n+1}$ in
  \cref{eq:reduced-error-eqs}, sum the equations, use the algebraic
  inequality $a(a-b) \ge (a^2-b^2)/2$ for the time-derivative terms,
  and use \cref{eq:coercivityahthresult}, to obtain:
  \begin{equation}
    \label{eq:sum-error-eqs}
    \begin{split}
      \tfrac{1}{2\Delta t}&
      (\norm[0]{e_{u^f}^{h,n+1}}_{\Omega^f}^2-\norm[0]{e_{u^f}^{h,n}}_{\Omega^f}^2) 
      +\tfrac{1}{2}c_{ae}^f\mu^f \tnorm{\boldsymbol{e}_{u^f}^{h,n+1}}_{v,f}^2
      + \tfrac{1}{2\Delta t}(a_h^b(\boldsymbol{e}_{u^b}^{h,n+1}, \boldsymbol{e}_{u^b}^{h,n+1})
      -a_h^b(\boldsymbol{e}_{u^b}^{h,n}, \boldsymbol{e}_{u^b}^{h,n})) 
      \\
      & \hspace{2em}
      + \gamma \mu^f \kappa^{-1/2} \norm[0]{ (\bar{e}_{u^f}^{h,n+1} - d_t\bar{e}_{u^b}^{h,n+1})^t }_{\Gamma_I}^2
      + \tfrac{\lambda^{-1}}{2\Delta t} (\norm[0]{\alpha e_{p^p}^{h,n+1} - e_{p^b}^{h,n+1}}_{\Omega^b}^2
      -\norm[0]{\alpha e_{p^p}^{h,n} - e_{p^b}^{h,n}}_{\Omega^b}^2)
      \\
      & \hspace{2em}
      + \tfrac{c_0}{2\Delta t} (\norm[0]{e_{p^p}^{h,n+1}}_{\Omega^b}^2
      -\norm[0]{e_{p^p}^{h,n}}_{\Omega^b}^2)
      + \mu^f\kappa^{-1}\norm[0]{e_z^{h,n+1}}_{\Omega^b}^2
      \\
      \le & I_1^{n} + I_2^{n} + I_3^{n} + I_4^{n},
    \end{split}
  \end{equation}
  where $e_{u^{f}}^{n} = u^{f,n} - u_h^{f,n}$ and where
  \begin{equation*}
    \begin{split}
      I_1^n &:=  I_{1a}^n + I_{1b}^n + I_{1c}^n + I_{1d}^n + I_{1e}^n + I_{1f}^n
      \\
      &:= (\partial_t u^{f,n+1} - (\Delta t)^{-1} (\Pi_V^{f} u^{f,n+1} - \Pi_V^{f} u^{f,n}), e_{u^f}^{h,n+1})_{\Omega^f}
      \\
      &\quad + \sbr[1]{t_h(u^{f,n+1}; \boldsymbol{u}^{f,n+1}, \boldsymbol{e}_{u^f}^{h,n+1}) - t_h(u^{f,n}; \boldsymbol{u}^{f,n+1}, \boldsymbol{e}_{u^f}^{h,n+1})}
      \\
      &\quad + t_h(u^{f,n}; \boldsymbol{e}_{u^f}^{I,n+1}, \boldsymbol{e}_{u^f}^{h,n+1}) +  \sbr[1]{t_h(u^{f,n}; \boldsymbol{\Pi}_{V}^{f} \boldsymbol{u}^{f,n+1}, \boldsymbol{e}_{u^f}^{h,n+1})
      - t_h(u_h^{f,n}; \boldsymbol{\Pi}_{V}^{f} \boldsymbol{u}^{f,n+1}, \boldsymbol{e}_{u^f}^{h,n+1})
      }
      \\
      &\quad + a_h^f(\boldsymbol{e}_{u^f}^{I,n+1}, \boldsymbol{e}_{u^f}^{h,n+1}) + a_h^b(\boldsymbol{e}_{u^b}^{I,n+1}, d_t\boldsymbol{e}_{u^b}^{h,n+1}),
      \\
      I_2^n &:= a_h^I((0, \partial_t \bar{u}^{b,n+1} - d_t \bar{u}^{b,n+1}),(\bar{e}_{u^f}^{h,n+1},d_t\bar{e}_{u^b}^{h,n+1})) ,
      \\
      I_3^n &:= I_{3a}^n + I_{3b}^n + I_{3c}^n
      \\
      &:=(c_0(\partial_t p^{p,n+1} - d_t p^{p,n+1}), e_{p^p}^{h,n+1})_{\Omega^b}
      \\
      & \hspace{2em}
      + c_h((\partial_t p^{p,n+1}-d_tp^{p,n+1}, \partial_t p^{b,n+1}-d_t p^{b,n+1}), \alpha e_{p^p}^{h,n+1})
      \\
      & \hspace{2em}
      - b_h^I((0, \partial_t \bar{u}^{b,n+1} - d_t \bar{u}^{b,n+1}), \bar{e}_{p^p}^{h,n+1}),
      \\
      I_4^n &:= (\mu^f\kappa^{-1} e_z^{I,n+1}, e_z^{h,n+1})_{\Omega^b} .      
    \end{split}
  \end{equation*}
  By definition of $I_{1a}^{n}$, the Cauchy--Schwarz inequality,
  \cref{eq:dpoincareineq}, the triangle inequality, approximation
  properties of the BDM interpolation operator ($\Pi_V^f$), Taylor's
  theorem, and Young's inequality,
  \begin{equation}
    \label{eq:I1a-estimate}
    \begin{split}
      |I_{1a}^{n}|
      &\le C\norm[0]{-d_t e_{u^f}^{I,n+1} + d_t u^{f,n+1} -\partial_t u^{f,n+1}}_{\Omega^f} \tnorm{\boldsymbol{e}_{u^f}^{h,n+1}}_{v,f} 
      \\
      &\le 2\psi \tnorm{\boldsymbol{e}_{u^f}^{h,n+1}}_{v,f}^2
      + \tfrac{C}{\psi}h^{2k}(\Delta t)^{-1} \norm[0]{\partial_t u^f}_{L^2(t_n, t_{n+1}; H^{k}(\Omega^f))}^2
      + \tfrac{C}{\psi}\Delta t \norm[0]{\partial_{tt} u^f}_{L^2(t_n, t_{n+1}; L^2(\Omega^f))}^2,      
    \end{split}
  \end{equation}
  where $\psi>0$ will be chosen later. Next, by
  \cref{eq:upperboundthdif} and Young's inequality,
  \begin{equation}
    \label{eq:I1b-estimate}
    \begin{split}
      |I_{1b}^n|
      &\le C\norm[0]{u^{f,n+1}-u^{f,n}}_{1,h,\Omega^f} \norm[0]{u^{f,n+1}}_{1,\Omega^f}\tnorm{\boldsymbol{e}_{u^f}^{h,n+1}}_{v,f}
      \\
      &\le C(\Delta t)^{1/2} \norm[0]{\partial_t u^f}_{L^2(t_n, t_{n+1}; H^1(\Omega^f))} \norm[0]{u^{f,n+1}}_{1,\Omega^f}
      \tnorm{\boldsymbol{e}_{u^f}^{h,n+1}}_{v,f}
      \\
      &\le \psi\tnorm{\boldsymbol{e}_{u^f}^{h,n+1}}_{v,f}^2
      + \frac{C}{\psi} \Delta t \norm[0]{\partial_t u^f}_{L^2(t_n, t_{n+1}; H^1(\Omega^f))}^2 \norm[0]{u^{f,n+1}}_{1,\Omega^f}^2.      
    \end{split}
  \end{equation}
  By \cref{eq:upperboundthdif}, approximation properties of the BDM
  interpolation operator ($\Pi_V^f$) and facet $L^2$-projection
  ($\bar{\Pi}_V^f$), and Young's inequality,
  \begin{equation}
    \label{eq:I1c-estimate}
    \begin{split}
      |I_{1c}^n|
      &\le C\norm[0]{u^{f,n}}_{1,h,\Omega^f}\tnorm{\boldsymbol{e}_{u^f}^{I,n+1}}_{v,f} \tnorm{\boldsymbol{e}_{u^f}^{h,n+1}}_{v,f}
      \\
      &\le C h^k \norm[0]{u^{f,n}}_{1,\Omega^f} \norm[0]{u^{f,n+1}}_{k+1,\Omega^f} \tnorm{\boldsymbol{e}_{u^f}^{h,n+1}}_{v,f}
      \\
      &\le \psi \tnorm{\boldsymbol{e}_{u^f}^{h,n+1}}_{v,f}^2
      + \frac{C}{\psi} h^{2k} \norm[0]{u^{f,n}}_{1,\Omega^f}^2 \norm[0]{u^{f,n+1}}_{k+1,\Omega^f}^2.      
    \end{split}
  \end{equation}  
  For $I_{1d}^n$ we have for $\psi > 0$ (see\cite[Appendix~C]{Cesmelioglu:2023b}):
  \begin{equation}
    \label{eq:I1d-estimate}
    |I_{1d}^n|
    \le 2\psi \tnorm{\boldsymbol{e}_{u^f}^{h,n+1}}_{v,f}^2
    + \frac{C}{\psi}h^{2k}\norm[0]{u^{f,n+1}}_{k+1,\Omega^f}^2 \norm[0]{u^{f,n}}_{k+1,\Omega^f}^2
    + \frac{C}{\psi}\norm[0]{e_{u^f}^{h,n}}_{\Omega^f}^2 \norm[0]{u^{f,n+1}}_{W_3^1(\Omega^f)}^2.
  \end{equation}
  For $I_{1e}^n$, using \cref{eq:ah-continuity-j}, Young's
  inequalities, and interpolation properties, we find
  \begin{equation}
    \label{eq:I1e-estimate}
    |I_{1e}^n|
    \le C \mu^f h^{2k} \norm[0]{u^{f,n+1}}_{k+1,\Omega^f}^2 + \tfrac{1}{8}c_{ae}^f\mu^f \tnorm{\boldsymbol{e}_{u^f}^{h,n+1}}_{v,f}^2.
  \end{equation}
  We postpone estimating $I_{1f}^n$ until later and proceed with
  estimating $I_2^n$. By the Cauchy--Schwarz inequality, Young's
  inequality, and the trace inequality,
  \begin{equation}
    \label{eq:I2-estimate}
    \begin{split}
      |I_2^n|
      &\le \gamma \mu^f \kappa^{-1/2}\norm[0]{d_t\bar{u}^{b,n+1} - \partial_t \bar{u}^{b,n+1}}_{\Gamma_I}
      \norm[0]{\bar{e}_{u^f}^{h,n+1} - d_t \bar{e}_{u^b}^{h,n+1}}_{\Gamma_I}
      \\
      &\le \tfrac{1}{2} \gamma \mu^f \kappa^{-1/2} \norm[0]{\bar{e}_{u^f}^{h,n+1} - d_t \bar{e}_{u^b}^{h,n+1}}_{\Gamma_I}^2
      + \tfrac{1}{2}\gamma \mu^f \kappa^{-1/2}\norm[0]{d_t\bar{u}^{b,n+1} - \partial_t \bar{u}^{b,n+1}}_{\Gamma_I}^2
      \\
      &\le \tfrac{1}{2} \gamma \mu^f \kappa^{-1/2} \norm[0]{\bar{e}_{u^f}^{h,n+1} - d_t \bar{e}_{u^b}^{h,n+1}}_{\Gamma_I}^2
      + C \mu^f \kappa^{-1/2}\Delta t \norm[0]{\partial_{tt} \bar{u}^b}_{L^2(t_n, t_{n+1}; L^2(\Gamma_I))}^2
      \\
      &\le \tfrac{1}{2} \gamma \mu^f \kappa^{-1/2} \norm[0]{\bar{e}_{u^f}^{h,n+1} - d_t \bar{e}_{u^b}^{h,n+1}}_{\Gamma_I}^2
      + C \mu^f \kappa^{-1/2}\Delta t \norm[0]{\partial_{tt} u^b}_{L^2(t_n, t_{n+1}; H^1(\Omega^b))}^2.       
    \end{split}
  \end{equation}
  By the Cauchy--Schwarz inequality, Lemma~\ref{lem:BDM-inf-sup},
  Young's inequality, and an argument similar to the estimate in
  \cref{eq:I1a-estimate},
  \begin{equation}
    \label{eq:I3a-estimate}
    \begin{split}
      |I_{3a}^n|
      &\le c_0 \norm[0]{d_t p^{p,n+1} - \partial_t p^{p,n+1}}_{\Omega^b} \norm[0]{e_{p^p}^{h,n+1}}_{\Omega^b}
      \\
      &\le c_0 \mu^f\kappa^{-1} C\norm[0]{d_t p^{p,n+1} - \partial_t p^{p,n+1}}_{\Omega^b} (\norm[0]{e_{z}^{h,n+1}}_{\Omega^b}
      + \norm[0]{e_{z}^{I,n+1}}_{\Omega^b})
      \\
      &\le \tfrac{1}{5} \mu^f\kappa^{-1} \norm[0]{e_z^{h,n+1}}_{\Omega^b}^2
      +  C\mu^f\kappa^{-1} h^{2k}\norm[0]{z^{n+1}}_{H^{k}(\Omega^b)}^2
      + Cc_0^2\mu^f\kappa^{-1}\Delta t \norm[0]{\partial_{tt}p^p}_{L^2(t_n, t_{n+1}; L^2(\Omega^b))}^2.      
    \end{split}
  \end{equation}
  Likewise, and using that $0< \alpha \le 1$, we find for $I_{3b}^n$,
  \begin{equation}
    \label{eq:I3b-estimate}
    \begin{split}
      |I_{3b}^n| 
      &\le \lambda^{-1}\norm[0]{\alpha (d_tp^{p,n+1}-\partial_t p^{p,n+1}) - (d_t p^{b,n+1}-\partial_t p^{b,n+1})}_{\Omega^b}
      \norm[0]{e_{p^p}^{h,n+1}}_{\Omega^b}
      \\
      &\le \tfrac{1}{5} \mu^f\kappa^{-1} \norm[0]{e_z^{h,n+1}}_{\Omega^b}^2
      +  C\mu^f\kappa^{-1} h^{2k}\norm[0]{z^{n+1}}_{H^{k}(\Omega^b)}^2
      \\
      &\quad + C\lambda^{-2}\mu^f \kappa^{-1}\Delta t (\norm[0]{\partial_{tt}p^p}_{L^2(t_n, t_{n+1}; L^2(\Omega^b))}^2
      +\norm[0]{\partial_{tt}p^b}_{L^2(t_n, t_{n+1}; L^2(\Omega^b))}^2).
    \end{split}
  \end{equation}
  For $I_{3c}^n$, by \cref{eq:trace-estimate} and an argument similar
  to the estimate \cref{eq:I2-estimate},
  \begin{equation}
    \label{eq:I3c-estimate}
    |I_{3c}^n|
    \le \tfrac{1}{5} \mu^f\kappa^{-1} \norm[0]{e_z^{h,n+1}}_{\Omega^b}^2
    +  C\mu^f\kappa^{-1} h^{2k}\norm[0]{z^{n+1}}_{H^{k}(\Omega^b)}^2
    + C \mu^f\kappa^{-1}\Delta t \norm[0]{\partial_{tt}u^b}_{L^2(t_n, t_{n+1}; H^1(\Omega^b))}^2.      
  \end{equation}
  For $I_4^n$, using the Cauchy--Schwarz inequality, Young's
  inequality, and the approximation properties of the BDM interpolant,
  we find
  \begin{equation}
    \label{eq:I4-estimate}
    |I_4^n| \le \tfrac{1}{5}\mu^f\kappa^{-1}\norm[0]{e_z^{h,n+1}}_{\Omega^b}^2
    + C \mu^f \kappa^{-1} h^{2k} \norm[0]{z^{n+1}}_{H^{k}(\Omega^b)}^2.
  \end{equation}
  Combining the estimates in
  \cref{eq:I1a-estimate,eq:I1b-estimate,eq:I1c-estimate,eq:I1d-estimate,eq:I1e-estimate,eq:I2-estimate,eq:I3a-estimate,eq:I3b-estimate,eq:I3c-estimate,eq:I4-estimate}
  with \cref{eq:sum-error-eqs}, choosing $\psi = c_{ae}^f\mu^f/24$,
  summing for $n=0$ to $n=m-1$, multiplying both sides of the
  inequality by $\Delta t$, and taking into account the vanishing
  initial data, we find:
  \begin{equation}
    \label{eq:summed-error-estimate}
    \begin{split}
      & \tfrac{1}{2}\norm[0]{e_{u^f}^{h,m}}_{\Omega^f}^2 
      + \tfrac{\Delta t}{8}c_{ae}^f\mu^f \sum_{i=1}^m\tnorm{\boldsymbol{e}_{u^f}^{h,i}}_{v,f}^2
      + \tfrac{1}{2} a_h^b(\boldsymbol{e}_{u^b}^{h,m}, \boldsymbol{e}_{u^b}^{h,m})
      + \tfrac{\Delta t}{2}\gamma \mu^f \kappa^{-1/2} \sum_{i=1}^m\norm[0]{ (\bar{e}_{u^f}^{h,i} - d_t\bar{e}_{u^b}^{h,i})^t }_{\Gamma_I}^2
      \\
      & \hspace{2em}      
      + \tfrac{\lambda^{-1}}{2} \norm[0]{\alpha e_{p^p}^{h,m} - e_{p^b}^{h,m}}_{\Omega^b}^2
      + \tfrac{c_0}{2} \norm[0]{e_{p^p}^{h,m}}_{\Omega^b}^2
      + \tfrac{\Delta t}{5} \mu^f\kappa^{-1}\sum_{i=1}^m\norm[0]{e_z^{h,i}}_{\Omega^b}^2
      \\
      &\quad
      \le  C(\mu^f)^{-1}h^{2k} \norm[0]{\partial_t u^f}_{L^2(0, t_{m}; H^{k}(\Omega^f) )}^2
      +  C(\mu^f)^{-1}(\Delta t)^2 \norm[0]{\partial_{tt} u^f}_{L^2(0, t_{m}; L^2(\Omega^f))}^2 
      \\
      &\qquad + C (\mu^f)^{-1} (\Delta t)^2 \sum_{i=1}^m\norm[0]{\partial_t u^f}_{L^2(t_{i-1}, t_{i}; H^1(\Omega^f))}^2 \norm[0]{u^{f,i}}_{1,\Omega^f}^2
      \\
      &\qquad + C (\mu^f)^{-1} \Delta t h^{2k} \sum_{i=1}^m((\mu^f)^2+\norm[0]{u^{f,i-1}}_{1,\Omega^f}^2+\norm[0]{u^{f,i-1}}_{k+1,\Omega^f}^2) \norm[0]{u^{f,i}}_{k+1,\Omega^f}^2 + \Delta t \sum_{i=1}^{m} I_{1f}^{i-1}
      \\
      &\qquad + C (\mu^f)^{-1} \Delta t\sum_{i=1}^m\norm[0]{e_{u^f}^{h,i-1}}_{\Omega^f}^2 \norm[0]{u^{f,i}}_{W_3^1(\Omega^f)}^2
      + C \mu^f \kappa^{-1/2}(\Delta t)^2 \norm[0]{\partial_{tt} u^b}_{L^2(0, t_{m}; H^1(\Omega^b))}^2
      \\
      &\qquad +  C\mu^f\kappa^{-1} \Delta t h^{2k}\sum_{i=1}^m\norm[0]{z^{i}}_{H^{k}(\Omega^b)}^2
      + C c_0^2 \mu^f\kappa^{-1} (\Delta t)^2 \norm[0]{\partial_{tt} p^p}_{L^2(0, t_{m}; L^2(\Omega^b))}^2 
      \\
      &\qquad + C \lambda^{-2} \mu^f \kappa^{-1}(\Delta t)^2 \big(\norm[0]{\partial_{tt}p^p}_{L^2(0, t_{m}; L^2(\Omega^b))}^2+\norm[0]{\partial_{tt}p^b}_{L^2(0, t_{m}; L^2(\Omega^b))}^2\big)
      \\
      &\qquad + C \mu^f\kappa^{-1} (\Delta t)^2\norm[0]{\partial_{tt}u^b}_{L^2(0, t_{m}; H^1(\Omega^b))}^2.      
    \end{split}
  \end{equation}
  Let us now consider the term $I_{1f}$. Using summation-by-parts and
  that the initial data vanishes,
  \begin{equation*}
    \Delta t \sum_{i=1}^m I_{1f}^{i-1}
    =a_h^b(\boldsymbol{e}_{u^b}^{I,m},\boldsymbol{e}_{u^b}^{h,m})
    + \sum_{i=2}^m a_h^b(\boldsymbol{e}_{u^b}^{I,i-1}-\boldsymbol{e}_{u^b}^{I,i},\boldsymbol{e}_{u^b}^{h,i-1}).
  \end{equation*}
  Note also that, by the Cauchy--Schwarz and Young's inequalities,
  \cref{eq:ah-continuity-j}, and approximation properties of the BDM
  interpolant,
  \begin{equation*}
    \begin{split}
      a_h^b(\boldsymbol{e}_{u^b}^{I,i-1}- \boldsymbol{e}_{u^b}^{I,i},\boldsymbol{e}_{u^b}^{h,i-1})
      &\le
      a_h^b(\boldsymbol{e}_{u^b}^{I,i-1}- \boldsymbol{e}_{u^b}^{I,i},\boldsymbol{e}_{u^b}^{I,i-1}- \boldsymbol{e}_{u^b}^{I,i})^{1/2}
      a_h^b(\boldsymbol{e}_{u^b}^{h,i-1},\boldsymbol{e}_{u^b}^{h,i-1})^{1/2}
      \\
      &\le
      (\Delta t)^{-1} a_h^b(\boldsymbol{e}_{u^b}^{I,i-1}- \boldsymbol{e}_{u^b}^{I,i},\boldsymbol{e}_{u^b}^{I,i-1}- \boldsymbol{e}_{u^b}^{I,i})
      +\tfrac{\Delta t}{4}a_h^b(\boldsymbol{e}_{u^b}^{h,i-1},\boldsymbol{e}_{u^b}^{h,i-1})
      \\
      &\le
      (\Delta t)^{-1} c_{ac}^b{\mu^b}\tnorm{\boldsymbol{e}_{u^b}^{I,i-1}- \boldsymbol{e}_{u^b}^{I,i}}_{v',j}^2
      +\tfrac{\Delta t}{4}a_h^b(\boldsymbol{e}_{u^b}^{h,i-1},\boldsymbol{e}_{u^b}^{h,i-1})
      \\
      &\le
      (\Delta t)^{-1}C{\mu^b}h^{2k}\norm[0]{u^{b,i-1}-u^{b,i}}_{k+1,\Omega^b}^2
      +\tfrac{\Delta t}{4}a_h^b(\boldsymbol{e}_{u^b}^{h,i-1},\boldsymbol{e}_{u^b}^{h,i-1})
      \\
      &\le
      C{\mu^b} h^{2k}\norm[0]{\partial_tu^b}_{L^2(t_{i-1},t_i;H^{k+1}(\Omega^b))}^2
      +\tfrac{\Delta t}{4}a_h^b(\boldsymbol{e}_{u^b}^{h,i-1},\boldsymbol{e}_{u^b}^{h,i-1}),
    \end{split}
  \end{equation*}
  and, similarly,
  \begin{equation*}
    a_h^b(\boldsymbol{e}_{u^b}^{I,m},\boldsymbol{e}_{u^b}^{h,m})
    \le C {\mu^b} h^{2k} \norm[0]{u^{b,m}}_{k+1,\Omega^b}^2
    + \tfrac{1}{4} a_h^b(\boldsymbol{e}_{u^b}^{h,m},\boldsymbol{e}_{u^b}^{h,m}).    
  \end{equation*}
  The above inequalities together with \cref{eq:summed-error-estimate}
  result in:
  \begin{align*}
    &\tfrac{1}{2}\norm[0]{e_{u^f}^{h,m}}_{\Omega^f}^2 
      +\tfrac{\Delta t}{8}c_{ae}^f\mu^f \sum_{i=1}^m\tnorm{\boldsymbol{e}_{u^f}^{h,i}}_{v,f}^2
      + \tfrac{1}{4} a_h^b(\boldsymbol{e}_{u^b}^{h,m}, \boldsymbol{e}_{u^b}^{h,m})
      + \tfrac{\Delta t}{2}\gamma \mu^f \kappa^{-1/2} \sum_{i=1}^m\norm[0]{ (\bar{e}_{u^f}^{h,i} - d_t\bar{e}_{u^b}^{h,i})^t }_{\Gamma_I}^2
    \\
    & \hspace{2em}
      + \tfrac{\lambda^{-1}}{2} \norm[0]{\alpha e_{p^p}^{h,m} - e_{p^b}^{h,m}}_{\Omega^b}^2
      + \tfrac{c_0}{2} \norm[0]{e_{p^p}^{h,m}}_{\Omega^b}^2
      + \tfrac{\Delta t}{5} \mu^f\kappa^{-1}\sum_{i=1}^m\norm[0]{e_z^{h,i}}_{\Omega^b}^2
    \\
    &\quad \le C (\mu^f)^{-1}h^{2k} \norm[0]{\partial_t u^f}_{L^2(0, t_{m}; H^{k}(\Omega^f))}^2
      +  C (\mu^f)^{-1}(\Delta t)^2 \norm[0]{\partial_{tt} u^f}_{L^2(0, t_{m}; L^2(\Omega^f))}^2 
    \\
    &\qquad + C T (\mu^f)^{-1} (\Delta t)^2 \norm[0]{\partial_t u^f}_{L^2(0, t_{m}; H^1(\Omega^f))}^2 \norm[0]{u^{f}}_{\ell^\infty(0,t_m;H^1(\Omega^f))}^2
    \\
    &\qquad + C T (\mu^f)^{-1} h^{2k} ((\mu^f)^2+\norm[0]{u^{f}}_{\ell^\infty(0,t_m;H^1(\Omega^f))}^2+\norm[0]{u^{f}}_{\ell^\infty(0,t_m;H^{k+1}(\Omega^f))}^2) \norm[0]{u^{f}}_{\ell^\infty(0,t_m;H^{k+1}(\Omega^f))}^2 
    \\
    &\qquad + C \mu^b h^{2k} \norm[0]{\partial_t u^b}_{L^2(0,t_m;H^{k+1}(\Omega^b))}^2 + C \mu^b h^{2k} \norm[0]{u^{b}}_{\ell^\infty(0,t_m;H^{k+1}(\Omega^b))}^2 
    \\
    &\qquad + C (\mu^f)^{-1} \Delta t\sum_{i=1}^m\norm[0]{e_{u^f}^{h,i-1}}_{\Omega^f}^2 \norm[0]{u^{f,i}}_{W_3^1(\Omega^f)}^2
      + C \mu^f \kappa^{-1/2}(\Delta t)^2 \norm[0]{\partial_{tt} u^b}_{L^2(0, t_{m}; H^1(\Omega^b))}^2
    \\
    &\qquad      
      + \tfrac{\Delta t}{4} \sum_{i=2}^{m}a_h^b(\boldsymbol{e}_{u^b}^{h,i-1},\boldsymbol{e}_{u^b}^{h,i-1})
      +  C\mu^f\kappa^{-1} T h^{2k}\norm[0]{z}_{\ell^{\infty}(0,t_m;H^{k}(\Omega^b))}^2      
    \\
    &\qquad + C c_0^2 \mu^f\kappa^{-1} (\Delta t)^2 \norm[0]{\partial_{tt} p^p}_{L^2(0, t_{m}; L^2(\Omega^b))}^2
      + C \mu^f\kappa^{-1} (\Delta t)^2\norm[0]{\partial_{tt}u^b}_{L^2(0, t_{m}; H^1(\Omega^b))}^2
    \\
    &\qquad
      + C \lambda^{-2} \mu^f \kappa^{-1}(\Delta t)^2 (\norm[0]{\partial_{tt}p^p}_{L^2(0, t_{m}; L^2(\Omega^b))}^2+\norm[0]{\partial_{tt}p^b}_{L^2(0, t_{m}; L^2(\Omega^b))}^2).
  \end{align*}
  \Cref{eq:ehufubppzestimate} follows by a discrete Gr\"{o}nwall
  inequality (see, e.g., \cite[Lemma~28]{Layton:book}).

  We next prove \cref{eq:ehpbestimate}. For this, let us first note
  that by \cref{eq:inf-sup-bj} there exists a
  $\tilde{\boldsymbol{v}}_h^b \in \tilde{\boldsymbol{V}}_h^b$, with
  $\tilde{\boldsymbol{V}}_h^b$ defined in \cref{eq:Vh-tilde}, such
  that
  $b_h^b(\tilde{\boldsymbol{v}}_h^b,
  \boldsymbol{e}_{p^b}^{h,n+1})=\tnorm{\boldsymbol{e}_{p^b}^{h,n+1}}_{q,b}^2$
  and
  $\tnorm{\tilde{\boldsymbol{v}}_{h}^{b}}_{v,b} \le
  C\tnorm{\boldsymbol{e}_{p^b}^{h,n+1}}_{q,b}$. Take
  $\boldsymbol{v}_h^f = \boldsymbol{0}$ and
  $\boldsymbol{v}_h^b = \tilde{\boldsymbol{v}}_h^b$ in
  \cref{eq:reduced-error-eq1} which then reduces to
  $a_h^b(\boldsymbol{e}_{u^b}^{h,n+1}, \tilde{\boldsymbol{v}}_h^b)
  +\tnorm{\boldsymbol{e}_{p^b}^{h,n+1}}_{q,b}^2 =
  a_h^b(\boldsymbol{e}_{u^b}^{I,n+1}, \tilde{\boldsymbol{v}}_h^b)$. By
  \cref{eq:ah-continuity-j} we obtain:
  \begin{equation*}
    \tnorm{\boldsymbol{e}_{p^b}^{h,n+1}}_{q,b}^2
    \le |a_h^b(\boldsymbol{e}_{u^b}^{h,n+1}, \tilde{\boldsymbol{v}}_h^b)|
    + |a_h^b(\boldsymbol{e}_{u^b}^{I,n+1}, \tilde{\boldsymbol{v}}_h^b)|
    \le C \mu^b\del[1]{\tnorm{\boldsymbol{e}_{u^b}^{h,n+1}}_{v,b} 
      + \tnorm{\boldsymbol{e}_{u^b}^{I,n+1}}_{v',b}}\tnorm{\boldsymbol{e}_{p^b}^{h,n+1}}_{q,b},
  \end{equation*}
  so that \cref{eq:ehpbestimate} follows by using
  \cref{eq:ah-coercive-j} and approximation properties of the
  interpolant.
\end{proof}

An immediate consequence of Theorem \ref{thm:errorEstimate}, the
triangle inequality, and approximation properties of the different
interpolants is the following corollary.

\begin{corollary}
  Suppose that all the assumptions in Theorem \ref{thm:errorEstimate}
  hold. Then:
  \begin{equation*}
    \begin{split}
      &\norm[0]{u^{f,m} - u_h^{f,m}}_{\Omega^f}^2
      + \mu^b\tnorm{\boldsymbol{u}^{b,m}-\boldsymbol{u}_h^{b,m}}_{v,b}^2
      + \lambda^{-1}\norm[0]{\alpha(p^{p,m}-p_h^{p,m})-(p^{b,m}-p_h^{b,m})}_{\Omega^b}^2
      \\
      &+ c_0 \norm[0]{p^{p,m}-p_h^{p,m}}_{\Omega^b}^2 + \tnorm{\boldsymbol{p}^{b,m}-\boldsymbol{p}_h^{b,m}}_{q,b}^2
      \\
      &+ \Delta t\sum_{i=1}^m \sbr[2]{\mu^f\tnorm{\boldsymbol{u}^{f,i} - \boldsymbol{u}_h^{f,i}}_{v,f}^2
        + \gamma\mu^f\kappa^{-1/2} \norm[0]{\alpha (\bar{u}^{f,i} - \bar{u}_h^{f,i}) - d_t (\bar{u}^{b,i}- \bar{u}_h^{b,i})}_{\Gamma_I}^2
        + \mu^f\kappa^{-1} \norm[0]{z^i - z_h^{i}}_{\Omega^b}^2 }
      \\
      &\le C_G' \sbr[1]{h^{2k} + (\Delta t)^2},      
    \end{split}
  \end{equation*}
  with $C_G'$ depending on $C_G$ in Theorem~\ref{thm:errorEstimate},
  the norms of the exact solutions, the constants of the approximation
  properties of the interpolation operators $\Pi_V^j$,
  $\bar{\Pi}_V^j$, $\Pi_Q^j$, $\bar{\Pi}_Q^j$, and the different model
  parameters.
\end{corollary}

\section{Numerical example}
\label{sec:numexample}

In this final section, we present a numerical example to confirm our
analysis. For this, we consider the time-dependent manufactured
solution of \cite[Section 6.2]{Cesmelioglu:2023a}. We consider the
domain $\Omega := (0,1)^2$ with $\Omega^f := (0,1) \times (0.5, 1)$
and $\Omega^b := (0,1) \times (0, 0.5)$. The boundaries of the domain
are defined as:
\begin{align*}
  \Gamma_D^f &:= \cbr[0]{x \in \Gamma^f\ :\ x_1=0 \text{ or } x_2 = 1},
  &
    \Gamma_N^f &:= \cbr[0]{x \in \Gamma^f\ :\ x_1 = 1},
  \\
  \Gamma_P^b = \Gamma_D^b &:= \cbr[0]{x \in \Gamma^b\ : \ x_0=0 \text{ or } x_2 = 0},
  &
    \Gamma_F^b = \Gamma_N^b &:= \cbr[0]{ x \in \Gamma^b\ :\ x_1 = 1}.
\end{align*}
We consider the Navier--Stokes/Biot problem
\cref{eq:navierstokes,eq:biot} with boundary conditions
\begin{align*}
  u^f &= U^f & & \text{on } \Gamma^{f}_D \times J,
  &
  u^b &= U^b & & \text{on } \Gamma^{b}_D \times J,
  &
  p^p &= P^p & & \text{on } \Gamma^b_P \times J,
  \\
  \sigma^f n &= S^{f} & & \text{on } \Gamma^f_{N} \times J,
  &
  \sigma^b n &= S^{b} & & \text{on } \Gamma^b_{N} \times J,
  &    
  z \cdot n &= Z^d & & \text{on } \Gamma^b_F \times J,
\end{align*}
and interface conditions
\begin{align*}
  u^f \cdot n &= (\partial_tu^b+z) \cdot n + M^u & & \text{on } \Gamma_I \times J,
  \\
  \sigma^fn &= \sigma^bn + M^s & & \text{on } \Gamma_I \times J,
  \\
  (\sigma^f n)\cdot n &= p^p + M^p & & \text{on } \Gamma_I \times J,
  \\    
  -2\mu^f\del[0]{\varepsilon(u^f)n}^t &= \gamma\mu^f \kappa^{-1/2} (u^f - \partial_tu^b)^t + M^e
                            & & \text{on } \Gamma_I \times J.
\end{align*}
The functions $M^u$, $M^s$, $M^p$, and $M^e$ in the aforementioned
modified interface conditions, as well as the boundary data, $U^f$,
$U^b$, $P^p$, $S^f$, $S^b$, and $Z^d$, body forces $f^f$ and $f^b$,
source/sink term $g^b$, and the initial conditions are chosen such
that the exact solution is given by
\begin{align*}
  u^f &=\begin{bmatrix}
          \pi x_1\cos(\pi (x_1 x_2-t)) + 1 \\
          -\pi x_2 \cos(\pi (x_1 x_2-t)) + 2x_1
        \end{bmatrix},
      &
  u^b &=\begin{bmatrix}
          \sin(10\pi t)\cos(4(x_1-t))\cos(3x_2) \\
          \sin(10\pi t)\sin(5x_1)\cos(2(x_2-t))
        \end{bmatrix},
  \\
  p^f &= \sin(3x_1)\cos(4(x_2-t)), 
      &
  p^p &= \sin(3(x_1x_2-t)).
\end{align*}
The model parameters are chosen as follows: $\mu^f = 10^{-2}$,
$\mu^b = 10^{-3}$, $\alpha=0.2$, $\lambda = 10^2$, $\kappa = 10^{-2}$,
$c_0 = 10^{-2}$, and $\gamma = 0.3$, while the HDG penalty parameters
are chosen as $\beta^f = \beta^b = 8k^2$, where $k$ is the polynomial
degree. We consider the time interval $J=[0, 0.01]$ and implement the
HDG method in the Netgen/NGSolve finite element library
\cite{Schoberl:1997,Schoberl:2014}.

We first consider the spatial rates of convergence. For this, we
compute the solution using $k=1$ and $k=2$ and list the errors
measured in the $L^2$-norm and rates of convergence of the unknowns in
$\Omega^f$ in \cref{tab:ratesconv-omegaf} and in $\Omega^b$ in
\cref{tab:ratesconv-omegaf}. We use a time step of
$\Delta t = \tfrac{1}{10} h^{k+2}$. From both tables we observe that
$\norm[0]{u_h^f - u^f}_{\Omega^f}$,
$\norm[0]{u_h^b - u^b}_{\Omega^b}$, and $\norm[0]{z_h - z}_{\Omega^b}$
are $\mathcal{O}(h^{k+1})$ while $\norm[0]{p_h^f - p^f}_{\Omega^f}$,
$\norm[0]{p_h^b - p^b}_{\Omega^b}$, and
$\norm[0]{p_h^p - p^p}_{\Omega^b}$ are $\mathcal{O}(h^{k})$.

We next consider the temporal rates of convergence. The errors,
measured in the $L^2$-norm, and temporal rates of convergence for the
unknowns in $\Omega^f$ are given in \cref{tab:ratesconvdt-omegaf} and
in $\Omega^b$ are given in \cref{tab:ratesconvdt-omegab}. To compute
these results we choose $k=4$ and compute on the solution on a mesh
consisting of 37548 cells. We observe that the error for all unknowns
is $\mathcal{O}(\Delta t)$.

\begin{table}
  \small
  \centering {
    \begin{tabular}{cccccc}
      \hline
      Cells
      & $\norm[0]{u_h^f - u^f}_{\Omega^f}$ & $r$ & $\norm[0]{p_h^f - p^f}_{\Omega^f}$
      & $r$ & $\norm[0]{\nabla \cdot u_h^f}_{\Omega^f}$ \\
      \hline
      \multicolumn{6}{c}{$k=1$} \\
      8 & 2.4e-01 &   - & 6.0e-01 &   - & 6.3e-16 \\ 
      28 & 5.8e-02 & 2.0 & 1.2e-01 & 2.4 & 5.3e-16 \\ 
      152 & 1.5e-02 & 2.0 & 4.7e-02 & 1.3 & 9.6e-16 \\ 
      576 & 3.2e-03 & 2.2 & 2.2e-02 & 1.1 & 8.5e-16 \\ 
      2348 & 6.8e-04 & 2.2 & 1.1e-02 & 1.1 & 8.4e-16 \\       
      \hline
      \multicolumn{6}{c}{$k=2$} \\
      8 & 5.8e-02 &   - & 1.1e-01 &   - & 1.1e-15 \\ 
      28 & 9.6e-03 & 2.6 & 2.6e-02 & 2.0 & 1.2e-15 \\ 
      152 & 1.5e-03 & 2.7 & 4.7e-03 & 2.5 & 1.5e-15 \\ 
      576 & 1.6e-04 & 3.2 & 9.5e-04 & 2.3 & 1.3e-15 \\ 
      2348 & 1.5e-05 & 3.4 & 2.2e-04 & 2.1 & 1.2e-15 \\       
      \hline
    \end{tabular}
  } \caption{Errors and spatial rates of convergence $r$ for the
    solution in $\Omega^f$ for the test case described in
    \cref{sec:numexample}.}
  \label{tab:ratesconv-omegaf}
\end{table}

\begin{table}
  \small
  \centering {
    \begin{tabular}{ccccccccc}
      \hline
      Cells
      & $\norm[0]{u_h^b - u^b}_{\Omega^b}$ & $r$ & $\norm[0]{p^b_h - p^b}_{\Omega^b}$
      & $r$ & $\norm[0]{z_h-z}_{\Omega^b}$ & $r$ & $\norm[0]{p_h^p - p^p}_{\Omega^b}$ & $r$ \\
      \hline
      \multicolumn{9}{c}{$k=1$} \\
      8 & 6.9e-02 &   - & 3.2e+01 &   - & 1.6e-01 &   - & 1.4e-01 &   -  \\ 
      28 & 1.8e-02 & 2.0 & 1.6e+01 & 0.9 & 4.0e-02 & 2.0 & 6.8e-02 & 1.0 \\ 
      152 & 2.8e-03 & 2.6 & 6.6e+00 & 1.3 & 6.6e-03 & 2.6 & 2.9e-02 & 1.2 \\ 
      576 & 7.2e-04 & 2.0 & 3.4e+00 & 1.0 & 1.9e-03 & 1.8 & 1.5e-02 & 1.0 \\ 
      2348 & 1.7e-04 & 2.1 & 1.7e+00 & 1.0 & 5.2e-04 & 1.9 & 7.3e-03 & 1.0 \\      
      \hline
      \multicolumn{9}{c}{$k=2$} \\
      8 & 9.7e-03 &   - & 6.8e+00 &   - & 4.5e-02 &   - & 2.0e-02 &   -  \\ 
      28 & 2.2e-03 & 2.1 & 2.3e+00 & 1.5 & 4.9e-03 & 3.2 & 5.3e-03 & 1.9 \\ 
      152 & 1.9e-04 & 3.6 & 4.4e-01 & 2.4 & 4.7e-04 & 3.4 & 1.2e-03 & 2.2 \\ 
      576 & 2.2e-05 & 3.1 & 1.0e-01 & 2.1 & 1.0e-04 & 2.2 & 2.8e-04 & 2.1 \\ 
      2348 & 2.3e-06 & 3.2 & 2.6e-02 & 2.0 & 1.7e-05 & 2.7 & 6.6e-05 & 2.1 \\       
      \hline
    \end{tabular}
  }
  \caption{Errors and spatial rates of convergence $r$ for the
    solution in $\Omega^b$ for the test case described in
    \cref{sec:numexample}.}
  \label{tab:ratesconv-omegab}
\end{table}

\begin{table}
  \small
  \centering {
    \begin{tabular}{cccccc}
      \hline
      $\Delta t$
      & $\norm[0]{u_h^f - u^f}_{\Omega^f}$ & $r$ & $\norm[0]{p_h^f - p^f}_{\Omega^f}$
      & $r$ & $\norm[0]{\nabla \cdot u_h^f}_{\Omega^f}$ \\
      \hline
      $T/8$ & 6.0e-02 &   - & 8.2e-02 &   - & 6.1e-12 \\ 
      $T/16$ & 3.6e-02 & 0.7 & 5.3e-02 & 0.6 & 5.9e-12 \\ 
      $T/32$ & 2.1e-02 & 0.8 & 3.2e-02 & 0.7 & 5.7e-12 \\ 
      $T/64$ & 1.1e-02 & 0.9 & 1.8e-02 & 0.8 & 4.3e-12 \\ 
      $T/128$ & 6.0e-03 & 0.9 & 9.8e-03 & 0.9 & 3.5e-12 \\      
      \hline
    \end{tabular}
  } \caption{Errors and temporal rates of convergence $r$ for the
    solution in $\Omega^f$ for the test case described in
    \cref{sec:numexample}.}
  \label{tab:ratesconvdt-omegaf}
\end{table}

\begin{table}
  \small
  \centering {
    \begin{tabular}{ccccccccc}
      \hline
      $\Delta t$
      & $\norm[0]{u_h^b - u^b}_{\Omega^b}$ & $r$ & $\norm[0]{p^b_h - p^b}_{\Omega^b}$
      & $r$ & $\norm[0]{z_h-z}_{\Omega^b}$ & $r$ & $\norm[0]{p_h^p - p^p}_{\Omega^b}$ & $r$ \\
      \hline
      $T/8$ & 5.5e-04 &   - & 9.5e-06 &   - & 1.1e-02 &   - & 1.4e-03 &   - \\ 
      $T/16$ & 3.0e-04 & 0.9 & 4.5e-06 & 1.1 & 5.7e-03 & 1.0 & 7.1e-04 & 1.0 \\ 
      $T/32$ & 1.6e-04 & 0.9 & 2.2e-06 & 1.0 & 2.9e-03 & 1.0 & 3.6e-04 & 1.0 \\ 
      $T/64$ & 8.7e-05 & 0.9 & 1.1e-06 & 1.0 & 1.4e-03 & 1.0 & 1.8e-04 & 1.0 \\ 
      $T/128$ & 4.6e-05 & 0.9 & 5.6e-07 & 1.0 & 7.2e-04 & 1.0 & 8.9e-05 & 1.0 \\ 
      \hline
    \end{tabular}
  }
  \caption{Errors and temporal rates of convergence $r$ for the
    solution in $\Omega^b$ for the test case described in
    \cref{sec:numexample}.}
  \label{tab:ratesconvdt-omegab}
\end{table}

\section{Conclusions}
\label{sec:conclusions}

In this paper we introduced and analyzed an HDG discretization for the
time-dependent Navier--Stokes equations coupled to the Biot
equations. Appealing properties of the discretization include that the
velocities and displacement are divergence-conforming, that the
compressibility equations are satisfied pointwise on the elements, and
that mass is conserved pointwise for the semi-discrete problem when
source and sink terms are ignored. We proved stability and
well-posedness under a small data assumption and presented an a priori
error analysis of the method.

\subsubsection*{Acknowledgements}

Aycil Cesmelioglu and Jeonghun J. Lee acknowledge support from the
National Science Foundation through grant numbers DMS-2110782 and
DMS-2110781. Sander Rhebergen acknowledges support from the Natural
Sciences and Engineering Research Council of Canada through the
Discovery Grant program (RGPIN-2023-03237).

\bibliographystyle{abbrvnat}
\bibliography{references}
\end{document}